\documentclass[12pt,leqno]{amsart}

\usepackage{amsmath,amscd,amsthm,amsxtra,amssymb}
\usepackage{epsfig,graphics,color,colortbl}
\usepackage{amssymb,latexsym}
\usepackage{mathrsfs}
\usepackage[poly,all]{xy}
\usepackage{marginnote}
\usepackage{xspace}
\usepackage[colorlinks=true, pdfstartview=FitV, linkcolor=blue,citecolor=blue,urlcolor=blue]{hyperref}

\usepackage{yfonts}
\usepackage{enumerate}

\usepackage[usenames,dvipsnames,svgnames,table]{xcolor}

\usepackage[normalem]{ulem}  

\usepackage[colorinlistoftodos]{todonotes}

\allowdisplaybreaks[3]

\newlength{\mylength}

\theoremstyle{plain}
\newtheorem{thm}{\bf Theorem}[section]
\newtheorem{df}[thm]{\bf Definition}
\newtheorem{prop}[thm]{\bf Proposition}
\newtheorem{cor}[thm]{\bf Corollary}
\newtheorem{lem}[thm]{\bf Lemma}
\newtheorem{conj}[thm]{\bf Conjecture}

\theoremstyle{definition}
\newtheorem{ex}[thm]{\bf Example}
\newtheorem{rem}[thm]{\bf Remark}

\newcommand{\nc}{\newcommand}
\nc{\Prop}{\begin{prop}}
\nc{\enprop}{\end{prop}}
\nc{\Lemma}{\begin{lem}}
\nc{\enlemma}{\end{lem}}
\nc{\Exam}{\begin{ex}}
\nc{\enexam}{\end{ex}}
\nc{\Th}{\begin{thm}}
\nc{\enth}{\end{thm}}

\newenvironment{red}
{\relax\color{red}}
{\hspace*{.5ex}\relax}

\newcommand{\ber}{\begin{red}}
\newcommand{\er}{\end{red}}

\newcommand{\berE}{\begin{red} {}\marginnote{\fbox{\scshape\lowercase{E}}} }
\newcommand{\berm}{\begin{red} {}\marginnote{\fbox{\scshape\lowercase{M}}} }

\newcommand{\C}{{\mathbb C}}
\newcommand{\Q}{\mathbb {Q}}
\newcommand{\Z}{{\mathbb Z}}
\newcommand{\R}{{\mathbb R}}

\newcommand{\one}{{\bf{1}}}
\newcommand{\seteq}{\mathbin{:=}}

\newcommand{\To}[1][{\hs{0.8ex}}]{\xrightarrow{\ms{7mu}{#1}\ms{7mu}}}

\newcommand{\Sp}{\mathrm{span}_{\R_{\ge0}}}  	
\newcommand{\SpR}{\mathrm{span}_{\R}}  	

\newcommand{\g}{\mathfrak{g}}
\newcommand{\n}{\mathfrak{n}}
\newcommand{\Uq}[1][{\mathfrak{g}}]{{U_q(#1)}}

\newcommand{\Hom}{\operatorname{Hom}}
\newcommand{\HOM}{\mathrm{H{\scriptstyle OM}}}
\newcommand{\END}{\mathrm{E\scriptstyle ND}}
\newcommand{\End}{\operatorname{End}}

\newcommand{\isoto}[1][]{\mathop{\xrightarrow%
[{\raisebox{.3ex}[0ex][.3ex]{$\scriptstyle{#1}$}}]%
{{\raisebox{-.6ex}[0ex][-.6ex]{$\mspace{2mu}\sim\mspace{2mu}$}}}}}

\newcommand{\Mod}{\text{-}\mathrm{Mod}}
\newcommand{\gmod}{\text{-}\mathrm{gmod}}

\newcommand{\proj}{\text{-}\mathrm{proj}}

\newcommand{\conv}{{\mathbin{\scalebox{1.1}{$\mspace{1.5mu}\circ\mspace{1.5mu}$}}}}
\newcommand{\hconv}{\mathbin{\scalebox{.9}{$\nabla$}}}
\newcommand{\sconv}{\mathbin{\scalebox{.9}{$\Delta$}}}

\renewcommand{\Im}{\on{Im}}
\newcommand{\de}{\on{\textfrak{d}}}

\newcommand{\K}{{J}}

\newcommand{\cmA}{\mathsf{A}}  
\newcommand{\wlP}{\mathsf{P}}   
\newcommand{\rlQ}{\mathsf{Q}}   
\newcommand{\weyl}{\mathsf{W}}  
\newcommand{\prD}{\Delta_+}            
\newcommand{\nrD}{\Delta_-}            
\newcommand{\sg}{\mathfrak{S}}   
\newcommand{\Po}{\wlP}

\newcommand{\qQ}{\mathcal{Q}}
\newcommand{\bQ}{\overline{\qQ}}

\newcommand\Aq[1][{\g^+}]{A_q(#1)}
\newcommand\Aqq[1][{\g^+}]{A_{q}(#1)_{\Z[q,q^{-1}]}}

\newcommand{\Zq}{{\Z[q,q^{-1}]}}  		

\newcommand{\wt}{\mathrm{wt}} 		
\newcommand{\bR}{\mathbf{k}} 		
\nc{\corp}{\bR}
\newcommand{\catC}{ \mathscr{C}}  	
\newcommand{\tcatC}{ \widetilde{\mathscr{C}}}  	
\newcommand{\catT}{ \mathcal{T}}  	
\newcommand{\lT}{ \widetilde{\mathcal{T}}}  	

\newcommand{\catTc}{ \mathcal{T}_{\mathrm{br}}}  	
\newcommand{\lRg}[1][w]{ \tcatC_{#1} }  	
\newcommand{\dM}{ \mathsf{M }}              
\newcommand{\dC}{ \mathsf{C }}              
\newcommand{\dS}{ \mathsf{S }}              
\newcommand{\mD}{ \mathsf{D}}              
\newcommand{\gW}{\mathsf{W}}
\newcommand{\sgW}{\mathsf{W}^*}
\newcommand{\tF}{\widetilde{F}}  		
\newcommand{\tE}{\widetilde{E}}  		
\newcommand{\ep}{\varepsilon}  		
\newcommand{\ph}{\varphi}  		
\newcommand{\trivialM}{\mathbf{1}} 	

\newcommand{\eC}{\mathcal{C}} 		

\newcommand{\Ht}{\mathrm{ht}} 		
\newcommand{\cd}{\mathfrak{d}} 		
\newcommand{\Rr}{\mathbf{r}} 			
\newcommand{\nR}{\mathrm{R}^{\mathrm{norm}}} 		
\newcommand{\RR}{\mathrm{R}} 				
\newcommand{\coR}{R} 				
\newcommand{\La}{\Lambda} 			
\newcommand{\tLa}{\widetilde{\Lambda}} 			
\newcommand{\Dd}{\text{ \textfrak{d}}} 			

\newcommand{\Ma}{{\ms{1.5mu}\widehat{\mathsf{M}}}}

\newcommand{\Na}{{\widehat{\mathsf{N}}}}

\newcommand{\z}{{z_\Ma}}

\newcommand{\gr}{\mathrm{gr}}
\newcommand{\triv}{{\mathbf{1}}}   				
\newcommand{\id}{\ms{1mu}{\mathsf{id}}\ms{1mu}}   				
\newcommand{\Ds}{{\mathcal{D}}}   				
\newcommand{\Hm}{{\mathrm{H}}}   				
\newcommand{\gHm}{\mathrm{H}^{\gr}}   				
\newcommand{\dphi}{{\phi}}   				
\newcommand{\gH}{\mathrm{H}}   				
\newcommand{\gL}{\mathrm{L}}   				
\newcommand{\gzeta}{\zeta^\gr}   				
\newcommand{\gPsi}{\Psi^\gr}   				
\newcommand{\gtPsi}{\widetilde{\Psi}^\gr}   				
\newcommand{\gT}{T^\gr}   				
\newcommand{\gtT}{\widetilde{T}^\gr}   				

\newcommand{\lc}{\chi}   				
\newcommand{\lep}{\mathfrak{e}}   				
\newcommand{\aW}{ \mathrm{W}}              
\newcommand{\lpq}{\mathsf{q}}   				

\newcommand{\rD}{{\mathscr{D}}}    				
\newcommand{\lD}{{\mathscr{D}^{-1}}}    				

\newcommand{\lG}{\Gamma}   					

\nc{\be}{\begin{enumerate}}
\newcommand{\bnum}{\be[{\rm(i)}]}
\newcommand{\bna}{\be[{\rm(a)}]}

\newcommand{\rtl}{\rlQ}

\newcommand{\etens}{\boxtimes}

\newcommand{\rmat}[1]{\ms{1mu}{\mathbf{r}}_%
{\mspace{-2mu}\raisebox{-.6ex}{${\scriptstyle{#1}}$}}}

\newcommand{\shc}{\mathcal{C}}

\newcommand{\tC}{\widetilde{C}}

\newcommand{\Ob}{\on{Ob}}

\nc{\ms}{\mspace}
\nc{\cl}{\colon}
\nc{\ro}{{\rm (}}
\nc{\rf}{{\rm )}}
\nc{\noi}{\noindent}
\nc{\bl}{\bigl(}
\nc{\br}{\bigr)}

\newenvironment{myequation}
{\relax\setlength{\arraycolsep}{1pt}\begin{eqnarray}}
{\end{eqnarray}}
\newenvironment{myequationn}
{\relax\setlength{\arraycolsep}{1pt}\begin{eqnarray*}}
{\end{eqnarray*}}

\nc{\eq}{\begin{myequation}}
\nc{\eneq}{\end{myequation}}
\nc{\eqn}{\begin{myequationn}}
\nc{\eneqn}{\end{myequationn}}

\newenvironment{myarray}[1]{\relax\setlength{\arraycolsep}{1pt}
\begin{array}{#1}}{\end{array}\relax}

\newcommand{\ba}{\begin{myarray}}
\newcommand{\ea}{\end{myarray}}

\nc{\hs}{\hspace*}
\nc{\set}[2]{\left\{{#1}\mid{#2}\right\}}
\nc{\snoi}{\smallskip\noi}
\nc{\al}{\alpha}
\nc{\rmz}{\setminus\{0\}}
\nc{\tens}{\otimes}
\nc{\vphi}{\varphi}
\nc{\ee}{\end{enumerate}}
\nc{\la}{\lambda}
\nc{\bc}{\begin{cases}}
\nc{\ec}{\end{cases}}
\nc{\qtq}[1][and]{\quad\text{#1}\quad}
\nc{\qt}[1]{\quad\text{#1}}
\nc{\dual}{{\displaystyle{\ms{1mu}\star}}}
\nc{\wle}{\preceq}
\nc{\epito}{\twoheadrightarrow}
\nc{\epiTo}[1][]{\xymatrix@C=4ex{{}\ar@{->>}[r]^-{#1}&{}}}
\nc{\Proof}{\begin{proof}}
\nc{\lan}{\langle}
\nc{\ran}{\rangle}
\nc{\ang}[1]{\lan{#1}\ran}
\nc{\QED}{\end{proof}}
\nc{\soplus}{\scalebox{.65}{\raisebox{.2ex}{$\displaystyle\bigoplus$}}}
\nc{\eps}{\varepsilon}
\nc{\on}{\operatorname}
\nc{\supp}{\on{supp}}
\nc{\sct}{strongly commute\xspace}
\nc{\scts}{strongly commutes\xspace}
\nc{\bce}{\eta}			
\nc{\height}[1]{\Ht({#1})}
\nc{\braid}{{\ms{1mu}\mathrm{br}}}
\nc{\gp}{\mathfrak{p}}
\nc{\wtl}{\wlP}
\nc{\ra}{real and admits an affinization}
\nc{\ras}{real and admit affinizations}
\nc{\Cor}{\begin{cor}}
\nc{\encor}{\end{cor}}
\nc{\shf}{\mathcal{F}}
\nc{\Cw}{\catC_w}
\nc{\tCw}{\lRg}
\nc{\akew}[1][2ex]{\rule[-1ex]{#1}{0ex}}
\nc{\ake}[1][2ex]{\rule[-1ex]{0ex}{#1}}
\nc{\tRm}{(R\gmod)\widetilde{\mbox{$\ake[2.5ex]\akew[.9ex]$}}}
\nc{\monoTo}[1][]{\xymatrix{\ar@{>->}[r]^-{{#1}}&}}
\nc{\monoto}[1][]{\rightarrowtail}
\nc{\tX}{\widetilde{X}}
\nc{\corps}{\corp}
\nc{\tL}{\widetilde{L}}
\nc{\prtl}{\rtl_+}
\nc{\tK}{\widetilde{K}}
\nc{\tep}{\widetilde\ep}
\nc{\teps}{\widetilde\ep}
\nc{\teta}{\widetilde\eta}
\nc{\ga}{\mathfrak{a}}
\nc{\scbul}{{\,\raise1pt\hbox{$\scriptscriptstyle\bullet$}\,}}
\nc{\bwr}{\mbox{\large$\wr$}}
\nc{\tR}{\widetilde{\mathscr{R}}}
\nc{\lS}{\mathsf{S}}
\nc{\lZ}{\mathcal{Z}}
\nc{\prolim}{\mathop{\varprojlim}\limits}
\nc{\sym}{\sg}
\newcounter{myc}
\nc{\txi}{\tilde{\xi}}
\nc{\rl}{\rlQ}

\numberwithin{equation}{section}

\title[Localizations for quiver Hecke algebras]
{Localizations for quiver Hecke algebras}

\author[Masaki Kashiwara]{Masaki Kashiwara}
\thanks{The research of M.\ Ka.\
was supported by Grant-in-Aid for Scientific Research (B)
15H03608, Japan Society for the Promotion of Science.}
\address[Masaki Kashiwara]{
Kyoto University Institute for Advanced Study, Research Institute for
Mathematical Sciences, Kyoto University, Kyoto 606-8502, Japan \& 
Korea Institute for Advanced Study, Seoul 02455, Korea
}
\email[Masaki Kashiwara]{masaki@kurims.kyoto-u.ac.jp}

\author[Myungho Kim]{Myungho Kim}
\address[Myungho Kim]{Department of Mathematics, Kyung Hee University, Seoul 02447, Korea}
\email[Myungho Kim]{mkim@khu.ac.kr}
\thanks{The research of M.\ Kim was supported by the National Research Foundation of
Korea(NRF) Grant funded by the Korea government(MSIP) (NRF-2017R1C1B2007824).}

\author[Se-jin Oh]{Se-jin Oh}
\thanks{ The research of S.-j.\ Oh was supported by 
the Ministry of Education of the Republic of Korea and the
National Research Foundation of Korea (NRF-2019R1A2C4069647).}
\address[Se-jin Oh]{Department of Mathematics, Ewha Womans University, Seoul 120-750, Korea}
\email[Se-jin Oh]{sejin092@gmail.com}

\author[Euiyong Park]{Euiyong Park}
\thanks{The research of E.\ P.\ was supported by the National Research Foundation of Korea(NRF) Grant funded by the Korea Government(MSIP)(NRF-2017R1A1A1A05001058).}
\address[Euiyong Park]{Department of Mathematics, University of Seoul, Seoul 02504, Korea}
\email[Euiyong Park]{epark@uos.ac.kr}



\date{December 28, 2020}

\begin{document}

\maketitle

\begin{center}
{\it In Memory of Professor Bertram Kostant}
\end{center}

\begin{abstract}
In this paper, we provide a generalization of the localization procedure 
 for monoidal categories developed  in \cite{KKK18} by Kang-Kashiwara-Kim 
by introducing the notions of braiders and a real commuting family of braiders. 
Let $R$ be a quiver Hecke algebra of arbitrary symmetrizable type and $R\gmod$ the category of finite-dimensional graded $R$-modules.  
For an element $w$ of the Weyl group, $\catC_w$ is the subcategory of $R\gmod$ which categorifies the quantum unipotent coordinate algebra $\Aq[\n(w)]$.
We construct the localization $\lRg$ of $\catC_w$ 
by adding the inverses of
simple modules $\dM(w\La_i, \La_i)$
which correspond to the frozen variables in the quantum cluster algebra 
$\Aq[\n(w)]$. 
The localization $\lRg$ is left rigid and 
it is conjectured that $\lRg$ is rigid. 
\end{abstract}

\tableofcontents

\section*{Introduction}

 Khovanov-Lauda and Rouquier independently 
introduced a family  of $\Z$-graded associative algebras $R(\beta)$ over a commutative ring $\bR$, called the \emph{quiver Hecke algebras},
in order to categorify the half of a quantum group $U_q(\g)$ (\cite{KL09, R08}).
For an $R(\beta)$-module  $M$ and an $R(\gamma)$-module $N$, one can
  form an $R(\beta+\gamma)$-module $M\conv N$  from the $R(\beta)\tens R(\gamma)$-module $M\tens N$, called the \emph{convolution product},
which yields monoidal category structures on the direct sum $R\proj$ of the categories of finitely generated projective graded $R(\beta)$-modules and the direct sum $R\gmod$ of the categories of finite dimensional graded $R(\beta)$-modules.
It turns out that the Grothendieck rings $K(R\proj)$ and $K(R\gmod)$ are isomorphic to the half of quantum group $U_q^-(\g)$ and to its dual $U_q^-(\g)^*$, respectively.
Moreover, when $\g$ is symmetric and $\bR$ is a field of characteristic zero,  the set of classes of indecomposable projective graded modules and the set of classes of simple graded modules correspond to the lower and upper global basis, respectively (\cite{R11, VV09}).

The monoidal category $R\gmod$ has an interesting connection with quantum affine algebras via \emph{KLR-type quantum Schur-Weyl duality.} 
In  \cite{KKK18}, a family of functors  from $R\gmod$ to the categories of finite-dimensional modules over  quantum affine algebras was introduced. 
Briefly speaking, for a given family of simple modules over the quantum affine algebra that we are interested in, we have a procedure to obtain a quiver $\Gamma$ from the family 
 and a functor $\mathcal F$ from $R^\Gamma\gmod$ to the category of modules over the quantum affine algebra, where $R^\Gamma$ is the quiver Hecke algebra corresponding to the symmetric Kac-Moody algebra associated with the underlying graph of the  quiver $\Gamma$.
Among those functors, the case when  the quantum affine algebra is of type $A^{(1)}_{N-1}$ and the family is given as all the successive $q^{\pm 2}$-shifts of the first fundamental representation $V(\varpi_1)$,  was studied  in detail. 
In this case, the quiver Hecke algebra is of type $A_\infty$ which is denoted by $R^{A_{\infty}}$.   
Set $\mathcal A:=R^{A_{\infty}}\gmod$.  
Then the functor  $\mathcal F$  factors through the quotient category $\mathcal A/ \mathcal S_N$, where $\mathcal S_N$ denotes the kernel of the functor $\mathcal F$.
But still there are rather big differences between the category  $\mathcal A/ \mathcal S_N$ and  
the full subcategory consisting of images of $\mathcal F$, denoted by $\mathcal C^0_{A^{(1)}_{N-1}}$
: there is an infinite family of simples in $\mathcal A/ \mathcal S_N$  which are isomorphic to the unit object  (trivial representation) of $\mathcal C^0_{A^{(1)}_{N-1}}$ under $\mathcal F$, and 
the latter  is a rigid monoidal category, i.e., every object has left and right duals, whereas  the former  is not.
In an effort to obtain a more similar category to $\mathcal C^0_{A^{(1)}_{N-1}}$,
a localization of the category $\mathcal A/ S_N$ by a \emph{commuting family of central objects} was introduced in \cite{KKK18}. 
More precisely, a monoidal category $\mathcal T_N$ and a functor $\Phi :  \mathcal A/\mathcal S_N \to \mathcal T_N$ were constructed such that the functor $\mathcal F$ factors through $\Phi$ and
the simple objects in $\mathcal A/ \mathcal S_N$  which correspond to the trivial representation under $\mathcal F$ become isomorphic to the unit object in $\mathcal T_N$ under $\Phi$. 
Moreover it turned out that the category $\mathcal T_N$ is a rigid monoidal category.
One of key observations was that the family considered above forms a commuting family of central objects and this condition is enough to construct a larger category on which all the objects in the family can be inverted. 
Note that the construction in \cite[Appendix A]{KKK18} was done in an abstract setting: i.e., for an  arbitrary  monoidal category $\mathcal T$ with a commuting family of central objects, one obtains a localization of $\mathcal T$ with  desired properties.

Since the application of the localization procedure to the category $\mathcal A /\mathcal S_N$ was quite successful, it is natural to pursue further along this line  with other monoidal categories coming from $R\gmod$. 
For example, for each element $w$ in the Weyl group of $\g$, there is an interesting monoidal subcategory $\catC_w$ of $R\gmod$ whose Grothendieck ring is isomorphic to the \emph{ quantum unipotent coordinate ring} $\Aq[\n(w)]$.  Moreover the category $\catC_w$ reflects the quantum cluster algebra structure on $\Aq[\n(w)]$ in an essential way: i.e., all the cluster monomials belong to the set of  classes of simple modules  in $\catC_w$ (\cite{KKKO18}).
Note that in the quantum cluster algebra $\Aq[\n(w)]$, the frozen cluster variables are not invertible, whereas there are several works considering the version of unipotent quantum coordinate ring with invertible frozen variables (for example, see \cite{KO17, Qin17}).
Hence it is desirable to have a localization of the category $\catC_w$ on which the modules corresponding to the frozen variables are invertible. 
But the procedure in \cite[Appendix A]{KKK18} is not directly applicable to this case 
because the simple modules of $\catC_w$ corresponding to the frozen variables may not form a commuting family of central objects. 

\medskip
The first goal of this paper is to provide a generalization of the localization procedure developed  in \cite{KKK18}, which is applicable to more general cases.
We introduce a \emph{braider} and a \emph{real commuting family of braiders} which generalize a central object and a commuting family of central objects in a broad sense.
Let $\catT$ be a $\bR$-linear monoidal category with the tensor product $\tens$ and a unit object $\one$.
A pair $(C, R_C)$ of an object $C$ and a natural transformation $R_C: (C \tens -) \rightarrow (-\tens C)$  in $\catT$  is called a \emph{ (left) braider} of $\catT$, when $R_C$ is compatible with the tensor product and the unit object (see \eqref{Eq: central obj}).  Note that we do not require that $R_C(X)$ are isomorphisms which means that $C$ may not be central. 
We call a family $\{(C_i,R_{C_i})\}_{i\in I}$ of braiders in $\catT$ a \emph{real commuting family} if  $R_{C_i}(C_i) \in \bR^\times \id_{C_i\tens C_i}$ and $R_{C_j}(C_i) \circ R_{C_i}(C_j) \in \bR^\times  \id_{C_i \tens C_j}$,
which is a key difference from a commuting family of central objects.
In the case of a commuting family of central objects, the morphisms $R_{C_i}(C_i)$ and $R_{C_j}(C_i) \circ R_{C_i}(C_j)$ should be the identities. This difference seems to be small but important and essential (see Example \ref{Ex: constants}).
One of our main theorems is the existence and the uniqueness of localization of the monoidal category $\catT$ via the real commuting family $\{(C_i,R_{C_i})\}_{i\in I}$ of braiders (Theorem \ref{Thm: localization}):  
there exists a monoidal category $\lT$ and a monoidal functor $\Upsilon: \catT \rightarrow \lT$ 
such that the object $\Upsilon(C_i)$ is invertible  in $\lT$   and 
the morphisms $\Upsilon(R_{C_i}(X))\cl
\Upsilon(C_i\tens X)\to\Upsilon(X\tens C_i)$ are isomorphisms in $\lT$ for any $i\in I$ and $X\in\catT$. Moreover, the pair $(\lT,\Upsilon)$ is universal (an initial object) among those satisfying the above two conditions.
A  part of this paper is devoted to define  the compositions and the tensor products of morphisms in $\lT$   
so that $\lT$ becomes a monoidal category.
Since our main motivation is to construct  localizations of monoidal categories arising from $R\gmod$, 
we also developed the localization for \emph{graded cases}: assume that $\catT$ is a $\bR$-linear monoidal category with a decomposition
$ \catT = \bigoplus_{\lambda \in \Lambda} \catT_\lambda $ for some abelian group $\Lambda$ such that $\triv \in \catT_0$ and $\otimes$ induces a bifunctor $\catT_{\lambda} \times \catT_{\mu} \rightarrow \catT_{\lambda+\mu}$ for any $\lambda, \mu \in \Lambda$. We further assume that there is an invertible central object $q$ in $\catT_0$. 
Then we introduce a notion of \emph{graded braiders} $(C,R_C,\phi)$, which roughly means that $(C,R_C)$ is a braider with 
an abelian group homomorphism $\phi:\Lambda\to \Z$ such that $R_C(X): C\tens X \rightarrow q^{\phi(\la)} \tens X \tens C$  for $X \in \catT_\la$. 
These are modelled on 
simple modules in $R\gmod$
and homogeneous homomorphisms between convolution products. 
Then by a careful study on the degrees of morphisms, we obtain a localization of $\catT$ with analogous properties as ungraded cases.

Our second goal is to apply the localization procedure we have developed to subcategories of $R\gmod$. Assume that $R$ is a quiver Hecke algebra of \emph{arbitrary symmetrizable} type. 
A distinguished feature of the monoidal category $R\gmod$ is that every simple object  $M$ admits a  structure of  a  \emph{non-degenerate} graded braider $(M,R_M,\phi)$, which is unique in some sense (Lemma \ref{lem: c braider}). 
Thus one can consider various localizations of subcategories of $R\gmod$.
We are more interested in the  cases that  
 the real commuting family of graded braiders in a subcategory $\mathscr C$  of $R\gmod$  is originated from a family of  real  simple modules admitting \emph{affinizations}, because such a case enjoys more pleasant properties out of \emph{R-matrices}. 
For a quiver Hecke algebra of arbitrary type, the notion of affinization was developed in \cite{KP18} in order to define the \emph{R-matrices}, 
which is the distinguished nonzero homomorphism $\rmat{M,N}: q^{\La(M,N)} M\conv N \to N \conv M$ for simple $R$-modules $M$ and $N$ such that one of them admits an affinization.
Let $M$ be a real simple $R$-module with an affinization. Then we investigate the non-degenerate graded braider structure $(M,R_M,\phi)$, which  is closely related to the R-matrices. 
 Proposition \ref{prop:effB} tells that the morphism $R_M(N)$  for a simple module $N$ is either zero or equal to $\rmat{M,N}$ up to a constant multiple, and 
 Lemma \ref{lem:realbrai} gives a criterion to check when given non-degenerate braiders $ (S_a, R_{S_a}, \phi) $ form a real commuting family of graded braiders.   
It is shown in Proposition \ref{prop:str} that if $(C_a,R_{C_a},\phi_a)_{a\in A}$ is a real commuting family of non-degenerate graded braiders in $\catC$ such that every $C_a$ admits an affinization, then the localization functor $\Phi :\catC \to \tcatC$ sends a simple $S$ to a simple or zero.

Thus 
 it is important  to determine whether a given simple module admits an affinizations or not.
 Note that if the quiver Hecke algebra $R$ is of symmetric type, then there is a natural affinization for any module in $R\gmod$ (\cite{KKK18}). But in general, finding affinizations of a given $R$-module is a nontrivial problem.
 In this paper, we prove that a special family of simple modules, called the \emph{determinantial modules}, admit affinizations (Theorem \ref{Thm: aff for dM}). 
The determinantial module $\dM(\la,\mu)$ associated with a pair of weights $\la, \mu$ in the Weyl group orbit of a dominant integral weight of $\g$  is a  simple $R$-module
  corresponding to  the \emph{unipotent quantum minors} $D(\la,\mu)$, which is a certain element in the upper global basis (\cite{KKOP18}). 
To prove existence of affinizations of determinantial modules, we introduce and investigate the polynomials $\chi_i({M})$, $\lep_i({M})$ and $\lep^*_i({M})$ associated with an $R$-module ${M}$ (Definition \ref{Def: lc lep}). 
Note that $\lep_i(\widehat{L})$ and $\lep^*_i(\widehat{L})$ 
can be viewed as lifts of $\ep_i(L)$ and $\ep_i^*(L)$ when $\widehat{L}$ is an affinization of a simple module $L$.
Then we consider the determinantial module $\dM(\la,\mu; a_\Lambda)$ using the cyclotomic quiver Hecke algebra $R_A^{a_\Lambda}$ over the polynomial ring $A = \bR[z]$, and prove that 
$\dM(\la,\mu; a_\Lambda)$ is an affinization of $\dM(\la,\mu)$ by studying the inductive formula for $\lep_i( \dM(\la,\La; a_\Lambda))$ given in Proposition \ref{prop: lep for dM}.
Recall that the ring $\Aq[\mathfrak n(w)]$ has a set of  generators consisting of some unipotent quantum minors.
 In particular, the determinantial modules
$\{\dC_i:=\dM(w \Lambda_i, \Lambda_i) \}_{i \in I}$ belong to the category $\catC_w$ and they correspond to the frozen variables of the quantum cluster algebra $\Aq[\mathfrak n(w)]$.
Proposition \ref{Prop: canonical braiders} tells that $\{(\dC_i, R_{\dC_i}, \phi_i) \}_{i \in I}$ is a real commuting family of non-degenerate graded braiders in the category $\catC_w$,
which actually are central in $\catC_w$ (Theorem \ref{Thm: R Ci iso}).
Thus there exists a localization  $\tcatC_w$ of $\catC_w$ by the family $\{(\dC_i, R_{\dC_i}, \phi_i) \}_{i \in I}$ 
 with the canonical functor $\Phi\cl \catC_w \to \lRg$, which was one of our primary goals. 
It follows that the Grothendieck ring $K(\lRg)$ is a localization of the quantum cluster algebra $\Aq[\n(w)]$ at the set of frozen variables.

In the remaining part of the paper, we  investigate rigidity for the localization $\lRg$.
Surprisingly the localization $\lRg$ is left rigid, i.e., every object of $\lRg$ has a left dual, and every simple objects of $\lRg$ has a right dual. 
We conjecture that $\lRg$ is rigid, i.e., every object of $\lRg$ has left and right dual (Conjecture \ref{Conj: rigid tCw}).
We first show that every simple object in $\lRg$ has a right dual (Theorem \ref{thm:rightdual}) by using the  \emph{$T$-system}  appeared in \cite[Proposition 3.2]{GLS13} crucially.
Then, we use the localization $\widetilde{\mathscr{R}}$ of the whole category $R\gmod$ via the same family $\{(\dC_i, R_{\dC_i}, \phi_i) \}_{i \in I}$.
 Note that the generalization from the central objects to the braiders is essential  for the localization $\widetilde{\mathscr{R}}$,  since the braiders $(\dC_i, R_{\dC_i})$ are not central in $\widetilde{\mathscr{R}}$ anymore.
There are some advantage of working on this auxiliary category:   every finite-dimensional quotient of the affinization of $L(i)$  has a left dual in the category $\widetilde{\mathscr{R}}$ (Theorem \ref{th:gendual}). 
Since any nonzero module in $R\gmod$ is a cokernel of a homomorphism between some direct sums of convolution products of such quotients, it follows that every object in $\widetilde{\mathscr{R}}$ has a left dual.
Moreover, those left duals belong to the category $\lRg$ (Theorem \ref{th: left dual}). 
We next  prove that the functor $\lRg \to \widetilde{\mathscr{R}}$ induced from the embedding $\catC_w \to R\gmod$ is actually an equivalence of categories (Theorem \ref{Thm: eqiv R C}), so that the left rigidity of $\lRg$ follows. 
When $\g$ is of finite type and $w$ is the longest element $w_0$, 
we prove that Conjecture \ref{Conj: rigid tCw} is true, i.e., $\lRg[w_0]$ is rigid  (Theorem \ref{th:fin}).
 Moreover, the functor $\lD$ given by taking the left dual is 6-periodic on the set of simple objects  in this case (Proposition \ref{prop:periodic}).

Finally we remark that the rigidity of the localization $\lRg$
is closely related with the \emph{twist automorphism} on the coordinate ring  $\C[N^w]$ 
of the unipotent cells
and its quantum analogues (see \cite{BFZ96,  BD15, BZ97, GY16, GLS12, KO17}).
The twist automorphism was introduced in order to give formulas for the inverses  of parameterizations of Lusztig's totally positive parts of unipotent subgroups and Schubert varieties (\cite{BFZ96, BZ97}). 
When $\g$ is of symmetric type, 
the additive categorification of the twist automorphism  was given in \cite{GLS12} using 
the Sykes's functor $\Omega_w$ on the Frobenius subcategory associated with $w$ 
of the category of representations of a preprojective algebra.  It is well-known that  the functor $\Omega_w$ is 6-periodic when $\g$ is of finite type and $w$ is the longest element (\cite{AR96, ES98}). Thus the twist automorphism also has a periodicity in those cases. 
In \cite{KO17},  Kimura-Oya   constructed 
the quantum twist automorphism 
on the algebra $A_q[N_-^w]$ which is isomorphic to $\Q(q)\tens_{\Z[q^{\pm1}]}K(\tCw)$.
They showed that the  specialization at $q=1$ recovers the twist automorphism on $\C[N^w]$  and 
it is $6$-periodic when $\g$ is of finite type and $w$ is the longest element. 
Moreover,   it preserves the upper global basis.
Based on these properties, they
 asked whether the quantum twist automorphism
can be categorified by using finite-dimensional modules over quiver Hecke algebra or not (\cite[Remark 7.27]{KO17}).
The answer is affirmative, because  the quantum twist automorphism corresponds to the map induced from the left dual functor $\lD$ at the Grothendieck group level, up to the antiautomorphism which fixes classes of simple objects (see the formula in \cite[Theorem 6.1]{KO17}).
In this sense, we may say that the duality functor on $\tCw$ is a \emph{monoidal} categorification of the quantum twist automorphism in \cite{KO17}.

\medskip

This paper is organized as follows:
In Section 1, we recall quantum groups, quiver Hecke algebras, the category $\catC_w$ and some generalities on monoidal categories.
In Section 2, we develop the localization of monoidal categories via a real commuting family of  braiders.  
In Section 3, we recall the affinizations and $R$-matrices of quiver Hecke algebras and prove that the determinantial modules admit affinizations.
In Section 4, we study the braiders in $R\gmod$ and the localization of subcategories of $R\gmod$ in detail. 
In Section 5, we show that  the localization $\lRg$ of the category $\catC_w$ is left rigid and every simple objects in it has a right dual. In finite type $\g$ with the longest element $w_0$ case, we showed that $\lRg[w_0]$ is rigid and the duality functors are 6-periodic.

\medskip
{\bf Acknowledgments}

We thank Yoshiyuki Kimura for informing us his result with Hironori Oya and for fruitful communications, 
and the first author thank Christof Gei\ss\ and Bernard Leclerc for 
fruitful discussion.
The second, third and fourth authors gratefully acknowledge for the hospitality of RIMS (Kyoto University) during their visit in 2018.
The authors also would like to thank the anonymous reviewer for valuable comments.

\vskip 2em

\section{Preliminaries} \label{Sec: Preliminaries}

\subsection{Quantum groups}\

Let $I$ be an index set.
A {\it Cartan datum} $ (\cmA,\wlP,\Pi,\Pi^\vee,(\cdot,\cdot) ) $
consists of
\begin{enumerate}[{\rm (i)}]
\item a free abelian group $\wlP$, called the {\em weight lattice},
\item $\Pi = \{ \alpha_i \mid i\in I \} \subset \wlP$,
called the set of {\em simple roots},
\item $\Pi^{\vee} = \{ h_i \mid i\in I \} \subset \wlP^{\vee}\seteq
\Hom( \wlP, \Z )$, called the set of {\em simple coroots},
\item a $\Q$-valued symmetric bilinear form $(\cdot,\cdot)$ on $\wlP$,
\end{enumerate}
which satisfy
\begin{enumerate} [{\rm (a)}]
\item  $(\alpha_i,\alpha_i)\in 2\Z_{>0}$ for $i\in I$,
\item $\langle h_i, \lambda \rangle =\dfrac{2(\alpha_i,\lambda)}{(\alpha_i,\alpha_i)}$ for $i\in I$ and $\lambda \in \Po$,
\item $\cmA \seteq (\langle h_i,\alpha_j\rangle)_{i,j\in I}$ is
a {\em generalized Cartan matrix}, i.e.,
$\langle h_i,\alpha_i\rangle=2$ for any $i\in I$ and
$\langle h_i,\alpha_j\rangle \in\Z_{\le0}$ if $i\not=j$,
\item $\Pi$ is a linearly independent set,
\item for each $i\in I$, there exists $\Lambda_i \in \wlP$
such that $\langle h_j, \Lambda_i \rangle = \delta_{ij}$ for any $j\in I$.
\end{enumerate}
Let $\Delta$ (resp.\ $\prD$, $\nrD$) be the set of roots  (resp.\ positive roots, negative roots).
We set $\wlP_+ \seteq  \{ \lambda \in \wlP \mid  \langle h_i, \lambda\rangle \ge 0  \text{ for }i\in I\}$,
$ \rlQ = \bigoplus_{i \in I} \Z \alpha_i$, and $ \rlQ_+ = \sum_{i\in I} \Z_{\ge 0} \alpha_i$,
and write $\Ht (\beta)=\sum_{i \in I} k_i$  for $\beta=\sum_{i \in I} k_i \alpha_i \in \rlQ_+$.
For $i\in I$, we define
$$s_i(\lambda)=\lambda-\langle h_i, \lambda\rangle\alpha_i \qquad
\text{for $\lambda\in \wlP$},$$
and $\weyl$ is the subgroup of $\mathrm{Aut}(\wlP)$ generated by
$\{s_i\}_{i\in I}$.
For $w,v \in \weyl$, we write $w \ge v$
if there exists a reduced expression of $v$ which
appears in a subexpression of a reduced expression of $w$.

 For $\la$, $\mu\in\wlP$, we write
$\la\wle\mu$ if there exists a sequence of
real positive roots $\beta_k$ ($1\le k\le \ell$)
such that $\la=s_{\beta_\ell}\cdots s_{\beta_1}\mu$ and
$(\beta_k,s_{\beta_{k-1}}\cdots s_{\beta_1}\mu)>0$ for $1\le k\le\ell$.
Note that, for $\La\in\wlP_+$ and $\la$, $\mu\in \weyl\La$,
$\la\wle\mu$ holds if and only if there exist
$w$, $v\in\weyl$ such that
$\la=w\La$, $\mu=v\La$ and $v\le w$.

Let $U_q(\g)$ be the quantum group associated with $ (\cmA,\wlP,\Pi,\Pi^\vee,(\cdot,\cdot) ) $, which is an associative algebra over $\Q(q)$ generated by $e_i$, $f_i$ $(i\in I)$ and $q^h$
$(h\in \wlP^\vee)$ with certain defining relations (see \cite[Chap.\ 3]{HK02} for details).
Let $*$ be the $\Q(q)$-antiautomorphism of $U_q(\g)$ defined by
\[
e_i^* = e_i,\qquad f_i^* = f_i, \qquad {(q^h)}^* = q^{-h}.
\]
We denote by $U_q^+(\g)$ (resp.\ $U_q^-(\g)$) the subalgebra of $U_q(\g)$ generated by $e_i $ (resp.\ $f_i $) for $i\in I$.
For $n\in \Z_{\ge 0}$ and $i \in I$, we set $e_i^{(n)} \seteq e_i^n / [n]_i!$ and $f_i^{(n)} \seteq f_i^n / [n]_i!$,
where
\begin{align*}
q_i = q^{ (\alpha_i, \alpha_i)/2 }, \qquad
  [n]_i =\frac{ q^n_{i} - q^{-n}_{i} }{ q_{i} - q^{-1}_{i} } \quad  \text{ and } \quad
 [n]_i! = \prod^{n}_{k=1} [k]_i.
 \end{align*}
Let $U_\Zq^+(\g)$ (resp.\ $U_\Zq^-(\g)$) be the $\Zq$-subalgebra of $U_q^-(\g)$ generated by $e_i^{(n)}$ (resp.\ $f_i^{(n)}$) for $i\in I$ and $n \in \Z_{\ge0}$.
The \emph{unipotent quantum coordinate ring} is defined by
\[
\Aq[\n] = \bigoplus_{\beta \in \rlQ_-} \Aq[\n]_\beta, \qquad \text{ where $\Aq[\n]_\beta\seteq \Hom_{\Q(q)}(U_q^+(\g)_{-\beta}, \Q(q))$,}
\]
which is isomorphic to $U_q^-(\g)$ as a $\Q(q)$-algebra (\cite[Lemma 8.2.2]{KKKO18}).

We define its $\Z[q,q^{-1}]$-form $\Aq[\n]_{\Z[q,q^{-1}]}$ by
$$\Aq[\n]_{\Z[q,q^{-1}]}=\set{a\in \Aq[\n]}{\ang{a, U_\Zq^+(\g)}\subset\Z[q,q^{-1}]}.$$

\subsection{Quiver Hecke algebras}\

Let $\bR$ be a field.
 For $i,j\in I$, we choose polynomials
$\qQ_{i,j}(u,v) \in \bR[u,v]$ such that
\bna

\item $\qQ_{i,j}(u,v) = \qQ_{j,i}(v,u)$,

\item it is of the form
\begin{align*}
\qQ_{i,j}(u,v) =\bc
                   \sum\limits
_{p(\alpha_i , \alpha_i) + q(\alpha_j , \alpha_j) = -2(\alpha_i , \alpha_j) } t_{i,j;p,q} u^pv^q &
\text{if $i \ne j$,}\\[1.5ex]
0 & \text{if $i=j$,}
\ec
\end{align*}
where  $t_{i,j;-a_{ij},0} \in  \bR^{\times}$.
\end{enumerate}
For $\beta\in \rlQ_+$ with $ \Ht(\beta)=n$, we set
$$
I^\beta\seteq  \Bigl\{\nu=(\nu_1, \ldots, \nu_n ) \in I^n \;\big\vert\; \sum_{k=1}^n\alpha_{\nu_k} = \beta \Bigr\},
$$
on which the symmetric group $\mathfrak{S}_n = \langle s_k \mid k=1, \ldots, n-1 \rangle$ acts  by place permutations.

\begin{df}
\ For $\beta\in\rlQ_+$,
the {\em quiver Hecke algebra} $R(\beta)$ associated with $\cmA$ and $(\qQ_{i,j}(u,v))_{i,j\in I}$
is the $\bR$-algebra generated by
$$
\{e(\nu) \mid \nu \in I^\beta \}, \; \{x_k \mid 1 \le k \le n \},
 \; \{\tau_l \mid 1 \le l \le n-1 \}
$$
satisfying the following defining relations:
\begin{align*}
& e(\nu) e(\nu') = \delta_{\nu,\nu'} e(\nu),\ \sum_{\nu \in I^{\beta}} e(\nu)=1,\
x_k e(\nu) =  e(\nu) x_k, \  x_k x_l = x_l x_k,\\
& \tau_l e(\nu) = e(s_l(\nu)) \tau_l,\  \tau_k \tau_l = \tau_l \tau_k \text{ if } |k - l| > 1, \\[5pt]
&  \tau_k^2 e(\nu) = \qQ_{\nu_k, \nu_{k+1}}(x_k, x_{k+1}) e(\nu), \\[5pt]
&  (\tau_k x_l - x_{s_k(l)} \tau_k ) e(\nu) = \left\{
                                                           \begin{array}{ll}
                                                             -  e(\nu) & \hbox{if } l=k \text{ and } \nu_k = \nu_{k+1}, \\
                                                               e(\nu) & \hbox{if } l = k+1 \text{ and } \nu_k = \nu_{k+1},  \\
                                                             0 & \hbox{otherwise,}
                                                           \end{array}
                                                         \right. \\[5pt]
&( \tau_{k+1} \tau_{k} \tau_{k+1} - \tau_{k} \tau_{k+1} \tau_{k} )  e(\nu) \\[4pt]
&\qquad \qquad \qquad = \left\{
                                                                                   \begin{array}{ll}
\bQ_{\,\nu_k,\nu_{k+1}}(x_k,x_{k+1},x_{k+2}) e(\nu) & \hbox{if } \nu_k = \nu_{k+2}, \\
0 & \hbox{otherwise}, \end{array}
\right.\\[5pt]
\end{align*}
\end{df}
where
\begin{align*}
\bQ_{i,j}(u,v,w)=\dfrac{ \qQ_{i,j}(u,v)- \qQ_{i,j}(w,v)}{u-w}\in \bR[u,v,w].
\end{align*}
The algebra $R(\beta)$ has the $\Z$-grading defined by
\begin{align*}
\deg(e(\nu))=0, \quad \deg(x_k e(\nu))= ( \alpha_{\nu_k} ,\alpha_{\nu_k}), \quad  \deg(\tau_l e(\nu))= -(\alpha_{\nu_{l}} , \alpha_{\nu_{l+1}}).
\end{align*}

Note that  we have
\eq
\tau_ka(x)e(\nu)=
s_k\bl a(x) \br\tau_k e(\nu)+\delta(\nu_k=\nu_{k+1})\bl\partial_k a(x)\br e(\nu)
\label{eq:taupol}
\eneq
for $a(x)\in\corp[x_1,\ldots,x_n]$.
Here,  $\delta(P) =1$ if $P$ is true and $ \delta(P) = 0$ otherwise,  and  $\partial_k$ is the difference operator
\eqn
\partial_k a(x)=\dfrac{s_k a(x)-a(x)}{x_k-x_{k+1}}.
\eneqn

We denote by $R(\beta) \Mod$  the category of graded $R(\beta)$-modules with degree preserving homomorphisms.
We set $R(\beta)\gmod$ to be the full subcategory of $R(\beta)\Mod$ consisting of the modules which are  finite-dimensional over $\bR $, and set
 $R(\beta)\proj$ to be the full subcategory of $R(\beta)\Mod$ consisting of finitely generated  projective graded $R(\beta)$-modules.
We set $R\Mod \seteq \bigoplus_{\beta \in \rlQ_+} R(\beta)\Mod$, $R\proj \seteq \bigoplus_{\beta \in \rlQ_+} R(\beta)\proj$, and  $R\gmod \seteq \bigoplus_{\beta \in \rlQ_+} R(\beta)\gmod$.
The trivial $R(0)$-module of degree 0 is denoted by $\trivialM$.
For simplicity, we write ``a module" instead of ``a graded module''.
We define the grading shift functor $q$
by $(qM)_k = M_{k-1}$ for a graded module $M = \bigoplus_{k \in \Z} M_k $.
For $M, N \in R(\beta)\Mod $, $\Hom_{R(\beta)}(M,N)$ denotes the space of degree preserving module homomorphisms.
We define
\[
\HOM_{R(\beta)}( M,N ) \seteq \bigoplus_{k \in \Z} \Hom_{R(\beta)}(q^{k}M, N),
\]
and set $ \deg(f) \seteq k$ for $f \in \Hom_{R(\beta)}(q^{k}M, N)$.
When $M=N$, we write $\END_{R(\beta)}( M ) = \HOM_{R(\beta)}( M,M)$.
We sometimes write $R$ for $R(\beta)$ in $\HOM_{R(\beta)}( M,N )$ for simplicity.

For $M \in R(\beta)\gmod$, we set $M^\star \seteq  \HOM_{\bR}(M, \bR)$ with the $R(\beta)$-action given by
$$
(r \cdot f) (u) \seteq  f(\psi(r)u), \quad \text{for  $f\in M^\star$, $r \in R(\beta)$ and $u\in M$,}
$$
where $\psi$ is the antiautomorphism of $R(\beta)$ which fixes the generators.
We say that $M$ is \emph{self-dual} if $M \simeq M^\star$.

Let
$
e(\beta, \beta') \seteq \sum_{\nu \in I^\beta, \nu' \in I^{\beta'}} e(\nu, \nu'),
$
where $e(\nu, \nu')$ is the idempotent corresponding to the concatenation
$\nu\ast\nu'$ of
$\nu$ and $\nu'$.
Then there is an injective ring homomorphism
$$R(\beta)\tens R(\beta')\to e(\beta,\beta')R(\beta+\beta')e(\beta,\beta')$$
given by
$e(\nu)\tens e(\nu')\mapsto e(\nu,\nu')$,
$x_ke(\beta)\tens 1\mapsto x_ke(\beta,\beta')$,
$1\tens x_ke(\beta')\mapsto x_{k+\Ht(\beta)}e(\beta,\beta')$,
$\tau_ke(\beta)\tens 1\mapsto \tau_ke(\beta,\beta')$ and
$1\tens \tau_ke(\beta')\mapsto \tau_{k+\Ht(\beta)}e(\beta,\beta')$.
For $a\in R(\beta)$ and $a'\in R(\beta')$, the image of $a\tens a'$ is sometimes denoted by $a\etens a'$.

For $M \in R(\beta)\Mod$ and $N \in R(\beta')\Mod$, we set
$$
M \conv N \seteq R(\beta+\beta') e(\beta, \beta') \otimes_{R(\beta) \otimes R(\beta')} (M \otimes N).
$$
For $u\in M$ and $v\in N$, the image of
$u\tens v$ by the map $M\tens N\to M\conv N$ is sometimes denoted by
$u\etens v$.

We denote by $M \hconv N$ the head of $M \conv N$ and by $M \sconv N$ the socle of $M \conv N$.
We say that simple $R$-modules $M$ and $N$ \emph{strongly commute} if $M \conv N$ is simple.  A simple $R$-module
$L$ is \emph{real} if $L$ strongly commutes with itself. Note that if $M$ and $N$ strongly commute, then $M\conv N \simeq N \conv M$ up to a grading shift.

For $i\in I$ and the functors $E_i$ and $F_i$ are defined by
$$
E_i(M) = e(\alpha_i, \beta-\al_i) M \quad \text{ and }\quad F_i(M) = R(\alpha_i) \conv M \qt{for an $R(\beta)$-module $M$.} $$
For $i\in I $ and $n\in \Z_{>0}$, let $L(i)$ be the simple $R(\alpha_i)$-module concentrated on  degree 0 and
 $P(i^{n})$  the indecomposable  projective $R(n \alpha_i)$-module
whose head is isomorphic to $L(i^n) \seteq q_i^{\frac{n(n-1)}{2}} L(i)^{\conv n}$. 
For a simple module $M$, $\tF_i(M)$ (resp.\ $\tE_i(M)$) is the self-dual simple $R$-module being isomorphic to $L(i) \hconv M$ (resp.\ $\mathrm{soc}(E_i M)$). 
For $M\in  R(\beta) \Mod $, we define
\[
F_i^{(n)} M \seteq  P(i^{n}) \conv M, \qquad
E_i^{(n)} M \seteq \HOM_{R(n\alpha_i)} ( P(i^{n}), e(n\alpha_i, \beta - n\alpha_i) M).
\]
For $i \in I$ and $M \in R(\beta)\Mod$, we define
\begin{align*}
\wt(M) = - \beta, \ \   \ep_i(M) = \max \{ k \ge 0 \mid E_i^k M \ne 0 \}, \ \  \ph_i(M) = \ep_i(M) + \langle h_i, \wt(M) \rangle.
\end{align*}
We also can define $E_i^*$, $F_i^*$, $\ep^*_i$, etc in the same manner as above if we replace the roles of $e(\alpha_i, \beta-\al_i)$ and $R(\alpha_i)\conv -$ with the ones of
$e(\beta-\al_i, \alpha_i)$ and $- \conv R(\alpha_i)$.

We denote by $K(R\proj)$ and $K(R\gmod)$ the Grothendieck groups of $R\proj$ and $R\gmod$ respectively.

\begin{thm} [{\cite{KL09, KL11, R11}}] \label{Thm: categorification}
There exist isomorphisms of $\Zq$-bialgebras
\[
K(R\proj) \simeq  U_\Zq^-(\g) \quad \text{and} \quad K(R\gmod) \simeq  \Aqq[\n] .
\]
\end{thm}

For an $R(\beta)$-module $M$, we write
\eqn
e(i_1^{m_1}, \ldots, i_t^{m_t}, *) M &&\seteq  e( m_1 \alpha_{ i_1}, \ldots, m_t \alpha_{ i_t},  \beta - \sum_{k=1}^t  m_k \alpha_{  i_k } )M,\\[-1ex]
e(*,i_1^{m_1}, \ldots, i_t^{m_t}) M &&\seteq  e( \beta - \sum_{k=1}^t  m_k \alpha_{  i_k }, m_1 \alpha_{ i_1}, \ldots, m_t \alpha_{ i_t} )M.\\
\eneqn

For a sequence $(i_1,\ldots, i_m)$ in $I$ such that $(\al_{i_k},\al_{i_{k+1}})<0$
($1\le k\le m-1$) and $\al_{i_k}\not=\al_{i_{k+2}}$ ($1\le k\le m-2$), we denote by
$L(i_1,\ldots, i_m)$ the one-dimensional $R(\beta)$-module
($\beta=\sum_{k=1}^m\al_{i_k}$) with a basis $u$ on which $R(\beta)$ acts by
$e(i_1,\ldots ,i_m)u=u$, $x_ku=0$, $\tau_ku=0$.

\subsection{Cyclotomic quiver Hecke algebras} \label{Sec: cyclotomic} \

Let $A = \bigoplus_{n \in \Z}A_n$ be a commutative graded $\bR$-algebra such that $A_n = 0 $ for $n < 0$. For $\beta \in \rlQ_+$, we define the graded algebra
$R_A(\beta)$ by
$$
R_A(\beta) \seteq  A \otimes_\bR R(\beta).
$$
Let $\Lambda \in \wlP_+$ and $t_i$  an indeterminate of degree $(\alpha_i, \alpha_i )$ for $i\in I$.
A polynomial $f(t_i) = \sum_{ k=0}^{n} c_{k} t_i^k \in A[t_i]$
is \emph{homogeneous} of degree $d$ if 
$c_k\in A_{d-(\al_i,\al_i)k}$. 
We say that $f(t_i)$ is {\em quasi-monic} if $c_n\in A_0^\times$.
In this case, we write
$$
\deg f(t_i) = d \qtq\deg_{t_i} f(t_i)=n.
$$

We choose a family of polynomials
$a_{\Lambda}\seteq\{a_{\Lambda,i}(t_i)\}_{i\in I}$, where
$a_{\Lambda,i}(t_i)\in A[t_i]$ is a quasi-monic homogeneous polynomial with $ \deg_{t_i} a_{\Lambda, i} (t_i) =  \lan h_i,\Lambda\ran$.
Note that $\deg (  a_{\Lambda, i} (t_i)  ) = 2(\alpha_i, \Lambda)$.
Let $\lambda\in\Lambda-\rlQ_+$, and write $ \beta \seteq  \Lambda - \lambda$ and $n\seteq \Ht(\beta)$.
We define the \emph{cyclotomic quiver Hecke algebra}
\eq
R_A^{a_\Lambda}(\lambda)\seteq\dfrac{R_A(\beta)}{ \sum_{i\in I} R_A(\beta)a_{\Lambda, i}(x_n e(\beta-\alpha_i, \alpha_i))R_A(\beta)}.
\label{eq:cyclo}
\eneq
The algebra $R_A^{a_\Lambda}(\lambda)$ is a finitely generated projective module over $A$ by \cite[Remark 4.20]{KK11}.

Let $R_A^{a_\Lambda}(\lambda) \Mod$ be the category of graded $R_A^{a_\Lambda}(\lambda)$-modules.
We denote by $R_A^{a_\Lambda}(\lambda)\allowbreak\proj$ and  
 $R_A^{a_\Lambda}(\lambda)\gmod$
 the category of finitely generated  projective graded $R_A^{a_\Lambda}(\lambda)$-modules and
the category of graded $R_A^{a_\Lambda}(\lambda)$-modules which are finite-dimensional over $\bR$, respectively.
Their morphisms are homogeneous homomorphisms of degree zero.
Set $R_A^{a_\Lambda}\proj \seteq \bigoplus_{\beta \in \rlQ_+} R_A^{a_\Lambda}(\Lambda - \beta)\proj$, $R_A^{a_\Lambda}\gmod \seteq \bigoplus_{\beta \in \rlQ_+} R_A^{a_\Lambda}(\Lambda - \beta)\gmod$, etc.
We define the functors
\begin{align*}
F_i^{a_\Lambda} &:  R_A^{a_\Lambda}(\lambda)\Mod \rightarrow  R_A^{a_\Lambda}(\lambda-\alpha_i)\Mod , \\
E_i^{a_\Lambda} &:  R_A^{a_\Lambda}(\lambda)\Mod \rightarrow R_A^{a_\Lambda}(\lambda+\alpha_i)\Mod
\end{align*}
by
$
F_i^{a_\Lambda}M = R_A^{a_\Lambda}(\lambda-\alpha_i)e(\alpha_i,\beta)\otimes_{R_A^{a_\Lambda}(\lambda)}M $ and $
E_i^{a_\Lambda}M = e(\alpha_i,\beta-\alpha_i)M
$
for $M\in R_A^{a_\Lambda} (\lambda)\Mod$.
They give a categorification of  
the simple $\Uq$-module $V_q(\Lambda)$ with highest weight $\La$. 

When $A = \bR$ and $a_{\La, i}(t_i) = t_i^{\langle h_i, \Lambda \rangle}$ for $i\in I$, we simply write $R^\Lambda(\lambda)$, $F_i^\Lambda$, $E_i^\Lambda$ instead of $R_A^{a_\Lambda}(\lambda)$, $F_i^{a_\Lambda}$, $E_i^{a_\Lambda}$ respectively.

\begin{thm}[{\cite[Theorem 6.2]{KK11}}]
For $\La\in \wlP^+$, there exist $U_\Zq (\g)$-module isomorphisms
\[
K(R^{\Lambda}\proj) \simeq V_\Zq (\Lambda), \quad K(R^\Lambda\gmod) \simeq V_\Zq (\Lambda)^\vee.
\]
\end{thm}

For $n\in\Z_{\ge0}$, we define
\eqn
F_i^{a_\La\,(n)}\,M&=&\HOM_{R_A(n\al_i)}\bl P_A(i^{  n  }),(F_i^{a_\La})^n\,M\br,\\
E_i^{a_\La\,(n)}\,M&=&\HOM_{R_A(n\al_i)}\bl P_A(i^{  n  }),(E_i^{a_\La})^n\,M\br,
\eneqn
which give exact functors:
\eq
&&\ba{rcl}
F_i^{a_\Lambda\,(n)} &:&  R_A^{a_\Lambda}(\lambda)\Mod \rightarrow  R_A^{a_\Lambda}(\lambda-n\alpha_i)\Mod , \\
E_i^{a_\Lambda\,(n)} &:&  R_A^{a_\Lambda}(\lambda)\Mod \rightarrow R_A^{a_\Lambda}(\lambda+n\alpha_i)\Mod.
\ea\eneq
Then the following lemma is an easy consequence of the theory of
$\mathfrak{sl}_2$-categorification due to Rouquier (\cite{R08}).
\Prop\label{prop:htlt}
Let $\Lambda \in \wlP^+$ and $\la\in\weyl\La$ such that $n\seteq\ang{h_i,\la}\ge0$.
Then we have category equivalences, quasi-inverse to each other:
$$\xymatrix@C=9ex{
R_A^{a_\Lambda}(\lambda)\Mod\ar@<.7ex>[r]^{F_i^{a_\Lambda\,(n)}}&
R_A^{a_\Lambda}(s_i\lambda)\Mod\,.\ar@<.7ex>[l]^{E_i^{a_\Lambda\,(n)}}
}$$
In particular,
we have
$Z\bl R_A^{a_\Lambda}(\lambda)\br\simeq
Z\bl R_A^{a_\Lambda}( s_i \lambda)\br$.
Here $Z(R)$ denotes the center of an algebra $R$.
\enprop

Note that the last assertion follows from
$Z(R)\simeq\End(\id_{R\Mod})$.

\subsection{The categories $\catC_w$} \label{Sec: catC} \

In this subsection, we review briefly convex orders and the categories $\catC_w$ \cite{KKOP18,TW16}.
\begin{df} \
\begin{enumerate} [\rm(i)]
\item A \emph{preorder} $\preceq$ on a set $X$ is a binary relation on $X$ satisfying
\bna
\item $x \preceq x$ for any $x\in X$,
\item if $x \preceq y$ and $y \preceq z$ for $x,y,z \in X$, then $x \preceq z$.
\end{enumerate}
\item We say that a preorder $\preceq$ on $X$ is {\em total}
if we have either
$x\preceq y$ or $y\preceq x$ for any $x,y\in X$.
\item For a preorder $\preceq$, we say that $x$ and $y$ are \emph{$\preceq$-equivalent} if $x \preceq y$ and $y \preceq x$.
The equivalence class for $\preceq$ is called $\preceq$-equivalence class.
\item For subsets $A$ and $B$, we write $A \preceq B$ if $a \preceq b$ for any $a \in A$ and $b \in B$.
\end{enumerate}
\end{df}

\begin{df} \label{Def: face}
A \emph{face} is a decomposition of a subset $X$ of an $\R$-vector space  into three disjoint subsets $ X = A_- \sqcup A_0 \sqcup A_+ $ such that
\begin{align*}
( \Sp  A_+ + \SpR A_0) \cap \Sp A_- &= \{ 0\}, \\
( \Sp  A_- + \SpR A_0) \cap \Sp A_+ &= \{ 0\},
\end{align*}
where $\SpR S $ is the $\R$-vector space spanned by $S$ and
$\Sp S $ is the subset of $\SpR S $ whose elements are linear combinations of $S$ with non-negative coefficients.
\end{df}
We set $\SpR \emptyset = \Sp \emptyset = \{ 0 \}$.

\begin{df} \label{Def: convex}
Let $V$ be an $\R$-vector space and let $X$ be a subset of $V \setminus \{0\}$.
\begin{enumerate}[\rm(i)]
\item A \emph{convex} preorder $\preceq$ on $X$  is a total preorder on $X$ such that, for any $\preceq$-equivalence class $\eC$,
the triple
$(  \{ x \in X \mid x \prec \eC \}, \eC,  \{ x \in X \mid x \succ \eC \} )$
is a face.
\item A convex preorder $\preceq$ on $X$ is called a \emph{convex order} if every $\preceq$-equivalence class is of the form $X \cap l$
for some line $l$ in $V$ through the origin.
\end{enumerate}
\end{df}

For each reduced expression of  $w \in \weyl $, one can define a corresponding convex order on $\prD$ as follows.

\begin{prop}[{\cite[Proposition 1.24]{KKOP18}}] \label{Prop: convex preorder for w}
Let $\underline{w} = s_{i_1}s_{i_2}\cdots s_{i_l}$ be a reduced expression of $w \in \weyl$.
We set $\beta_k \seteq s_{i_1}\cdots s_{i_{k-1}}(\alpha_{i_k}) $ so that $\prD \cap w\nrD = \{ \beta_1, \ldots, \beta_l \}$.
Then there is a convex order $\preceq$ on $\prD$ such that
$$
\beta_1 \prec \beta_2 \prec \cdots \prec \beta_l \prec \gamma
$$
for any $\gamma \in \prD \cap w\prD$.
\end{prop}

We fix a convex order $\preceq$ on $\prD$ given in Proposition \ref{Prop: convex preorder for w}.
Note that it can be extended uniquely to the convex order on $\Z_{ >0 }\prD\seteq  \{ n \gamma \mid n \in \Z_{  >0  },\ \gamma \in \prD \}$, which means that
we write $n\beta \preceq n'\beta'$
for $\beta,\beta'\in\prD$ if
$\beta\preceq \beta'$.

For $M \in  R(\beta)\Mod$, we define
\begin{align*}
\gW(M) & \seteq  \{  \gamma \in  \rlQ_+ \cap (\beta - \rlQ_+)  \mid  e(\gamma, \beta-\gamma) M \ne 0  \}, \\
 \sgW(M) & \seteq  \{  \gamma \in  \rlQ_+ \cap (\beta - \rlQ_+)  \mid  e( \beta-\gamma, \gamma) M \ne 0  \}. 
\end{align*}
Let $L$ be a simple $R(\beta)$-module for $\beta \in \rlQ_+$.
We say that $L$ is \emph{$\preceq$-cuspidal} if
\bna
\item $\beta \in \Z_{>0} \prD $,
\item$\gW(L) \subset \Sp  \{ \gamma \in \prD \mid \gamma \preceq \beta \} $.
\end{enumerate}

\begin{thm}[{\cite[Theorem 2.8]{KKOP18}}]\label{Thm: cuspidal decomposition}
For  a simple $R$-module $L$, there exists a unique sequence $(L_1,L_2, \ldots, L_h)$ of $\preceq$-cuspidal modules \ro up to isomorphisms\rf\ such that
\bnum
\item $-\wt(L_k) \succ -\wt(L_{k+1}) $ for $k=1, \ldots, h-1$,
\item $L$ is isomorphic to the head of $L_1 \conv L_2 \conv \cdots \conv L_h$.
\end{enumerate}
\end{thm}
We set $\cd_{\preceq}(L) \seteq (L_1, \ldots, L_h)$  in Theorem \ref{Thm: cuspidal decomposition}, which is called  the \emph{$\preceq$-cuspidal decomposition} of $L$.
We write $\cd$ instead of $\cd_{\preceq}$ for simplicity when there is no afraid of confusion.

For $w\in \weyl$, we denote by $\catC_{w}$ the  full subcategory of $R\gmod$ whose objects $M$ satisfy
\begin{align*}
\gW(M) \subset \Sp( \prD \cap w \nrD ).
\end{align*}

\begin{prop} [{\cite[Proposition 2.18]{KKOP18}}] \label{Prop: membership}
Let $\underline{w} = s_{i_1}s_{i_2} \cdots s_{i_\ell}$ be a reduced expression of $w\in \weyl$, and
$\beta_\ell = s_{i_1} \cdots s_{i_{\ell-1}}(\alpha_{i_\ell})$.
We take a simple $R$-module $L$ and set
\[
\cd(L) \seteq (L_1, L_2, \ldots, L_h),\quad \gamma_k \seteq -\wt(L_k)\quad \text{for }k=1, \ldots, h.
\]
Then we have that $L \in \catC_{w}$ if and only if  $\beta_\ell \succeq \gamma_1$.
\end{prop}

By the construction and  Proposition \ref{Prop: membership},
$\catC_w$ is the smallest monoidal abelian full subcategory of $R\gmod$ such that
\begin{enumerate}[(a)]
\item $\catC_w$ is stable under the subquotients, extension, and grading shifts,
\item $\catC_w$ contains all the $\preceq$-cuspidal modules corresponding to $\beta_k = s_{i_1} \cdots s_{i_{k-1}}(\alpha_{i_k})$ for $k=1, \ldots, \ell$.
\end{enumerate}
It was proved in \cite{KKKO18} that $\catC_w$ gives a monoidal categorification of the quantum unipotent coordinate algebra $A_q(\n(w^{-1}))$ when it is of symmetric type.

\subsection{Monoidal categories}

We shall review monoidal categories and related notions.  We refer the reader to \cite[Appendix A]{KKK18} and \cite{KS06} for more details.

A \emph{monoidal category} (or \emph{tensor category}) is a datum consisting of
\begin{enumerate}[\rm (a)]
\item a category $\catT$,
\item  a bifunctor $\cdot \otimes \cdot : \catT \times \catT \rightarrow \catT$,
\item an isomorphism $a(X,Y,Z): (X \otimes Y) \otimes Z \buildrel \sim \over \longrightarrow X \otimes (Y \otimes Z)$ which is functorial in $X,Y,Z \in \catT$,
\item an object $\triv$ endowed with
an isomorphism  $\epsilon\cl
\triv \otimes \triv \isoto\triv$
\end{enumerate}
such that
\be[{(1)}]
\item
the diagram below commutes for all $X,Y,Z,W \in \catT$:
$$
\xymatrix{
( ( X \otimes Y ) \otimes Z ) \otimes W  \ar[d]_{ a( X,Y,Z) \otimes W  }  \ar[rr]^{ a(X \otimes Y,Z,W) } &&  \ar[dd]^{ a( X,Y,Z \otimes W)  }    ( X \otimes Y ) \otimes (Z  \otimes W)   \\
 (  X \otimes (Y  \otimes Z) ) \otimes W  \ar[d]_{ a( X,Y\otimes Z, W)  }   &&    \\
 X \otimes ( (Y \otimes Z ) \otimes W ) \ar[rr]_ {X \otimes a( Y,Z,W)} &&    X \otimes ( Y \otimes (Z \otimes W )) .
}
$$
\item the functors $\catT\ni X\mapsto \one \tens X\in \catT$
and  $\catT\ni X\mapsto X\tens \one\in\catT$ are fully faithful.
\ee

We call $\triv$ a \emph{unit object} of $\catT$.
We have canonical isomorphisms $\triv\tens X\simeq X\tens \triv\simeq X$
for any $X\in\catT$.
For $n \in \Z_{> 0}$ and $X\in\catT$, we set $X^{\otimes n} = \underbrace{X \otimes \cdots \otimes X}_{n \text{ times}}$, and $X^{\otimes 0} = \triv$.

For monoidal categories $\catT$ and $\catT'$, a functor $ F \cl  \catT \rightarrow \catT'$ is called a \emph{monoidal functor}
if it is endowed with an isomorphism $\xi_F\cl F(X \otimes Y)  \buildrel \sim \over \longrightarrow F(X) \otimes F(Y)$ which is functorial in $X,Y \in \catT$ such that
the diagram
$$
\xymatrix{
F( ( X \otimes Y ) \otimes Z )   \ar[d]_{ \xi_F(X \otimes Y, Z)  }  \ar[rrr]^{ F( a(X , Y,Z)) } &&&  \ar[d]^{ \xi_F(X , Y \otimes Z)  }   F ( X \otimes (Y   \otimes  Z) )   \\
 F(  X \otimes Y)   \otimes F(Z)   \ar[d]_{ \xi_F(X , Y)\otimes F( Z)  }   &&& \ar[d]^{ F(X)\otimes \xi_F( Y \otimes Z)  }   F ( X )\otimes F(Y   \otimes  Z) )   \\
 ( F(X) \otimes F (Y) ) \otimes F(Z ) \ar[rrr]_ { a(F(X) , F(Y),F(Z))  } &&&     F ( X ) \otimes (F(Y)   \otimes  F(Z) )  .
}
$$
commutes for all $X,Y,Z \in \catT$. We omit to write $\xi_F$ for simplicity.
A monoidal functor $F$ is  called \emph{unital} if $(F(\triv), F(\epsilon))$ is a unit object.
In this paper, we simply write
a ``monoidal functor'' for a unital monoidal functor.

We say that a monoidal category $\catT$ is an \emph{additive} (resp.\ \emph{abelian}) monoidal category if $\catT$ is additive (resp.\ abelian) and the bifunctor $\cdot \otimes \cdot$ is bi-additive.
Similarly, for a commutative ring $\bR$, a monoidal category $\catT$ is \emph{$\bR$-linear} if $\catT$ is $\bR$-linear and the bifunctor $\cdot \otimes \cdot$ is $\bR$-bilinear.

An object $X \in \catT$ is \emph{invertible} if the functors $\catT \rightarrow \catT$ given by $Z \mapsto Z \otimes X$ and $Z \mapsto X \otimes Z$ are equivalence of categories.
If $X$ is invertible, then one can find an object $Y$ and isomorphisms $f\cl X \otimes Y \buildrel \sim\over \rightarrow \triv$ and
$g\cl Y \otimes X \buildrel \sim\over \rightarrow \triv$ such that the diagrams below commute:
$$
\xymatrix{
X \otimes Y \otimes X  \ar[rr]^{ f \otimes X } \ar[d]_{ X \otimes g  } &&  \ar[d]  \triv \otimes X   \\
 X \otimes \triv \ar[rr]_ {  } &&   X,
}
\quad
\xymatrix{
Y \otimes X \otimes Y  \ar[rr]^{ g \otimes Y } \ar[d]_{ Y \otimes f  } &&  \ar[d]  \triv \otimes Y   \\
 Y \otimes \triv \ar[rr]_ {  } &&   Y.
}
$$
The triple $(Y,f,g)$ is unique up to a unique isomorphism. We write $Y = X^{\otimes -1}$.

\subsubsection*{Adjunctions}

\begin{df}
Let $X$ and $ Y $ be objects in a monoidal category $ \catT$, and $\ep\cl X \otimes  Y \rightarrow \triv $ and $ \eta\cl \triv \rightarrow Y \otimes X$ morphisms in $\catT$.
\begin{enumerate} [\rm(i)]
\item
We say that $(\ep, \eta)$ is an \emph{adjunction} if the following are satisfied:
\bna
\item the composition
$X \simeq X \otimes \triv \To[{X \otimes \eta}] X \otimes Y \otimes X
\To[{\ep \otimes X}] \triv \otimes X \simeq X$
is the identity of $X$,
\item the composition $Y \simeq \triv \otimes Y \To[{\eta \otimes Y}]
 Y \otimes X \otimes Y \To[{Y \otimes \ep}]Y \otimes \triv \simeq Y$
is the identity of $Y$.
\end{enumerate}
\item
The pair $(\ep, \eta)$ is a \emph{quasi-adjunction} if
 the compositions $ X \otimes \triv \To[{X \otimes \eta}]X \otimes Y \otimes X \To [{\ep \otimes X}] \triv \otimes X $
and $ \triv \otimes Y \To[{\eta \otimes Y} ] Y \otimes X \otimes Y
\To[Y \otimes \ep]  Y \otimes \triv$
are isomorphisms.
\end{enumerate}

\end{df}
In the case when $(\ep, \eta)$ is an adjunction, we say that
the pair $(X,Y)$ is called a  \emph{dual pair}, or $X$ is a \emph{left dual} to $Y$ and $Y$ is a \emph{right dual} to $X$.
Note that a left dual (resp.\ right dual) of an object is unique up to a unique isomorphism if it exists.

\begin{df}
A monoidal category $\catT$ is \emph{rigid} if every object in $\catT$ has left and right duals.
\end{df}

\begin{lem}
 Let $(X,Y)$ be a dual pair in a monoidal category $\catT$. Then for $Z,W \in \catT$, there are isomorphisms
$$
\Hom_\catT( Z, W \otimes X ) \simeq \Hom_\catT( Z \otimes Y,W ), \quad
\Hom_\catT( X \otimes Z, W ) \simeq \Hom_\catT( Z ,Y \otimes W ).
$$
\end{lem}

\begin{lem}[{\cite[Lemma A.2]{KKK18}}] \label{Lem: quasi-adj}
Let $\catT$ be a monoidal category and $\ep\cl X \otimes Y \rightarrow \triv$ be a morphism in $\catT$. Then the following are equivalent.
\begin{enumerate}[\rm (i)]
\item There exists a morphism $\eta\cl \triv \rightarrow Y \otimes X$ such that $(\ep, \eta)$ is an adjunction.
\item There exists a morphism $\eta\cl \triv \rightarrow Y \otimes X$ such that $(\ep, \eta)$ is a quasi-adjunction.
\item the map $\Hom_\catT( Z, W \otimes X ) \to \Hom_\catT( Z \otimes Y,W )
$ that associates $f\cl Z\to W\otimes X$ the morphism
$Z\tens Y\To[f\otimes Y]W\tens X\tens Y\To[W\otimes\eps]W$
is bijective for any $Z,W\in\catT$.
\item the map $\Hom_\catT( Z, Y\tens W ) \to \Hom_\catT(X\tens Z,W )
$ that associates $f\cl Z\to Y\otimes W$ the morphism
$X\otimes Z\To[X\otimes f]X\otimes Y\otimes W\To[\eps\otimes W]W$
is bijective for any $Z,W\in\catT$.
\end{enumerate}
Moreover if these equivalent conditions are satisfied, the morphism $\eta$ in {\rm(i)}
is unique.
\end{lem}

\begin{lem} [{\cite[Lemma A.3]{KKK18}}] \label{Lem: ep eta ind}
Let $\catT$ be a $\bR$-linear monoidal category such that  $\cdot \otimes \cdot$ is exact.
Let
$$
0 \rightarrow X' \To[f] X \To[f'] X'' \rightarrow 0 \qtq
0 \rightarrow Y'' \To[g'] Y \To[g] Y'\To0
$$
be exact sequences and let
\begin{align*}
&\ep'\cl X' \otimes Y' \rightarrow \triv, \quad \ep\cl X \otimes Y \rightarrow \triv, \quad \ep''\cl X'' \otimes Y'' \rightarrow \triv, \\
&\eta'\cl   \triv \rightarrow Y' \otimes X' , \quad \eta\cl   \triv \rightarrow  Y \otimes X, \quad \eta''\cl   \triv \rightarrow  Y'' \otimes X''
\end{align*}
be morphisms such that the diagrams below commute
up to constant multiples:
$$
\xymatrix{
X' \otimes Y  \ar[r]^{g} \ar[d]_{f}  &  \ar[d]^{\ep'}  X' \otimes Y'  \\
X \otimes Y  \ar[r]^{ \ep }   &   \triv  \\
 X \otimes Y''  \ar[u]^{g'}  \ar[r]_ {f'  } &   \ar[u]_{\ep''} X'' \otimes Y'',
}
\qquad
\xymatrix{
Y' \otimes X'  \ar[r]^{f}   &    Y' \otimes X  \\
\triv  \ar[d]_{\eta''} \ar[u]^{\eta'} \ar[r]^{ \eta }   &   \ar[d]^{f'}   \ar[u]_{g}
   Y \otimes X  \\
 Y'' \otimes X''    \ar[r]_ {g'  } &   Y \otimes X''.
}
$$
If $ (\ep', \eta')$ and $(\ep'', \eta'')$ are quasi-adjunctions, then  $(\ep, \eta)$ is also a quasi-adjunction.
\end{lem}

\subsubsection*{Commuting families}

We now consider a family $\{ P_i \}_{i\in I}$ of objects and a family of isomorphisms $\{ B_{i,j}\cl  P_i \otimes P_j  \buildrel \sim\over \longrightarrow P_j \otimes P_i  \}_{i,j\in I}$
in a monoidal category $\catT$. We say that $(\{ P_i \}_{i\in I},  \{B_{i,j}\}_{i,j\in I} )$ is a \emph{commuting family} if
it satisfies the following:
\begin{enumerate}
\item[(a)] $B_{i,i} = \id_{P_i \otimes P_i}$ for any $i\in I$,
\item[(b)] $B_{j,i}\circ B_{i,j} = \id_{P_i \otimes P_j}$ for any $i,j \in I$,
\item[(c)] the isomorphisms $ B_{i,j} $ satisfy the \emph{Yang-Baxter equation}, which means that the following
diagram commutes for any $i,j,k\in I$:
\begin{equation} \label{Eq: cf pentagon}
\begin{aligned}
\xymatrix@R=4ex@C=7ex{
& \ar[ld]_{B_{i,j} \otimes P_k}  P_i \otimes P_j \otimes P_k \ar[rd]^{P_i \otimes B_{j,k}}  \\
P_j \otimes P_i \otimes P_k \ar[d]_{P_j \otimes B_{i,k}}  &&   \ar[d]^{B_{i,k}\otimes P_j} P_i \otimes P_k \otimes P_j \\
P_j \otimes P_k \otimes P_i  \ar[rd]_{B_{j,k}\otimes P_i}   &&  \ar[ld]^{P_k \otimes B_{i,j}} P_k \otimes P_i \otimes P_j \\
& P_k \otimes P_j \otimes P_i \,.
}
\end{aligned}
\end{equation}
\end{enumerate}

We set
$$
\lG \seteq  \Z^{\oplus I} \quad \text{ and } \quad \lG_{\ge 0} \seteq  \Z_{\ge 0}^{\oplus I}.
$$
For $i\in I$, let $e_i$ be the image of $1$ under the canonical embedding $\Z \rightarrowtail \lG$ to the $i$-th component so that
$\{ e_i \mid i\in I \} $ is a basis of $\lG$.
Then we have the following lemma.

\begin{lem}[{\cite[Section A.4]{KKK18}}] \label{Lem: cf}
Let $(\{ P_i \}_{i\in I},  \{B_{i,j}\}_{i,j\in I} )$ is a commuting family in a monoidal category $\catT$. Then there exist
\begin{enumerate} [\rm(a)]
\item an object $P^\alpha$ for any $\alpha \in \lG_{\ge0}$,
\item  an isomorphism $\xi_{\alpha, \beta} \cl  P^\alpha \otimes P^\beta \isoto P^{\alpha+\beta}$ for any $\alpha,\beta\in \lG_{\ge0}$
\end{enumerate}
such that
\begin{enumerate} [\rm(i)]
\item  $P^0=\triv$  and $ P^{e_i} = P_i $ for $i\in I$,
\item  the diagram
\begin{equation} \label{Eq: xi sum}
\begin{aligned}
\xymatrix{
P^\alpha \otimes P^\beta \otimes P^\gamma \ar[d]_{ P^\alpha \otimes \xi_{\beta, \gamma}} \ar[rr]^{\xi_{\alpha, \beta} \otimes P^\gamma} && \ar[d]^{\xi_{\alpha+\beta, \gamma}} P^{\alpha+\beta}\otimes P^\gamma \\
P^\alpha \otimes P^{\beta + \gamma} \ar[rr]^{\xi_{\alpha, \beta+\gamma}} && P^{\alpha+\beta +\gamma}
}
\end{aligned}
\end{equation}
commutes for any $\alpha, \beta, \gamma \in \lG_{\ge0}$,
\item the diagrams
\begin{equation} \label{Eq: CF ij}
\begin{aligned}
\ba{ccc}
\xymatrix{
P^0 \otimes P^0  \ar[d]_{ \wr } \ar[rr]^{\xi_{0,0} } && \ar[d]^{ \wr } P^{0}\\
\triv \otimes \triv  \ar[rr]^{ \simeq} && \ \ \triv  \ ,
}
\ba{c}\\[3ex]\qtq\ea\quad
\xymatrix{
 \ar[d]_{ B_{i,j}} \ar[drr]^{\xi_{e_i, e_j}} P^{e_i}\otimes P^{e_j} && \\
P^{e_j} \otimes P^{e_i}  \ar[rr]^{\xi_{e_j, e_i}} && P^{e_i+e_j}
}\ea
\end{aligned}
\end{equation}

commute for any $i,j\in I$.
\end{enumerate}
Moreover,  such a datum $(\{ P^\alpha\}_{\alpha \in \lG_{\ge 0}} ,   \{ \xi_{\alpha, \beta} \}_{\alpha,\beta \in \lG_{\ge0}} )$
is unique up to a unique isomorphism.
\end{lem}

\section{Localizations} \label{Sec: Localization}

A localization of a monoidal category using a commuting family of central objects was explained in \cite[Appendix A]{KKK18}.
In this section, we shall generalize this localization to a \emph{real commuting family of \ro graded\/\rf\ braiders}.

\subsection{Real commuting family of braiders} \label{Sec: RCB}  \

\begin{df}\ \label{def: braider}
\begin{enumerate}[\rm (i)]
\item
A \emph{left braider} of a monoidal category $\catT$ is a pair $(C, R_C)$ of an object $C$ and a morphism
\begin{align*}
R_C(X)\cl  C \otimes X \longrightarrow X \otimes C
\end{align*}
which is functorial in $X \in \catT$ such that the following diagrams commutes:
\begin{equation} \label{Eq: central obj}
\begin{aligned}
\xymatrix{
C \otimes X \otimes Y   \ar[rr]^{R_C(X)\otimes Y}  \ar[drr]_{R_C(X \otimes Y)\ \ }  &  &   X \otimes C \otimes Y  \ar[d]^{X \otimes R_C(Y)}   \\
& &   X \otimes Y \otimes C  ,
}
\ \
\xymatrix{
C \otimes \triv   \ar[rr]^{R_C(\triv)}  \ar[drr]_{ \simeq }  &  &   \triv \otimes C   \ar[d]^{ \wr}   \\
& &    C
}
\end{aligned}
\end{equation}
\item Assume that $\catT$ is a $\bR$-linear monoidal category. A left braider $(C, R_C)$ is called \emph{real} if $R_C(C) \in \bR^\times \id_{C\otimes C}$.
\end{enumerate}
\end{df}

Similarly we can define the notion of {\em right braiders}
by reversing the order of tensor products.
Since we treat mainly left braiders in this paper,
{\em we simply say braiders for left braiders in the sequel.}

Note that, for $f \in \Hom_\catT (X,Y) $, the diagram
\begin{equation} \label{Eq: natural isomorphism}
\begin{aligned}
\xymatrix{
C \otimes X \ar[rr]^{C \otimes f} \ar[d]_{R_C(X)}  && C \otimes Y \ar[d]^{R_C(Y)}  \\
X \otimes C \ar[rr]^{f \otimes C}  && Y \otimes C
}
\end{aligned}
\end{equation}
commutes since $R_C$ is a natural transformation between left and right tensoring with $C$ by the definition.

A braider $ ( C, R_{C}  )$ is called a \emph{central object} if $ R_{C}(X)$ is an isomorphism for any $X \in \catT$.

Let $\catTc$ be the category of braiders in $\catT$.
A morphism from $(C,R_C)$ to $(C', R_{C'})$ in $\catTc$ is 
a morphism $f \in \Hom_\catT(C,C')$ such that the following diagram commutes for any $X \in \catT$:
$$
\xymatrix{
C \otimes X    \ar[rr]^{f \otimes X}  \ar[d]_{R_C(X )}  &  &  C' \otimes X  \ar[d]^{ R_{C'}(X)}   \\
X \otimes C  \ar[rr]^{X \otimes f} & &   X \otimes C'.
}
$$
For braiders $(C_1, R_{C_1})$ and $(C_2, R_{C_2})$ of $\catT$, let
$ R_{C_1 \otimes C_2}(X)$ be the composition
$$
 C_1 \otimes C_2 \otimes X \To[{ R_{C_2}(X)}] C_1 \otimes X \otimes C_2
\To[{R_{C_1}(X) }]X \otimes C_1 \otimes C_2
$$
for $X \in \catT$. Then, one can show that $(C_1 \otimes C_2, R_{C_1\otimes C_2}) $ is also a braider of $\catT$.
Thus the category $\catTc$ has a structure of a monoidal category, and
there is a canonical faithful monoidal functor $\catTc \rightarrow \catT$.
We regard an object of $\catTc$ as an object of $\catT$ if there is no afraid of confusion.

The morphism $R_{C_1}(C_2)\cl C_1 \otimes C_2 \rightarrow C_2\otimes C_1$ is a morphism of braiders because
the following diagram commutes:
$$
\xymatrix{
C_1 \otimes C_2 \otimes  X    \ar[rr]^{R_{C_2}(X)}  \ar[d]_{R_{C_1}(C_2  )} \ar[drr]|-{R_{C_1}(C_2 \otimes X )}  && C_1 \otimes X \otimes  C_2  \ar[drr]|-{ R_{C_1}(X\otimes C_2)}   \ar[rr]^{R_{C_1}(X)}   && X \otimes C_1 \otimes  C_2  \ar[d]^{ R_{C_1}(C_2)} \\
 C_2\otimes C_1 \otimes X    \ar[rr]_{R_{C_1}(X)}    && C_2 \otimes X \otimes C_1    \ar[rr]_{R_{C_2}(X)}   &&  X \otimes C_2  \otimes  C_1.
}
$$
Let $\bR$ be a commutative ring with unity, and let $\catT$ be a $\bR$-linear monoidal category.

\begin{df}[{cf.\ Example~\ref{Ex: constants}}]\label{Def: real cf}
Let $I$ be an index set and let $ \{(C_i, R_{C_i})\}_{i\in I}$ be a
family of braiders in $\catT$.
We say that  $ (  C_i ,  R_{C_i} )_{i\in I}$ is a \emph{real commuting family of braiders} if
 \bna
\item $R_{C_i} (C_i) \in \bR^\times \id_{C_i \otimes C_i}$,\label{cond:a}
\item $ R_{C_j}(C_i) \circ R_{C_i}(C_j) \in \bR^{\times} \id_{C_i \otimes C_j} $.
\label{cond:b} 
\end{enumerate}
\end{df}
Note that the morphisms $R_{C_j} (C_i)$ satisfy the Yang-Baxter equation $\eqref{Eq: cf pentagon}$.

\begin{lem} \label{Lem: rcf - cf}
Let $(C_i, R_{C_i} )_{i\in I}$ be a real commuting family of braiders in $\catT$.
\bnum
\item
Then there exists a family $\{\eta_{ij}\}_{i,j\in I}$
of elements in $\corp^\times$ such that
\eqn
R_{C_i}(C_i)&&=\eta_{ii}\; \id_{C_i \otimes C_i},\\
R_{C_j}(C_i) \circ R_{C_i}(C_j)&& = \eta_{ij}\eta_{ji}\;\id_{C_i \otimes C_j}
\eneqn
for all $ i ,j \in I$.
\item Let $\{\eta_{ij}\}_{i,j\in I}$ be as in {\rm(i)} and set $B_{ij}=\eta_{ij}^{-1}R_{C_i}(C_j)$.
Then $(\{C_i\}_{i\in I},\{ B_{ij}\}_{i,j\in I})$ is a commuting family in the monoidal category $\catTc$.
\ee
\end{lem}
\begin{proof}
(i) \   
Write
\eqn
R_{C_i}(C_i)&&=c_i\; \id_{C_i \otimes C_i},\\
R_{C_j}(C_i) \circ R_{C_i}(C_j)&& = c_{ij}\;\id_{C_i \otimes C_j}
\eneqn
with $c_i,c_{ij}\in\corp^\times$.
Then we can check easily that
$c_{ij}\;\id_{C_i \otimes C_j}=c_{ji}\;\id_{C_i \otimes C_j}$.
Taking a total order $\prec$ on $I$,
we obtain the assertion by setting  
\eqn
\eta_{ij}=
\bc c_i&\text{if $i=j$,}\\
c_{ij}&\text{if $i\prec j$,}\\
1&\text{if $j\prec i$.}
\ec
\eneqn

\smallskip
\noi
(ii) is obvious.
\end{proof}

Let $( C_i ,  R_{C_i} )_{i\in I}$
be a real commuting family of braiders in $\catT$.
Thanks to Lemma \ref{Lem: cf} and Lemma \ref{Lem: rcf - cf},
we can find
\begin{enumerate}[\rm(a)]
\item a braider $(C^\alpha, R_{C^\alpha})$ for each $\alpha \in \lG_{\ge0}$,
\item an isomorphism $\xi_{\alpha, \beta}\cl  C^\alpha \otimes C^\beta \buildrel \sim\over \longrightarrow C^{\alpha+\beta}$ for $\alpha, \beta \in \lG_{\ge0}$
\end{enumerate}
satisfying the conditions (i), (ii) and (iii) of Lemma \ref{Lem: cf}.
Here the right diagram in \eqref{Eq: CF ij}
reads as:
for $i,j\in I$,
 $$
\xymatrix{
C^{e_i} \otimes C^{e_j} \ar[d]_{ \xi_{e_i, e_j}} \ar[rr]^{R_{C^{e_i}}(C^{e_j}) } && \ar[d]^{\xi_{ e_j, e_i}} C^{e_j} \otimes C^{e_i} \\
C^{e_i + e_j} \ar[rr]^{\bce_{i,j} \id_{C^{e_i+e_j}}} && C^{e_i+e_j}
}
$$
commutes for $\bce_{i,j} \in \bR^\times  $ as in  Lemma~\ref{Lem: rcf - cf}~(i).
Note that $\xi_{\alpha, \beta}$ is a morphism of braiders, i.e.
the diagram
\begin{equation} \label{Eq: R alpha}
\begin{aligned}
\xymatrix{
C^{ \alpha} \otimes C^\beta \otimes X \ar[d]_{ \xi_{\alpha, \beta} \otimes X}    \ar[rr]^{ C^\alpha \otimes  R_{C^\beta}(X) } && C^{ \alpha} \otimes X \otimes C^\beta  \ar[rr]^{  R_{C^\alpha}(X) \otimes C^\beta } &&
\ar[d]^{X \otimes \xi_{ \alpha, \beta} } X \otimes C^{ \alpha} \otimes C^\beta \\
C^{\alpha+ \beta }\otimes X \ar[rrrr]^{ R_{C^{\alpha+\beta}}(X)  } &&  && X \otimes C^{\alpha+\beta}.
}
\end{aligned}
\end{equation}
commutes for $X \in \catT$.

We define
\begin{align} \label{Eq: eta}
\bce(\alpha, \beta) \seteq  \prod_{i,j \in I} \bce_{i,j}^{a_ib_j} \in \bR^\times
\end{align}
for $\alpha = \sum_{i\in I} a_i e_i$ and $\beta = \sum_{j\in I} b_j e_j$ in $\lG$.
Note that, by the definition, $\bce(\alpha,0) = \bce(0,\alpha)=1$ and
\begin{align} \label{Eq: comm for bce}
\bce(\alpha, \beta+\gamma) = \bce(\alpha, \beta)\cdot \bce(\alpha, \gamma)\ \  \text{ and } \ \  \bce(\alpha+ \beta, \gamma) = \bce(\alpha, \gamma)\cdot \bce(\beta, \gamma)
\end{align}
for $\alpha, \beta, \gamma \in \lG$.

\begin{lem} \label{Lem: mu scalar}
For $\alpha$, $\beta \in \lG_{\ge0}$, we have a commutative diagram:
$$
\xymatrix{
C^\alpha \otimes C^\beta \ar[d]_{ \xi_{\alpha, \beta}} \ar[rr]^{R_{C^\alpha}(C^\beta) } && \ar[d]^{\xi_{ \beta, \alpha}} C^\beta \otimes C^\alpha \\
C^{\alpha + \beta} \ar[rr]^{\bce(\alpha,\beta)\; \id_{C^{\alpha+\beta}} } && C^{\alpha+\beta}
}
$$
\end{lem}
\begin{proof}
For $\al=\sum_{i\in I} a_i e_i$, set $\height{\al}=\sum_{i\in I}a_i$.
We shall show the assertion by induction on $\height{\al}$.
It is trivial when either $\al=0$ or $\beta=0$. 

First we assume $\height{\al}>1$.
Then write $\al=\al'+\al''$ with $\al',\al''\not=0$.
Setting
$C^{\beta_1,\ldots,\beta_n}=C^{\beta_1}\tens\cdots \tens C^{\beta_n}$,
the assertion follows from 
the following commutative diagram:

\scalebox{.9}{\parbox{\textwidth}{ 
\eqn
\xymatrix@C=3ex{
&C^{\al',\al'',\beta}\ar[rrr]^{R_{C^{\alpha''}}(C^\beta)} \ar[dl]\ar[d]&&&
C^{\al',\beta,\al''}\ar[rrr] ^{R_{C^{\alpha'}}(C^\beta)} \ar[dl]\ar[dr]&&&
C^{\beta,\al',\al''}\ar[d]\ar[dr]\\
C^{\al,\beta}\ar[dr]
&C^{\al',\al''+\beta}\ar[d]
\ar[rr]^{\bce(\al'',\beta)}&&
C^{\al',\al''+\beta}\ar[dr]&
&C^{\al'+\beta,\al''}\ar[dl]\ar[rr]^{\bce(\al',\beta)}
&&C^{\al'+\beta,\al''}\ar[d]& C^{\beta,\al}\ar[dl]\\
&C^{\al+\beta}\ar[rrr]^{\bce(\al'',\beta)}&&&C^{\al+\beta}\ar[rrr]^{\bce(\al',\beta)}&&&C^{\al+\beta}}
\eneqn
}
}
\vskip 0.5em 

We next assume $\height{\al}=1$.  If $\height{\beta}=1$, then it's done by the definition of $\eta_{i,j}$.
When $\height{\beta}>1$, the assertion follows by arguing similarly by induction on $\height{\beta}$.
\end{proof}

\subsection{Localizations} \label{Sec: Loc}  \

Let $( C_i , R_{C_i} )_{i\in I}$ be a real commuting family of braiders in $\catT$.
We define an order $\preceq$ on $ \lG$ by
$$
\alpha \preceq \beta  \quad \text{ for } \alpha, \beta \in \lG \text{ with }  \beta - \alpha \in \lG_{\ge0}
$$
Note that $\preceq$ is \emph{directed}, i.e., for $\alpha, \beta \in \lG$, there exists $\gamma\in \lG$ such that $\alpha \preceq \gamma$ and $\beta \preceq \gamma$.
For $\alpha_1, \ldots, \alpha_k \in \lG$,  we set
$$
\Ds_{\alpha_1, \ldots, \alpha_k} \seteq  \{ \delta \in \lG \mid \alpha_i + \delta \in \lG_{\ge0} \text{ for any }i =1, \ldots, k \}.
$$

For $X,Y \in \catT$ and  $\delta, \delta' \in \Ds_{\alpha, \beta}$ with $\delta \preceq \delta'$, we define the map
\begin{align} \label{Eq: zeta}
\zeta_{\delta', \delta}  \cl  \Hom_{\catT}( C^{\delta+\alpha}\otimes X ,  Y  \otimes C^{ \delta+\beta} ) \rightarrow \Hom_{\catT}( C^{ \delta'+\alpha}\otimes X ,  Y \otimes C^{ \delta'+\beta} )
\end{align}
by sending $f$ to $\zeta_{\delta', \delta }(f)$ such that the following diagram commutes:
$$
\xymatrix{
C^{\delta' - \delta} \otimes C^{ \delta+\alpha} \otimes X   \ar[d]^{\wr}_{  \xi_{\delta'-\delta,  \delta+\alpha} }  \ar[rr]^{ C^{\delta'-\delta} \otimes f }   &&      C^{\delta' - \delta} \otimes Y \otimes C^{ \delta+\beta}   \ar[rr]^{  R_{C^{\delta'-\delta}} (Y) }  &&  Y \otimes C^{\delta' - \delta} \otimes C^{ \delta+\beta}
  \ar[d]_{\wr}^{\xi_{ \delta'-\delta,   \delta+\beta}}  \\
C^{  \delta'+\alpha} \otimes X  \ar[rrrr]^{\zeta_{\delta', \delta}(f)} &&&& Y \otimes C^{ \delta'+\beta}
}
$$
For $\delta, \delta', \delta'' \in \Ds_{\alpha, \beta} $ with $\delta \preceq \delta' \preceq \delta''$,
it follows from $\eqref{Eq: xi sum}$ and $\eqref{Eq: R alpha}$ that the following diagram
$$
\xymatrix{
C^{\delta''+\alpha} \otimes X  \ar[ddd]_{\zeta_{\delta'', \delta}(f)}  && \ar[ll]_{\sim}^{\xi_{\delta''-\delta', \delta'+\alpha} \quad } C^{\delta''-\delta'} \otimes C^{\delta'+\alpha} \otimes X \ar[dd]_{\zeta_{\delta',\delta }(f)}
&&  \ar[ll]_{\sim}^{\xi_{\delta' - \delta, \delta + \alpha} \quad } C^{\delta''-\delta'} \otimes C^{\delta'-\delta} \otimes C^{\delta+\alpha} \otimes X  \ar[d]_{f}  \\
&&  && C^{\delta''-\delta'} \otimes C^{\delta'-\delta} \otimes Y \otimes C^{\delta+\beta} \ar[d]_{R_{C^{\delta'-\delta}}(Y)} \\
 &&  C^{\delta''-\delta'} \otimes Y \otimes C^{\delta'+\beta}  \ar[d]_{R_{C^{\delta''-\delta'}}(Y)}  &&  \ar[ll]_{\sim}^{\xi_{\delta' - \delta, \delta + \beta} \quad }  C^{\delta''-\delta'} \otimes Y \otimes C^{\delta'-\delta}  \otimes C^{\delta+\beta} \ar[d]_{R_{C^{\delta''-\delta'}}(Y)}  \\
Y \otimes C^{\delta''+\beta}  && \ar[ll]_{\sim}^{\xi_{\delta'' - \delta', \delta' + \beta} \quad } Y \otimes C^{\delta''-\delta'} \otimes   C^{\delta'+\beta}  &&  \ar[ll]_{\sim}^{\xi_{\delta' - \delta, \delta + \beta} \quad } Y \otimes C^{\delta''-\delta'} \otimes  C^{\delta'-\delta}  \otimes C^{\delta+\beta}
}
$$
commutes, which implies that
$$
\zeta_{\delta'', \delta' } \circ \zeta_{\delta', \delta} = \zeta_{\delta'', \delta}.
$$
Thus, it becomes an inductive system (or direct system) on $ \Ds_{\alpha, \beta}$.

We shall define a localization $\lT$ of $\catT$ by $( C_i ,  R_{C_i} )_{i\in I}$ as follows.
We set
\begin{align*}
\Ob (\lT) &\seteq  \Ob(\catT) \times \lG, \\
\Hom_{\lT}( (X, \alpha), (Y, \beta) ) &\seteq   \lim_{  \substack{\longrightarrow \\ \delta \in \Ds_{\alpha, \beta}}    }  \Hm_\delta( (X, \alpha  ), (Y, \beta  ) ) ,
\end{align*}
where
$$
\Hm_\delta( (X, \alpha  ), (Y, \beta  ) ) \seteq  \Hom_\catT( C^{ \delta+\alpha}\otimes X,  Y \otimes C^{  \delta+\beta} ).
$$
Here, the inductive limit is taken over the inductive system on $\Ds_{\alpha, \beta}$ defined as above.

\bigskip
For  $f \in \Hm_\delta( (X, \alpha  ), (Y, \beta  ) ) $ and  $ g \in  \Hm_\epsilon( (Y, \beta  ), (Z, \gamma  ) )$, we define
$$
\Psi_{\delta, \epsilon} (f,g)  \seteq  \bce(\delta+ \beta, \beta-\gamma) \cdot \widetilde{\Psi}_{\delta, \epsilon}(f,g) \in \Hm_{\delta+\epsilon+\beta}( (X, \alpha  ), (Z, \gamma ) ),
$$
where $\bce$ is given in $\eqref{Eq: eta}$ and  $\widetilde{\Psi}_{\delta, \epsilon}(f,g)$ is the morphism such that the following diagram commutes:
$$
\xymatrix{
C^{\epsilon+\beta}\otimes C^{\delta+\alpha}\otimes X \ar[d]^{\wr}_{\xi_{\epsilon+\beta, \delta+\alpha} }  \ar[rr]^f &&  C^{\epsilon+\beta}\otimes   Y\otimes  C^{\delta+\beta} \ar[rr]^g && Z\otimes C^{\epsilon+\gamma}\otimes C^{\delta+\beta} \ar[d]_{\wr}^{\xi_{ \epsilon+\gamma, \delta+\beta}} \\
C^{\delta+\epsilon+\beta+\alpha} \otimes X  \ar[rrrr]^{\widetilde{\Psi}_{\delta, \epsilon}(f,g)}&& && Z \otimes  C^{\delta+\epsilon+\beta+\gamma} .
}
$$
We set
$$
C^{\alpha_1, \ldots, \alpha_t} \seteq  C^{\alpha_1} \otimes \cdots \otimes C^{\alpha_t}
 $$
 for $\alpha_1, \ldots \alpha_t \in \lG_{\ge0}$.
Then, for $\delta' \succeq \delta$ and $\epsilon' \succeq \epsilon$, one can show that the following diagram commutes:
$$
\xymatrix{
C^{\delta' + \epsilon' +\beta+ \alpha } \otimes X \ar[rr]^{\bce(\delta'-\delta, \epsilon+\beta)}   \ar@/_7pc/[ddd]_{p}
 && C^{\delta' + \epsilon' +\beta+ \alpha} \otimes X    \ar@/^7pc/[ddd]^{q} \\
C^{\epsilon'-\epsilon,\delta'-\delta,\epsilon+\beta,\delta+\alpha} \otimes X \ar[rr]^{R_{C^{\delta'-\delta}} (C^{\epsilon+\beta}) } \ar[d]_{ a} \ar[u]^\wr
 && C^{\epsilon'-\epsilon,\epsilon+\beta,\delta'-\delta,\delta+\alpha} \otimes X \ar[d]^{ b}  \ar[u]_\wr \\
Z \otimes C^{\epsilon'-\epsilon,\delta'-\delta,\epsilon+\gamma,\delta+\beta} \ar[rr]^{R_{C^{\delta'-\delta}} (C^{\epsilon+\gamma}) }  \ar[d]_\wr && Z \otimes  C^{\epsilon'-\epsilon,\epsilon+\gamma,\delta'-\delta,\delta+\beta}  \ar[d]^\wr \\
Z \otimes  C^{\delta' + \epsilon' + \beta+\gamma} \ar[rr]^{\bce(\delta'-\delta, \epsilon+\gamma)} && Z \otimes  C^{\delta' + \epsilon' + \beta+\gamma},
}
$$
where
\begin{align*}
&a = R_{C^{\epsilon'-\epsilon, \delta'-\delta}}(Z)  \circ g \circ f, \qquad \  b = R_{C^{\epsilon' - \epsilon}}(Z) \circ g \circ R_{C^{\delta'-\delta}}(Y)\circ f, \\
&p = \zeta_{  \delta' + \epsilon' + \beta, \delta+\epsilon+\beta}(\widetilde{\Psi}_{\delta, \epsilon}(f,g)), \ \ \
q= \widetilde{\Psi}_{ \delta', \epsilon'}( \zeta_{\delta', \delta}(f), \zeta_{\epsilon', \epsilon}(g) ),
\end{align*}
 and other vertical morphisms are given by $\xi$'s. Thus, by $\eqref{Eq: comm for bce}$, we have
\begin{align*}
\Psi_{ \delta', \epsilon'}( \zeta_{\delta', \delta}(f), \zeta_{\epsilon', \epsilon}(g) ) &= \eta(\delta'+\beta, \beta-\gamma) \cdot
\widetilde{\Psi}_{\delta', \epsilon'}( \zeta_{\delta', \delta}(f), \zeta_{\epsilon', \epsilon}(g) ) \\
&= \eta(\delta+\beta, \beta-\gamma)  \cdot \zeta_{  \delta' + \epsilon' + \beta, \delta+\epsilon+\beta }(\widetilde{\Psi}_{\delta, \epsilon}(f,g)) \\
&=  \zeta_{  \delta' + \epsilon' + \beta, \delta+\epsilon+\beta }(\Psi_{\delta, \epsilon}(f,g)),
\end{align*}
which tells that the following diagram commutes:
$$
\xymatrix{
\Hm_\delta( (X, \alpha  ), (Y, \beta  ) ) \times H_\epsilon( (Y, \beta  ), (Z, \gamma  ) )  \ar[r] \ar[d] & \ar[d]
\Hm_{\delta+\epsilon+\beta}( (X, \alpha  ), (Z, \gamma  ) )\\
\Hm_{\delta'}( (X, \alpha  ), (Y, \beta  ) ) \times H_{\epsilon'}( (Y, \beta  ), (Z, \gamma  ) )  \ar[r] &
\Hm_{\delta'+\epsilon'+\beta}( (X, \alpha  ), (Z, \gamma  ) ),
}
$$
where the horizontal maps are given by $\Psi$'s and the vertical maps are given by $\zeta$'s.
This yields the composition
$$
\Hom_{\lT}( (X, \alpha), (Y, \beta) )  \times \Hom_{\lT}( (Y, \beta), (Z, \gamma) )  \rightarrow
\Hom_{\lT}( (X, \alpha), (Z, \gamma) )
$$
in $\lT$.
We can check also the associativity, and
we conclude that  $\lT$ is a category. 
Note that the identity $\id \in \End_{\lT}\bl(X, \alpha)\br$ is induced from the identity 
$$\id_X \in \Hom_\catT(X,X)=\Hm_{- \alpha} ( (X, \alpha), (X, \alpha) ).$$

\smallskip \smallskip \smallskip
We shall define a bifunctor $\otimes$ on the category $\lT$.
For $\alpha, \beta \in \lG$ and $X,Y \in \catT$, we set
$$
(X, \alpha) \otimes (Y, \beta) \seteq  ( X \otimes Y, \alpha+ \beta ).
$$
To define a tensor of morphisms, for $f \in \Hm_\delta((X, \alpha), (X', \alpha'))$ and $g \in \Hm_\epsilon((Y, \beta), (Y', \beta'))$ we define
$$
T_{\delta, \epsilon}(f,g) \seteq  \eta(\epsilon, \alpha-\alpha') \widetilde{T}_{\delta, \epsilon}(f,g) \in \Hm_{\delta+\epsilon}((X\otimes Y, \alpha+\beta), (X' \otimes Y', \alpha'+\beta')),
$$
where $ \widetilde{T}_{\delta, \epsilon}(f,g)$ is the morphism such that the following diagram commutes:
$$
\xymatrix{
C^{\delta+\alpha}  \otimes X \otimes C^{\epsilon+\beta}\otimes  Y  \ar[rr]^{f \otimes g} &&  X' \otimes C^{\delta+\alpha'}  \otimes  Y' \otimes  C^{\epsilon+\beta'}   \ar[d]^{R_{C^{\delta+\alpha'}} (Y') } \\
C^{\delta+\alpha} \otimes C^{\epsilon+\beta}\otimes X \otimes Y  \ar[u]^{R_{C^{\epsilon+\beta}} (X) }  \ar[d]^\wr_{\xi_{\delta+\alpha, \epsilon+\beta}}
&&  X' \otimes Y' \otimes C^{\delta+\alpha'} \otimes C^{\epsilon+\beta'}  \ar[d]_\wr^{\xi_{\delta+\alpha', \epsilon+\beta'}} \\
C^{\delta+\epsilon+\alpha+\beta}\otimes X \otimes Y \ar[rr]^{\widetilde{T}_{\delta,\epsilon} (f, g)}  &&  X' \otimes Y' \otimes C^{\delta+\epsilon+\alpha'+\beta'}.
}
$$
Then, for $\delta' \succeq \delta$ and $\epsilon' \succeq \epsilon$, one can show that
$$
\zeta_{\delta' + \epsilon, \delta+\epsilon}(  \widetilde{T}_{\delta, \epsilon} (f,g)) = \widetilde{T}_{\delta', \epsilon}( \zeta_{\delta', \delta}(f), g),
$$
and the following diagram commutes:
$$
\xymatrix{
C^{\delta + \epsilon' + \alpha+\beta } \otimes X \otimes Y \ar[rr]^{\bce(\epsilon'-\epsilon, \delta+\alpha)}   \ar@/_7pc/[ddd]_{r}
 && C^{\delta + \epsilon' + \alpha+\beta} \otimes X \otimes Y   \ar@/^7pc/[ddd]^{s} \\
C^{\epsilon'-\epsilon,\delta+\alpha,\epsilon+\beta} \otimes X\otimes Y  \ar[rr]^{R_{C^{\epsilon'-\epsilon}} (C^{\delta+\alpha}) } \ar[d]_{ c} \ar[u]^\wr
 && C^{\delta+\alpha,\epsilon'-\epsilon,\epsilon+\beta} \otimes X \otimes Y \ar[d]^{ d} \ar[u]_\wr \\
X'\otimes Y' \otimes C^{\epsilon'-\epsilon,\delta+\alpha',\epsilon+\beta'} \ar[rr]^{R_{C^{\epsilon'-\epsilon}} (C^{\delta+\alpha'}) }  \ar[d]_\wr && X'\otimes Y' \otimes C^{\delta+\alpha',\epsilon'-\epsilon,\epsilon+\beta'}   \ar[d]^\wr \\
X' \otimes Y' \otimes  C^{\delta + \epsilon' + \alpha' + \beta'} \ar[rr]^{\bce(\epsilon'-\epsilon, \delta+\alpha')} && X'\otimes Y' \otimes  C^{\delta + \epsilon' + \alpha' + \beta'},
}
$$
where
\begin{align*}
&c = R_{C^{\epsilon'-\epsilon, \delta+\alpha'}}(Y') \circ R_{C^{\epsilon'-\epsilon}}(X') \circ (f \otimes g) \circ R_{C^{\epsilon+\beta}}(X), \\
&d = R_{C^{\delta+\alpha', \epsilon'-\epsilon}}(Y')  \circ (f \otimes g) \circ R_{C^{\epsilon'-\epsilon, \epsilon+\beta}}(X), \\
&r = \zeta_{\delta+\epsilon', \delta + \epsilon}(\widetilde{T}_{\delta, \epsilon}(f,g)), \ \ \
s= \widetilde{T}_{ \delta, \epsilon'}( f, \zeta_{\epsilon', \epsilon}(g) ),
\end{align*}
 and other vertical morphisms are given by $\xi$'s. Thus, by $\eqref{Eq: comm for bce}$, we have
\begin{align*}
T_{ \delta', \epsilon'} & ( \zeta_{\delta', \delta}(f), \zeta_{\epsilon', \epsilon}(g) ) \\ 
& = \eta(\epsilon', \alpha-\alpha') \cdot
\widetilde{T}_{\delta', \epsilon'}( \zeta_{\delta', \delta}(f), \zeta_{\epsilon', \epsilon}(g) ) \\
&= 
 \eta(\epsilon', \alpha-\alpha')  \eta(\epsilon'-\epsilon, \delta+\alpha)^{-1}  \eta(\epsilon'-\epsilon, \delta+\alpha')
 \cdot \zeta_{\delta'+\epsilon', \delta  + \epsilon  }(\widetilde{T}_{\delta, \epsilon}(f,g)) \\
&=  \zeta_{\delta'+\epsilon', \delta  + \epsilon  }(T_{\delta, \epsilon}(f,g)).
\end{align*}

Therefore,
for $f \in  \Hom_{\lT}( (X, \alpha)  , (X', \alpha')  )$ and $g \in  \Hom_{\lT}( (Y, \beta)  , (Y', \beta')  )$,
we have the tensor product
\begin{align} \label{Eq: tensor for localization}
f \otimes g \in \Hom_{\lT}( (X, \alpha) \otimes (Y, \beta) , (X', \alpha') \otimes (Y', \beta') )
\end{align}
induced from $T_{\delta, \epsilon}$ over the inductive system.

\begin{prop} \label{Prop: tensor and composition}
Let $f_k \in \Hm_{\delta_k}( (X_k, \alpha_k), (Y_k, \beta_k) )  $ and $g_k \in \Hm_{\epsilon_k}( (Y_k, \beta_k), (Z_k, \gamma_k) ) $ for $k=1,2$.
Then, we have
$$
\Psi_{\delta_1+\delta_2, \epsilon_1 + \epsilon_2} (T_{\delta_1, \delta_2} (f_1, f_2),  T_{\epsilon_1, \epsilon_2} (g_1, g_2) )
= T_{\delta_1+\epsilon_1+\beta_1, \delta_2+\epsilon_2 + \beta_2 } (\Psi_{\delta_1, \epsilon_1} (f_1,g_1), \Psi_{\delta_2, \epsilon_2} (f_2,g_2) ).
$$
\end{prop}
\begin{proof}
By Lemma \ref{Lem: mu scalar}, we have the commutative diagram
$$
\xymatrix{
C^{\epsilon_2 + \beta_2} \otimes C^{\delta_1 + \beta_1} \otimes Y_2 \ar[rrr]^{R_{C^{\epsilon_2 + \beta_2} } ( C^{\delta_1 + \beta_1} ) }  \ar[rrrd]_{ \rho \cdot R_{C^{\delta_1 + \beta_1}}  (Y_2 )\ \  }
&&& C^{\delta_1 + \beta_1} \otimes C^{\epsilon_2 + \beta_2} \otimes Y_2  \ar[d]^{R_{C^{\delta_1 + \beta_1}}  ( C^{\epsilon_2 + \beta_2} \otimes Y_2  )}  \\
&&&  C^{\epsilon_2 + \beta_2} \otimes Y_2 \otimes C^{\delta_1 + \beta_1}   ,
}
$$
where
$
\rho = \bce(\delta_1 + \beta_1, \epsilon_2 + \beta_2)\bce(\epsilon_2 + \beta_2, \delta_1 + \beta_1).
$
Then, one can prove that the following diagram commutes.
$$ \small
\xymatrix{
C^{\epsilon_1 + \beta_1 , \epsilon_2+\beta_2, \delta_1+\alpha_1, \delta_2+\alpha_2}\otimes X_1 \otimes X_2 \ar[r]^{a}  \ar[d]_{p} &
C^{\epsilon_1 + \beta_1 ,  \delta_1+\alpha_1, \epsilon_2+\beta_2, \delta_2+\alpha_2}\otimes X_1 \otimes X_2  \ar[d]^{q} \\
C^{\epsilon_1+\beta_1} \otimes Y_1 \otimes C^{\epsilon_2 + \beta_2, \delta_1 + \beta_1} \otimes Y_2 \otimes C^{\delta_2 + \beta_2}  \ar[r]^{b} \ar[d]_{\rho \cdot R_{C^{\delta_1 + \beta_1}}  (Y_2 )} &
C^{\epsilon_1+\beta_1} \otimes Y_1 \otimes C^{\delta_1 + \beta_1, \epsilon_2 + \beta_2} \otimes Y_2 \otimes C^{\delta_2 + \beta_2}  \ar[dl]^{\qquad \qquad  R_{C^{\delta_1 + \beta_1}}  ( C^{\epsilon_2 + \beta_2} \otimes Y_2  )}  \ar[dd]^{g_1 \otimes g_2} \\
C^{\epsilon_1+\beta_1} \otimes Y_1 \otimes C^{\epsilon_2 + \beta_2} \otimes Y_2 \otimes C^{\delta_1 + \beta_1, \delta_2 + \beta_2}    \ar[d]_{g_1 \otimes g_2} &   \\
Z_1\otimes C^{\epsilon_1+\gamma_1} \otimes Z_2 \otimes C^{\epsilon_2 + \gamma_2, \delta_1 + \beta_1, \delta_2 + \beta_2} \ar[d]_{R_{C^{\epsilon_1 + \gamma_1}} (Z_2)} & \ar[l]_{c}
Z_1\otimes C^{\epsilon_1+\gamma_1, \delta_1+\beta_1} \otimes Z_2 \otimes C^{\epsilon_2 + \gamma_2,  \delta_2 + \beta_2}  \ar[d]^{R_{C^{\epsilon_1 + \gamma_1, \delta_1 + \beta_1}} (Z_2)}   \\
Z_1\otimes Z_2 \otimes C^{\epsilon_1+\gamma_1, \epsilon_2 + \gamma_2, \delta_1 + \beta_1, \delta_2 + \beta_2}  & \ar[l]_{d}
Z_1\otimes Z_2 \otimes C^{\epsilon_1+\gamma_1, \delta_1+\beta_1,\epsilon_2 + \gamma_2,  \delta_2 + \beta_2},
}
$$
where
\begin{align*}
&a = R_{C^{\epsilon_2+\beta_2}}(C^{\delta_1+\alpha_1}) , \qquad \qquad\qquad\qquad\qquad \
b = R_{C^{\epsilon_2+\beta_2}}(C^{\delta_1+\beta_1}) , \\
&c = R_{C^{\delta_1+\beta_1} }(Z_2 \otimes C^{\epsilon_2+\gamma_2}), \qquad\qquad\qquad\qquad
d = R_{C^{\delta_1+\beta_1}}(C^{\epsilon_2+\gamma_2}) , \\
& p = R_{C^{\epsilon_2 + \beta_2}}(Y_1)\circ (f_1 \otimes f_2) \circ R_{C^{\delta_2 + \alpha_2}}(X_1), \quad
q = (f_1\otimes f_2) \circ R_{C^{\epsilon_2 + \beta_2, \delta_2+\alpha_2}}(X_1).
\end{align*}
We simply write
\begin{align*}
\widetilde{\Psi} ( \widetilde{T} ,  \widetilde{T}  ) &\seteq  \widetilde{\Psi}_{\delta_1+\delta_2, \epsilon_1 + \epsilon_2} (\widetilde{T}_{\delta_1, \delta_2} (f_1, f_2),  \widetilde{T}_{\epsilon_1, \epsilon_2} (g_1, g_2) ) ,\\
\widetilde{T}( \widetilde{\Psi} , \widetilde{\Psi}  ) &\seteq  \widetilde{T}_{\delta_1+\epsilon_1+\beta_1, \delta_2+\epsilon_2 + \beta_2 } (\widetilde{\Psi}_{\delta_1, \epsilon_1} (f_1,g_1), \widetilde{\Psi}_{\delta_2, \epsilon_2} (f_2,g_2) ), \\
\widetilde{C} &\seteq   C^{\delta_1 + \delta_2 + \epsilon_1+\epsilon_2+\beta_1+\beta_2+\alpha_1+\alpha_2}.
\end{align*}
Then, it follows from the above diagram that the following diagram commutes:
\begin{equation} \label{Eq: comm for tensor n composition}
\begin{aligned}
\xymatrix{
\widetilde{C} \otimes X_1 \otimes X_2  \ar[rrr]^{\bce(\epsilon_2 + \beta_2, \delta_1 + \alpha_1)}  \ar[d]_{ \rho \cdot \widetilde{\Psi} (  \widetilde{T} ,  \widetilde{T}  ) }   &&&
\widetilde{C}\otimes X_1 \otimes X_2  \ar[d]^{\widetilde{T}( \widetilde{\Psi} , \widetilde{\Psi}  )  }    \\
 Z_1 \otimes Z_2 \otimes \widetilde{C}   &&& \ar[lll]_{\bce( \delta_1 + \beta_1, \epsilon_2 + \gamma_2)}
Z_1 \otimes Z_2 \otimes \widetilde{C}.
}
\end{aligned}
\end{equation}
Since
\begin{align*}
&\Psi_{\delta_1+\delta_2, \epsilon_1 + \epsilon_2}  (T_{\delta_1, \delta_2} (f_1, f_2),  T_{\epsilon_1, \epsilon_2} (g_1, g_2) )  \\
& \ \  =  \bce(\delta_1 + \delta_2 + \beta_1 + \beta_2, \beta_1 + \beta_2 - \gamma_1 -\gamma_2 ) \bce(\delta_2, \alpha_1 - \beta_1) \bce(\epsilon_2, \beta_1 - \gamma_1)
\widetilde{\Psi} ( \widetilde{T} ,  \widetilde{T}  ) , \\
&T_{\delta_1+\epsilon_1+\beta_1, \delta_2+\epsilon_2 + \beta_2 } (\Psi_{\delta_1, \epsilon_1} (f_1,g_1), \Psi_{\delta_2, \epsilon_2} (f_2,g_2) )   \\
& \ \  =  \bce(\delta_2 + \epsilon_2 + \beta_2 , \alpha_1 - \gamma_1 ) \bce(\delta_1 + \beta_1, \beta_1 - \gamma_1) \bce(\delta_2+\beta_2, \beta_2 - \gamma_2)
\widetilde{T}( \widetilde{\Psi} , \widetilde{\Psi}  ),
\end{align*}
the assertion follows from $\eqref{Eq: comm for bce}$ and $\eqref{Eq: comm for tensor n composition}$.
\end{proof}

By Proposition \ref{Prop: tensor and composition}, we conclude that the tensor product $\eqref{Eq: tensor for localization}$ is a bifunctor
\begin{align*} 
 \Hom_{\lT}(   (X, \alpha) , (X', \alpha') ) \times  & \Hom_{\lT}(   (Y, \beta) , (Y', \beta') ) \\ 
& \longrightarrow
 \Hom_{\lT}(   (X, \alpha) \otimes (Y, \beta) , (X', \alpha') \otimes (Y', \beta') ).
\end{align*}
Moreover, one can check that  the composition of tensor products
of morphisms
satisfies the associativity, 
which tells that $\lT$ is a monoidal category.
Note that $(\triv, 0)$ is a unit of $\lT$.

Let us define the functor $\Upsilon\cl\catT\to \lT $  by
$X\mapsto (X,0)$. It is easy to check that $\Upsilon$ is a monoidal functor.

\Prop \label{prop: localization}
Let $( C_i ,   R_{C_i} )_{i\in I}$ be a real commuting family  of braiders in a monoidal category $\catT$.
Then there is a real commuting family  of braiders
$(\widetilde{C}_i    ,  R_{\widetilde{C}_i} )_{i\in I} $ in $\lT$
satisfy the following properties:
\bnum
\item
for $i\in I$, $\Upsilon(C_i) $ is isomorphic to $ \widetilde{C}_i$ and it is invertible in $(\lT)_{\,\mathrm{br}}$,
\item \label{Eq: loc 3}
for $i\in I$ and $X\in\catT$, the diagram
$$
\xymatrix{
\Upsilon(C_i \otimes X)  \ar[r]^\sim  \ar[d]_{\Upsilon( R_{C_i} (X)  )}  & \widetilde{C}_i \otimes \Upsilon(X) \ar[d]^{ R_{\widetilde{C}_i} (\Upsilon(X)  )}  \\
\Upsilon( X \otimes  C_i )  \ar[r]^\sim &  \Upsilon(X)\otimes  \widetilde{C}_i
}
$$
commutes.
\end{enumerate}
\enprop

\Proof

For $\al\in\lG$, set $\tC^\al=(\one,\al)\in  \lT $.
For $Z=(X,\beta)\in \lT$, 
define $R_{\tC^\al}(Z)\cl \tC^\al\tens Z\to Z\tens \tC^\al$ as the image of
$\id_X\in \Hom_{\catT}(X, X)
=\Hm_{-\al-\beta}(\tC^\al\tens Z,Z\tens\tC^\al)$ by the map
$\Hm_{-\al-\beta}(\tC^\al\tens Z,Z\tens\tC^\al)\to 
\Hom_{\lT}(\tC^\al\tens Z,Z\tens\tC^\al)$.
We can easily check that $(\tC^\al,R_{\tC^\al })$ is a 
central braider and invertible in $({\lT})_\braid$.

\smallskip
Let us show that $\Upsilon(C^\al)$ is isomorphic to
$\tC^{\al}$ for $\al\in\lG_{\ge0}$.

Define $f\cl(C^\al,0)\to \tC^{\al}=(\one,\al)$ by
$\id_{C^\al}\in\Hom_{\catT}(C^\al,C^{\al}) = \Hm_{0}\bl(C^\al,0), (\one,\al)\br$, and 
 $g\cl\tC^{\al}=(\one,\al)\to(C^\al,0)$ by
$\id_{C^\al}\in\Hom_{\catT}(C^{\al},C^\al)=\Hm_{0}\bl(\one,\al), (C^\al,0)\br$.
Then one can see easily that $f$ and $g$ are inverse to each other.

Hence (i) and (ii) follow.
\QED

\begin{thm} \label{Thm: localization}
Let $( C_i ,   R_{C_i} )_{i\in I}$ be a real commuting family  of braiders in a monoidal category $\catT$.
There exist a monoidal category $ \lT $ and a monoidal functor
$\Upsilon\cl \catT \rightarrow \lT$
such that
\bnum
\item  \label{Eq: loc 1}
$\Upsilon(C_i) $ is invertible in $\lT$ for any $i\in I$.
\item 
for any $i\in I$ and $X\in\catT$, $\Upsilon(R_{C_i}(X))\cl
\Upsilon(C_i\tens X)\to\Upsilon(X\tens C_i)$ is an isomorphism.
\label{Eq: loc 2}\setcounter{myc}{\value{enumi}}
\ee
Moreover $\Upsilon$ satisfies the following universal property.
\bnum\setcounter{enumi}{\value{myc}}
\item If there are another monoidal category $\catT'$ and a monoidal functor $\Upsilon'\cl  \catT \rightarrow \catT'$ 
such that  the properties \eqref{Eq: loc 1} and \eqref{Eq: loc 3} 
above hold for $\Upsilon'$, then there exists a monoidal functor $F$, which is unique up to a unique isomorphism,  such that
the diagram
$$
\xymatrix{
\catT \ar[r]^{\Upsilon} \ar[dr]_{\Upsilon'}  & \lT \ar@{.>}[d]^F \\
& \catT'
}
$$
commutes.
\end{enumerate}
\end{thm}
We denote by $ \catT[ C_i^{\otimes -1} \mid i \in I ]$ the localization $\lT$ given in Theorem \ref{Thm: localization}.

\Proof
Let $\lT$ and $\Upsilon\cl \catT\to\lT$ be the monoidal category and the functor 
that we have constructed.
Then it satisfies (i) and (ii) by Proposition~\ref{prop: localization}.

Let us show (iii).
Set $\tC_i=\Upsilon'(C_i)$.
Then, it forms a commuting family of invertible objects in $\catT'$, and there exist 
a family $\{\tC^\al\}_{\al\in\lG}$ and isomorphism
$\txi_{\al,\beta}\cl \tC^\al\tens\tC^\beta\isoto \tC^{\al+\beta}$
such that $\Upsilon'(C^\al)\simeq\tC^\al$ for $\al\in\lG_{\ge0}$.

Now let us define the functor $F\cl\lT\to\lT'$.
Set $F\bl(X, \al)\br=\tC^\al\tens\Upsilon'(X)$ for 
$X\in\catT$ and $\al\in\lG$.
Let $(X,\al)$, $(Y,\beta)\in\lT$ and $\delta\in \Ds_{\alpha, \beta}$.
For $f\in \Hm_\delta( (X, \alpha  ), (Y, \beta  ) )=
\Hom_\catT( C^{ \delta+\alpha}\otimes X,  Y \otimes C^{  \delta+\beta} )$,
we define
$F(f)\in \Hom_{\lT'}\bl F\bl(X, \al)\br, F\bl(Y, \beta)\br\br$ 
by the following commutative diagram

\scalebox{.93}{
$$\hs{-3ex}\xymatrix@C=8ex{\tC^\delta\tens\tC^\al\tens\Upsilon(X)
\ar[r]^{\tC^\delta\tens F(f)}\ar[d]_{ \bwr }^-{\txi_{\delta,\al}}&\tC^\delta\tens\tC^\beta\tens\Upsilon(Y)
\ar[r]^-\sim_-{\txi_{\delta,\beta}}&\tC^{\delta+\beta}\tens\Upsilon(Y)
\ar[r]^-{\sim}&\Upsilon(\tC^{\delta+\al}\tens Y)\ar[d]^\bwr_-{\Upsilon(R_{C^{\delta+\beta}}(Y))}\\
\tC^{\delta+\al}\tens\Upsilon(X)\ar[r]^-\sim&
\Upsilon(C^{\delta+\al}\tens X) \ar[rr]^{\Upsilon(f)}&&
\Upsilon(Y\tens C^{\delta+\beta}).}
$$}

\noi
Then it is easy to see that $F$ is a well defined functor
and satisfies the desired properties.
\QED

\Lemma\label{lem:C}
Let $( C_i ,   R_{C_i}  )_{i\in I}$ be a real commuting family  of braiders in a monoidal category $\catT$, and  assume that $I$ is a finite set.
Assume that $\al\in\lG$ satisfies $\al-e_i\in\lG_{\ge0}$ for any $i\in I$.
Then $\catT[ (C^\al)^{\otimes -1}]$ is equivalent to 
$\catT[ C_i^{\otimes -1} \mid i \in I ]$.
\enlemma
\Proof
It is enough to show that $\catT'\seteq\catT[ (C^\al)^{\otimes -1}]$
satisfies properties \eqref{Eq: loc 1}, 
\eqref{Eq: loc 2} in Theorem~\ref{Thm: localization}.
It is obvious that $(C^\al)^{\otimes -1}\tens C^{\al-e_i}$ is an inverse of $C_i$
in $\catT'$, and hence  \eqref{Eq: loc 1} holds.

For any $X\in\catT$,
the morphism $R_{C^\al}(X)=
\bl R_{C^{\al-e_i}}(X)\tens C_i\br\circ\bl
C^{\al-e_i}\tens R_{C_i}(X) \br$ is invertible in  $\catT'$, and hence
$C^{\al-e_i}\tens R_{C_i}(X)$ has a left inverse.
Since $C^{\al-e_i}$ is invertible, $R_{C_i}(X)$ has a left inverse.
Similarly, $R_{C_i}(X)$ has a right inverse.
\QED

\begin{prop}  [\protect{cf.  \cite[Proposition A.8]{KKK18}}]  \label{Prop: localization}
Let $( C_i ,   R_{C_i}  )_{i\in I}$ be a real commuting family  of braiders in a monoidal category $\catT$, and  set $\lT \seteq  \catT[ C_i^{\otimes -1} \mid i \in I ]$.
Assume that
\begin{enumerate}[\rm (a)]
\item  $\catT$ is an abelian category,
\item  $\otimes$ is exact.
\end{enumerate}
Then $\lT$ is an abelian category in which $\otimes $ is exact, and the functor $\Upsilon: \catT \rightarrow \lT $ is exact.
\end{prop}

\begin{prop} \label{prop:subloc} \
Let $( C_i ,   R_{C_i}  )_{i\in I}$ be a real commuting family  of braiders in a monoidal category $\catT$, and  set $\lT \seteq  \catT[ C_i^{\otimes -1} \mid i \in I ]$.
Assume that
\begin{enumerate}[\rm (a)]
\item  $\catT$ is an abelian category,
\item  $\otimes$ is exact.
\end{enumerate}
Let $\catT'$ be a full subcategory of $\catT$ closed by
taking subquotients, extensions, and contains all the $C_i$'s, and 
let $\lT'$ be the localization $\catT'[ C_i^{\otimes -1} \mid i \in I ]$.
Then $\lT'$ is a full subcategory of $\lT$ closed by taking subquotients and
extensions.
\enprop
\Proof
It is obvious that $\lT'$ is a full subcategory of $\lT$ closed by taking subquotients.
Let us show that $\lT'$ is closed by extensions.

Let us denote by $\Phi\cl \catT\to\lT$ the canonical functor.
Let $0\to X'\to X\to X''\to0$ be an exact sequence in $\lT$ such that
$X'$, $X''\in\lT'$.
We may assume that there exist a sequence of morphisms
$Z'\To[f] Z\To[g] Z''$ in $\catT$ with $Z',Z''\in\catT'$ and $\al\in\lG_{\ge0}$ such that
$\Phi(C^{\al})^{\tens -1}\tens \Phi(Z')\to\Phi(C^{\al})^{\tens-1}\tens\Phi(Z)\to\Phi(C^{\al})^{\tens-1}\tens\Phi(Z'')$ is isomorphic to $X'\to X\to X''$.
Since
$0\to \ker(f)\to Z\to \Im(g)\to 0$ is exact,
$Z$ belongs to $\catT'$.
Hence $X\simeq \Phi(C^{\al})^{\tens-1}\tens\Phi(Z)$ belongs to $\lT'$.
\QED

\subsection{Graded cases}\label{sec:graded}

Let $\Lambda$ be a $\Z$-module. A $\bR$-linear monoidal category $\catT$ is \emph{$\Lambda$-graded} if $\catT$ has a decomposition
$ \catT = \bigoplus_{\lambda \in \Lambda} \catT_\lambda $ such that $\triv \in \catT_0$ and $\otimes$ induces a bifunctor $\catT_{\lambda} \times \catT_{\mu} \rightarrow \catT_{\lambda+\mu}$
for any $\lambda, \mu \in \Lambda$.

Let $\catT$ be a $\Lambda$-graded monoidal category, and let $q$ be an invertible central braider in $\catT$, which belongs to $\catT_0$. 
We write $q^n$ ($n\in\Z$) for $q^{\tens n}$ for the sake of simplicity. 
A \emph{graded braider} is a triple $(C, R_C, \dphi)$ of an object $C$, an $\Z$-linear map $\dphi\cl  \Lambda \rightarrow \Z$ and a morphism
$$
R_C(X) \cl  C \otimes X \longrightarrow q^{\dphi(\lambda)} \otimes X \otimes C
$$
which is functorial in $X \in \catT_\lambda$ for $\lambda \in \Lambda$ such that  a graded version of $\eqref{Eq: central obj}$ holds. That means that the diagram
$$
\xymatrix{
C \otimes X \otimes Y   \ar[rr]^{R_C(X)\otimes Y}  \ar[drr]_{R_C(X \otimes Y) \ \ }  &  &   q^{\dphi(\lambda)} \otimes X \otimes C \otimes Y  \ar[d]^{ X \otimes R_C(Y)}   \\
  & &   q^{ \dphi(\lambda+\mu) } \otimes X \otimes Y \otimes C
}
$$
commutes for any $X \in \catT_\lambda$ and $Y \in \catT_{\mu}$.

In a similar manner, we can consider \emph{graded commuting family} in a graded category $\catT$. Let $I$ be an index set, and set $\lG = \Z^{\oplus I}$ and $\lG_{\ge0} = \Z_{\ge0}^{\oplus I}$ .
Let $P_i$ be an object for $i\in I$ and $B_{i,j}\cl  P_i \otimes P_j \buildrel \sim\over \longrightarrow q^{\phi_{i,j}} \otimes P_j\otimes P_i$ an isomorphism for $i,j\in I$.
We call a triple $( \{ P_i \}_{i\in I}, \{  B_{i,j} \}_{i,j\in I}, \{ \phi_{i,j} \}_{i,j \in I} )$ a graded commuting family of $\catT$ if it satisfies
\bna
\item $\{\phi_{i,j}\}_{i,j\in I}$ is a $\Z$-valued skew-symmetric matrix,
\item $B_{i,i} = \id_{P_i \otimes P_i}$ for $i\in I$,
\item  $ B_{j,i} \circ B_{i,j}=\id_{P_i \otimes P_j}$ for $i,j\in I$,
\item a graded version of $\eqref{Eq: cf pentagon}$ holds, i.e., for $i,j,k\in I$,
$$
(B_{j,k}\otimes P_i) \circ (P_j \otimes B_{i,k} ) \circ (B_{i,j}\otimes P_k) = (P_k \otimes  B_{i,j})  \circ (B_{i,k}\otimes P_j) \circ (P_i \otimes B_{j,k} ) .
$$
\end{enumerate}
Then we have a graded version of Lemma \ref{Lem: cf}. 

\begin{lem} \label{Lem: graded cf}
Let $(\{ P_i \}_{i\in I},  \{B_{i,j}\}_{i,j\in I}, \{\phi_{i,j}\}_{i,j\in I} )$ be a graded commuting family in $\catT$,
and let $\gH\cl    \lG \otimes \lG \rightarrow \Z$ be a $\Z$-bilinear map such that  $ \phi_{i,j} =  \gH(e_i, e_j) - \gH(e_j, e_i) $ for $i,j\in I$.
 Then there exist
\begin{enumerate} [\rm (a)]
\item an object $P^\alpha$ for any $\alpha \in \lG_{\ge0}$,
\item  an isomorphism $\xi_{\alpha, \beta} \cl  P^\alpha \otimes P^\beta \buildrel \sim\over \longrightarrow q^{\gH(\alpha, \beta) } \otimes P^{\alpha+\beta}$ for any $\alpha,\beta\in \lG_{\ge 0}$
\end{enumerate}
such that
\begin{enumerate} [ \rm(i) ]
\item  $P^0 = \triv$ and $P^{e_i} = P_i$ for $i\in I$,
\item  the diagram
\begin{align*}
\xymatrix{
P^\alpha \otimes P^\beta \otimes P^\gamma \ar[d]_{ P^\alpha \otimes \xi_{\beta, \gamma}} \ar[rr]^{\xi_{\alpha, \beta} \otimes P^\gamma} && \ar[d]^{\xi_{\alpha+\beta, \gamma}}
q^{\gH(\alpha, \beta)} \otimes P^{\alpha+\beta}\otimes P^\gamma \\
q^{\gH(\beta, \gamma)} \otimes P^\alpha \otimes P^{\beta + \gamma} \ar[rr]^{\xi_{\alpha, \beta+\gamma}\quad } && q^{\gH(\alpha, \beta)+\gH(\alpha, \gamma)+\gH(\beta, \gamma)} \otimes P^{\alpha+\beta +\gamma}
}
\end{align*}
commutes for any $\alpha, \beta, \gamma \in \lG_{\ge0}$,
\item the diagram
\begin{equation*}
\begin{aligned}
\xymatrix{
P^0 \otimes P^0  \ar[d]_{ \wr } \ar[rr]^{\xi_{0,0} } && \ar[d]^{ \wr } P^{0}\\
\triv \otimes \triv  \ar[rr]^{ \simeq} && \ \ \triv  \ ,
}
\qquad
\xymatrix{
 \ar[d]_{ B_{i,j}} \ar[drr]^{\xi_{e_i, e_j}} P^{e_i}\otimes P^{e_j} && \\
q^{\phi_{i,j}} \otimes P^{e_j} \otimes P^{e_i}  \ar[rr]^{\xi_{e_j, e_i}} && q^{ \gH(e_i, e_j) } \otimes P^{e_i+e_j}
}
\end{aligned}
\end{equation*}
commutes for any $i,j\in I$.
\end{enumerate}
Moreover,  such a datum
is unique up to a unique isomorphism.
\end{lem}

We now consider a graded version of a real commuting family of braiders.
Let $ (C_i, R_{C_i}, \dphi_i ) $ be a graded braider for $i\in I$.
We say that $(  C_i ,  R_{C_i} ,  \dphi_i )_{i\in I}$ is a \emph{real commuting family of graded braiders} in $\catT$ if
\bna
\item  \label{Eq: 1 in grcf} $ C_i \in \catT_{\lambda_i}$  for some $\lambda_i \in \Lambda$, and
$\dphi_i(\lambda_i) = 0$, $\dphi_i( \lambda_j ) + \dphi_j( \lambda_i ) = 0$,
\item  $R_{C_i}(C_i) \in \bR^\times \id_{C_i \otimes C_i}$ for $i\in I$,
\item   \label{Eq: 2 in grcf} $R_{C_j}(C_i) \circ R_{C_i}(C_j) \in \bR^\times \id_{C_i \otimes C_j}$ for $i,j\in I$.
\end{enumerate}

We choose a $\Z$-bilinear map $\gH\cl    \lG \otimes \lG \rightarrow \Z$  such that  $ \dphi_i( \lambda_j ) =  \gH(e_i, e_j) - \gH(e_j, e_i) $ for $i,j\in I$.
Using the same argument of the proof of Lemma \ref{Lem: rcf - cf},  by Lemma \ref{Lem: graded cf}, we have  $C^\alpha$, $\xi_{\alpha, \beta}$ and $\bce(\alpha,\beta)$ as in Section \ref{Sec: RCB}.
We define a $\Z$-linear map
$$
\gL\cl  \lG \rightarrow \Lambda, \qquad  e_i \mapsto  \lambda_i \text{ for } i\in I,
$$
and a $\Z$-bilinear map
\begin{align*}
\dphi &\cl  \lG \times \Lambda \rightarrow \Z, \qquad (e_i, \lambda)  \mapsto \dphi_i(\lambda).
\end{align*}
It is easy to check that
\begin{align} \label{Eq: dphi gL}
 \dphi(\alpha, \gL(\beta)) = \gH(\alpha, \beta) - \gH(\beta , \alpha) \quad \text{for $\alpha, \beta \in \lG $.}
\end{align}
Note that the diagram below commutes:
\begin{align*}
\xymatrix{
C^\alpha \otimes C^\beta \ar[rr]^{R_{C^\alpha}(C^\beta) \qquad } \ar[d]_{\xi_{\alpha, \beta}} &&  q^{ \dphi( \alpha, \gL(\beta)) } \otimes C^\beta \otimes C^\alpha  \ar[d]^{\xi_{  \beta, \alpha}} \\
q^{\gH(\alpha, \beta)} \otimes C^{\alpha + \beta} \ar[rr]^{\bce(\alpha, \beta) \id\qquad } &&  q^{ \dphi( \alpha, \gL(\beta)) + \gH(\beta, \alpha) } \otimes C^{\alpha+\beta}
}
\end{align*}

For $X \in \catT_\lambda$,  $Y \in \catT_\mu$  and $ \delta  \in \Ds_{\alpha, \beta}$,
we set
$$
\gHm_\delta( (X, \alpha  ), (Y, \beta  ) ) \seteq  \Hom_{\catT}(  C^{\delta + \alpha}\otimes X, q^{\gH(\delta, \beta-\alpha) + \dphi(\delta+\beta, \mu)} \otimes  Y \otimes C^{ \delta + \beta} ).
$$
For $ \delta, \delta' \in \Ds_{\alpha, \beta}$ with $\delta \preceq  \delta'$ and $ f \in \gHm_\delta( (X, \alpha  ), (Y, \beta  ) )$,
we define $\gzeta_{\delta', \delta}(f)$ to be the morphism such that the following diagram commutes:
$$
\xymatrix{
C^{\delta' - \delta} \otimes C^{ \delta + \alpha} \otimes X   \ar[dd]^{\wr}_{  \xi_{\delta'-\delta,  \delta+\alpha} }  \ar[rr]^{ C^{\delta'-\delta} \otimes f \qquad \qquad }   &&
q^{\gH(\delta, \beta-\alpha) + \dphi(\delta+\beta, \mu)} \otimes   C^{\delta' - \delta} \otimes Y \otimes C^{ \delta+\beta}   \ar[d]^{  R_{C^{\delta'-\delta}} (Y) } \\
 &&  q^{\gH(\delta, \beta-\alpha) + \dphi(\delta'+\beta, \mu)} \otimes Y \otimes C^{\delta' - \delta} \otimes C^{ \delta + \beta}
  \ar[d]_{\wr}^{\xi_{ \delta'-\delta,   \delta+\beta }}  \\
q^{\gH(\delta'-\delta,  \delta+\alpha)} \otimes C^{  \delta'+\alpha} \otimes X  \ar[rr]^{\gzeta_{\delta', \delta}(f) \qquad \qquad} &&
q^{\gH(\delta, \beta-\alpha) + \dphi(\delta'+\beta, \mu) + \gH(\delta'-\delta,  \delta+\beta) } \otimes Y \otimes C^{ \delta'+\beta}.
}
$$
Since
\begin{align*}
\gH& (\delta,  \beta-\alpha) + \dphi(\delta'+\beta, \mu)  + \gH(\delta'-\delta, \delta+\beta) - \gH(\delta'-\delta,  \delta+\alpha)\\
 &= \gH(\delta',  \beta-\alpha) + \dphi(\delta'+\beta, \mu),
\end{align*}
we have the map
\begin{align*} 
\gzeta_{\delta', \delta }  \cl  \gHm_\delta( (X, \alpha  ), (Y, \beta  ) )  \rightarrow \gHm_{\delta'}( (X, \alpha  ), (Y, \beta  ) ).
\end{align*}
One can show that $\zeta^\gr_{\delta'', \delta' }  \circ \zeta^\gr_{\delta', \delta } = \zeta^\gr_{\delta'', \delta } $
for $\delta \preceq \delta' \preceq \delta''$, which tells that it becomes an inductive system.

As in Section \ref{Sec: Loc}, we now obtain the category $\lT$ defined by
\begin{align*}
\Ob (\lT) &\seteq  \Ob(\catT) \times \lG, \\
\Hom_{\lT}( (X, \alpha), (Y, \beta) ) &\seteq   \lim_{  \substack{\longrightarrow \\ \delta \in \Ds_{\alpha, \beta}, \\  \lambda + \gL(\alpha) = \mu + \gL(\beta)  }    }
\Hm^\gr_{\delta}( (X, \alpha  ), (Y, \beta  ) ) ,
\end{align*}
where $X \in \catT_{\lambda}$ and $Y \in \catT_{\mu}$. By the construction, we have the decomposition
$$
\lT = \bigoplus_{\mu \in \Lambda} \lT_{\mu}, \qquad \text{where }  \lT_{\mu} \seteq  \{ (X, \alpha) \mid X \in \catT_\lambda, \ \lambda + \gL(\alpha)=\mu  \}.
$$
Let $X \in \catT_\lambda$, $Y \in \catT_\mu$ and $Z \in \catT_\nu$. For  $f \in \gHm_\delta( (X, \alpha  ), (Y, \beta  ) ) $ and  $ g \in  \gHm_\epsilon( (Y, \beta  ), (Z, \gamma  ) )$, we define
$$
\gPsi_{\delta, \epsilon} (f,g)  \seteq  \bce(\delta+ \beta, \beta-\gamma) \cdot \gtPsi_{\delta, \epsilon}(f,g) ,
$$
where  $\gtPsi_{\delta, \epsilon}(f,g)$ is the morphism such that the following diagram commutes:
$$
\xymatrix{
C^{\epsilon+\beta}\otimes C^{\delta+\alpha}\otimes X \ar[dd]^{\wr}_{\xi_{\epsilon+\beta, \delta+\alpha} }  \ar[rr]^f &&  q^{a}\otimes C^{\epsilon+\beta}\otimes   Y\otimes  C^{\delta+\beta} \ar[d]^g \\
&& q^{b}\otimes  Z\otimes C^{\epsilon+\gamma}\otimes C^{\delta+\beta} \ar[d]_{\wr}^{\xi_{ \epsilon+\gamma, \delta+\beta}} \\
q^{\gH(\epsilon+\beta, \delta+\alpha)}\otimes  C^{\delta+\epsilon+\beta+\alpha} \otimes X  \ar[rr]^{\qquad \gtPsi_{\delta, \epsilon}(f,g)} && q^{c}\otimes  Z \otimes  C^{\delta+\epsilon+\beta+\gamma} ,
}
$$
where
\begin{align*}
&a = \gH(\delta,\beta-\alpha) + \dphi(\delta+\beta, \mu), \quad b = a+\gH(\epsilon,\gamma-\beta) + \dphi(\epsilon+\gamma, \nu), \\
& c = b + \gH(\epsilon+\gamma, \delta+\beta).
\end{align*}
As $ \gL(\beta) + \mu = \gL(\gamma)+\nu $, by $\eqref{Eq: dphi gL}$, we obtain
$$
\dphi(\delta+\beta, \mu) = \gH(\delta+\beta, \gamma-\beta) - \gH(\gamma-\beta, \delta+\beta ) + \dphi(\delta+\beta, \nu).
$$
This implies that
\begin{align*}
c &= \gH (\delta,\beta-\alpha) + \dphi(\delta+\beta, \mu) + \gH(\epsilon,\gamma-\beta) + \dphi(\epsilon+\gamma, \nu) +  \gH(\epsilon+\gamma, \delta+\beta) \\
&=  \gH(\delta+\epsilon+\beta, \gamma-\alpha) + \dphi(\delta+\epsilon+\beta+\gamma, \nu) + \gH(\epsilon+\beta, \delta+\alpha),
\end{align*}
which tells that
$$
\gPsi_{\delta, \epsilon} (f,g) \in \gHm_{\delta+\epsilon+\beta}( (X, \alpha  ), (Z, \gamma ) ) .
$$
Using the same argument as in Section \ref{Sec: Loc}, $\gPsi_{\delta, \epsilon}$ gives the composition in $\lT$, which means that
$\lT$ is a category.

We shall define a bifunctor $\otimes$ on the category $\lT$ as follows. For $\alpha, \alpha', \beta, \beta' \in \lG $, $X \in \catT_\lambda$, $X' \in \catT_{\lambda'}$, $Y \in \catT_\mu$ and $Y' \in \catT_{\mu'}$,
we define
$$
(X, \alpha) \otimes (Y, \beta) \seteq  ( q^{- \dphi(\beta, \lambda) + \gH(\alpha, \beta)} \otimes X \otimes Y, \alpha+\beta  ),
$$
and, for $f \in \gHm_\delta((X, \alpha), (X', \alpha'))$ and $g \in \gHm_\epsilon((Y, \beta), (Y', \beta'))$, we define
$$
\gT_{\delta, \epsilon}(f,g) \seteq  \eta(\epsilon, \alpha-\alpha') \gtT_{\delta, \epsilon}(f,g) ,
$$
where $ \gtT_{\delta, \epsilon}(f,g)$ is the morphism such that the following diagram commutes:
$$
\xymatrix{
C^{\delta+\alpha}  \otimes X \otimes C^{\epsilon+\beta}\otimes  Y  \ar[rr]^{f \otimes g} &&  q^{b}\otimes  X' \otimes C^{\delta+\alpha'}  \otimes  Y' \otimes  C^{\epsilon+\beta'}   \ar[d]^{R_{C^{\delta+\alpha'}} (Y') } \\
q^{-\dphi(\epsilon+\beta, \lambda) }\otimes C^{\delta+\alpha} \otimes C^{\epsilon+\beta}\otimes X \otimes Y  \ar[u]^{R_{C^{\epsilon+\beta}} (X) }  \ar[d]^\wr_{\xi_{\delta+\alpha, \epsilon+\beta}}
&&  q^{c}\otimes  X' \otimes Y' \otimes C^{\delta+\alpha'} \otimes C^{\epsilon+\beta'}  \ar[d]_\wr^{\xi_{\delta+\alpha', \epsilon+\beta'}} \\
q^{a}\otimes  C^{\delta+\epsilon+\alpha+\beta}\otimes X \otimes Y \ar[rr]^{\gtT_{\delta,\epsilon} (f, g)}  &&  q^{d}\otimes  X' \otimes Y' \otimes C^{\delta+\epsilon+\alpha'+\beta'},
}
$$
where
\begin{align*}
a &= -\dphi(\epsilon+\beta, \lambda) + \gH(\delta+\alpha, \epsilon+\beta), \\
b &=  \gH(\delta, \alpha' - \alpha) + \dphi(\delta+\alpha', \lambda') + \gH(\epsilon, \beta'-\beta) + \dphi(\epsilon+\beta', \mu'), \\
c &= b + \dphi(\delta + \alpha', \mu'), \qquad d = c + \gH(\delta + \alpha', \epsilon+\beta').
\end{align*}
By $\lambda + \gL(\alpha) = \lambda' + \gL(\alpha')$,  we have
\begin{align*}
\dphi(\epsilon+\beta, \lambda) &= \dphi(\epsilon+\beta', \lambda') - \dphi(\beta', \lambda') + \dphi(\beta, \lambda) + \gH(\epsilon, \alpha' - \alpha) - \gH(\alpha' - \alpha, \epsilon),
\end{align*}
which implies that
\begin{align*}
d-a&=
\gH(\delta+ \epsilon, \alpha' + \beta' - \alpha-\beta) + \dphi( \epsilon + \delta+\alpha' + \beta', \lambda' + \mu') \\
& \ \  \  - \dphi(\beta', \lambda') + \gH(\alpha', \beta') +  \dphi(\beta, \lambda) - \gH(\alpha, \beta).
\end{align*}
This tells that
$$
\gT_{\delta, \epsilon}(f,g) \in \gHm_{\delta+\epsilon}((X, \alpha)\otimes (Y,\beta),   (X', \alpha') \otimes (Y', \beta')).
$$
By the same argument as in Section \ref{Sec: Loc}, one can show that $\gT_{\delta, \epsilon}$ yields a bifunctor
\begin{equation*} 
\begin{aligned}
 \Hom_{\lT}(   (X, \alpha) ,  (X', \alpha') ) & \times  \Hom_{\lT}(  (Y, \beta) , (Y', \beta') ) \buildrel \otimes \over  \longrightarrow \\
 & \qquad  \Hom_{\lT}(   (X, \alpha) \otimes  (Y, \beta)  , (X', \alpha') \otimes (Y', \beta') ),
\end{aligned}
\end{equation*}
which gives the following theorem.

\begin{thm} \label{Thm: graded localization}
$\lT$ becomes a monoidal category. Moreover, it satisfies the same conditions as in Theorem \ref{Thm: localization}.
\end{thm}

We denote by $\catT [ C_i^{\otimes -1} \mid i\in I]$ the localization $\lT$ in Theorem \ref{Thm: graded localization}.
Note that
$$
(X, \alpha+\beta) \simeq  q^{-\gH(\beta, \alpha)} \otimes (C^\alpha \otimes X, \beta), \quad
(\triv, \beta) \otimes  (\triv, -\beta) \simeq q^{-\gH(\beta, \beta)} (\triv, 0)
$$
for $\alpha \in \lG_{\ge 0}$ and  $\beta \in \lG$.
We now have a graded version of Proposition \ref{Prop: localization}
\begin{prop}
Let $( C_i ,  R_{C_i}  ,   \dphi_i  )_{i\in I}$ be a real commuting family of graded braiders in a graded monoidal category $\catT$,
and  set $\lT \seteq   \catT  [ C_i^{\otimes -1} \mid i \in I ]$.
Assume that
\bna
\item $\catT$ is an abelian category,
\item $\otimes$ is exact,
\end{enumerate}
Then $\lT$ is an abelian category with exact $\tens$,
and the functor $\Upsilon\cl  \catT \rightarrow \lT $ is exact.
\end{prop}

\vskip 2em

\section{Affinizations and R-matrices} \label{Sec: Affinizations and R-matrices}

\subsection{Affinizations}

Let $R$ be a quiver Hecke algebra of \emph{arbitrary type}.
We recall the notions of affinizations and R-matrices given in \cite{KP18}.
For $\beta \in \rlQ_+$ and $i\in I$, let
\begin{align} \label{Eq: def of p}
\mathfrak{p}_{i, \beta}  = \sum_{\nu \in I^\beta} \Bigl(\hs{1ex}  \prod_{a \in \{1, \ldots, \Ht(\beta) \},\ \nu_a=i} x_a \Bigr) e(\nu)\in R(\beta).
\end{align}
Then $\mathfrak{p}_{i, \beta} $ belongs to the center of $R(\beta)$.
When there is no afraid of confusion, we simply write $\mathfrak{p}_{i} $ for $\mathfrak{p}_{i, \beta}$.
\begin{df} \label{Def: aff}
Let $M$ be a simple $R(\beta)$-module. An \emph{affinization} of $M$ 
 with degree $d_{\Ma}$ is an $R(\beta)$-module $\Ma$ with an endomorphism $z_{\Ma}$ of $\Ma$
with degree $d_{\Ma} \in \Z_{>0}$
and an isomorphism $\Ma / z_{\Ma} \Ma \simeq M$ such that
\begin{enumerate}[\rm (i)]
\item $\Ma$ is a finitely generated free module over the polynomial ring $\bR[z_{\Ma}]$,
\item $\mathfrak{p}_{i} \Ma \ne 0$ for all $i\in I$.
\end{enumerate}
\end{df}

Note that every affinization is essentially even, i.e. $d_{\Ma} \in 2 \Z_{>0} $ \cite[Proposition 2.5]{KP18}.
Thus, from now on, we assume that every affinization is even.

\begin{lem}[{\cite[Lemma 2.7]{KP18}}] \label{Lem: aff end}
Let $\Ma$ be an affinization of a simple $R$-module $M$. Then we have
$$
\END_{R_{\bR[\z]} (\beta)} (\Ma) \simeq \bR[\z] \id_\Ma.
$$
In particular,  $\mathfrak{p}_{i}\vert_ \Ma$ is a monomorphism.
\end{lem}

\Lemma\label{lem:affei}
Let $i\in I$ and $(\Ma,\z)$ an affinization of a simple $R(\beta)$-module $M$.
We set $m\seteq \ep_i(M)$ and $m' \seteq \ep_i^*(M)$. Then we have
\bnum
 \item $E_i^{m+1}\Ma\simeq 0$ and
$E_i^{(m)}\Ma$ is an affinization of $E_i^{(m)} M$,
 \item $E_i^{\,*\;m'+1}\Ma\simeq 0$ and $E_i^{* \hskip 0.1em (m')}   \Ma$ is an affinization of $E_i^{* \hskip 0.1em (m')} M$.
\ee
\enlemma
\begin{proof}
We shall only prove (i) since the case (ii) is similar.
In the sequel, we ignore grading shifts.
Let $d = \deg \z$.
Applying the exact functor $E_i^{m+1}$ to
$$0 \longrightarrow \Ma \buildrel \z \over \longrightarrow \Ma \longrightarrow M \longrightarrow 0,
$$
we have an exact sequence
$$
0 \longrightarrow  E_i^{m+1} \Ma  \To[z]E_i^{m+1}  \Ma \To 0.
$$
Hence Nakayama's lemma shows $E_i^{m+1} \Ma\simeq0$.
Applying the exact functor $E_i^{ (m)}$ to the same exact sequence,
we have an exact sequence
$$
0 \longrightarrow q^d \Ma_0 \buildrel \mathsf{z} \over \longrightarrow \Ma_0 \longrightarrow M_0 \longrightarrow 0,
$$
where $\Ma_0 = E_i^{(m)}\Ma$,  $ M_0 = E_i^{(m)} M$ and $ \mathsf{z} = E_i^{(m)}(\z)$.

Thus, it remains to show that $\mathfrak{p}_{j,\gamma} \vert_{\Ma_0} $ does not vanish for any $j\in I$, where $\gamma=\beta-m\al_i$.
We can regard $\Ma_0$ as  a  subspace of $\Ma$ by
$$
\Ma_0 \simeq \{  u \in  e(m \alpha_i, \beta - m\alpha_i)\Ma \mid  \tau_k u = 0
\ \text{for $1 \le k < m$}  \}.
$$
If $j \ne i$, then it is clear because
$\gp_{j,\gamma} \vert_{\Ma_0}=\gp_{j,\beta} \vert_{ \Ma_0 }$  is injective. 
Suppose that $j=i$. Then,
$ \mathfrak{p}_{i,\beta}\vert_{\Ma_0} =\bigl((x_1\cdots x_m)\mathfrak{p}_{i,\gamma}\bigr)\vert_{\Ma_0} $ is injective.
Hence $ \mathfrak{p}_{i,\gamma}\vert_{\Ma_0}$ is also injective.
\end{proof}

In the sequel,
\eq
\text{$t_i$ is an indeterminate of degree $(\alpha_i, \alpha_i)$.}
\eneq

\begin{df} \label{Def: lc lep}
Let $i\in I$ and let $M$ be an $R(\beta)$-module. We set $m = \ep_i(M)$, $m' = \ep^*_i(M)$ and $n = \Ht(\beta)$.
We define
\begin{align*} \textstyle
\lc_i(M) (t_i) &\seteq   \Bigl(\sum_{\nu  \in I^\beta} \prod_{\nu_k = i} (t_i - x_k) e(\nu) \Bigr) {\big\vert}_M \ \in \END_R(M)[t_i], \\
\lep_i(M) (t_i) &\seteq  (t_i-x_1)\cdots (t_i-x_m)\mid_{E_i^{(m)} M} \ \in    \END_R(E_i^{(m)}M)[t_i],  \\
\lep^*_i(M) (t_i) &\seteq  (t_i-x_{n-m'+1}) \cdots (t_i-x_n)\mid_{E_i^{* \hskip 0.1em  (m')  } M} \  \in    \END_R(E_i^{*\hskip 0.1em (m')}M)[t_i].
\end{align*}
\end{df}

Note that $\lc_i(M)$, $\lep_i(M)$ and $\lep_i^*(M)$ are homogeneous.
If $M$ is a simple module, then $\lc_i(M)$, $\lep_i(M)$ and $\lep_i^*(M)$
are powers of $t_i$.
For an affinization $\Ma$ of a simple module,
 by Lemma \ref{Lem: aff end} and Lemma~\ref{lem:affei}, we have
$$
\lc_i(\Ma), \lep_i(\Ma), \lep_i^*(\Ma) \in \bR[\z, t_i].
$$

\begin{lem} \label{Lem: minimal}
Let $\gamma\in \rlQ_+$, $m\in\Z_{>0}$ and $\beta=\gamma+m\al_i$.
Let $M$ be a non-zero $R(\beta)$-module such that
$\eps_i(M)=m$.  Set $B\seteq \END_R(E_i^{(m)}M)$.
Let $f(t_i) \in B[t_i]$. 
Then the following conditions are equivalent:
\bna
\item $f(x_k)=0$ in $\END_\bR(E_i^mM)$ for any $k$ such that $1\le k\le m$,
\item $f(x_k)=0$ in $\END_\bR(E_i^mM)$ for some $k$ such that $1\le k\le m$,
\item $\lep_i(M)(t_i)$ divides $ f(t_i)$, i.e., $f(t_i)\in B[t_i]\, \lep_i(M)(t_i)$. 
\ee

If we assume further that $M$ is generated by $e(m\al_i,\gamma)M=E_i^mM$
as an $R(\beta)$-module
and $f(t_i)\in\End_R(M)[t_i]$, 
then
the equivalent conditions above implies

\noi
{\rm(d)} $f(x_1)\vert_{E_iM}=0$.
\end{lem}
\begin{proof}
Set $K=\bigl\{a(x)\in \END_R(E_i^{(m)}M)[x_1,\dots, x_m]\mid
\text{$a(x)$ kills $E_i^{m}M$}\bigr\}$.
Then $K$ is stable by $s_k$
since $\vphi_k\cdot a(x)\cdot\vphi_k=s_k\bl a(x)\br$ (see Lemma~\ref{Lem: intertwiners} below).
It is stable also by $\partial_k$ by \eqref{eq:taupol}.
Hence the equivalence of (a) and (b) is obvious.
Assuming that $f(x_1)\in K$, let us prove
$ f(t_i)\in B[t_i]\,  \lep_i(M) $.

 It is easy to see $ \lep_i(M)(x_1)  \in K$.
By dividing $f(t_i)$ by the monic polynomial $ \lep_i(M)(t_i) $ of degree $m$,
we may assume that $f(t_i)$ is a polynomial in $t_i$
of degree strictly less than $m$.
Write $f(t_i)=\sum_{j=0}^na_jt_i^j$ with $0\le n<m$ and $a_j\in  B $.
In order to see $f(t_i)=0$,
it is enough to show that $a_n=0$.
We have
$\partial_n\cdots\partial_1 f(x_1)=(-1)^na_n\in K$. Hence $a_n=0$.

\medskip\noi
Let us show (d).
Set $$e=\sum_{\nu\in I^\gamma, \nu_1\not=i}e(\nu)\in R(\gamma).$$
Then we have
$e(m\al_i,\gamma)M=\bl e(m\al_i)\etens e\br M$.
By the shuffle lemma, we have
$$R(\beta)e(m\al_i,\gamma)=\sum_{w\in \sg_{m,n}}\tau_w\bl R(m\al_i)\etens R(\gamma)\br.$$
Here $\sg_{m,n}=\set{w\in \sg_{m+n}}{\text{$w(k)<w(k+1)$  if $1\le k<m+n$ and $k\not=m$}}$.
Hence \allowbreak $e(\al_i,  \beta-\al_i) R(\beta)\bl e(m\al_i)\etens e\br
=\sum_{w}\tau_wR(m\al_i)\etens R(\gamma)e$,
where $w$ ranges over $\sg_{m,n}$ and $w(m+1)\not=1$. It implies $w(1)=1$. 
Hence we obtain 
$$e(\al_i,  \beta-\al_i) R(\beta)\bl e(m\al_i)\etens e\br
\subset \bl R(\al_i)\etens R(\beta-\al_i)\br\cdot\bl R(m\alpha_i) \etens R(\gamma)\br.$$ 
Since we have a surjective map
$P(i^m)\conv E_i^{(m)}M\epito M$ by the assumption, 
we have 
\eqn
e(\al_i,\beta-\al_i)\bl P(i^m)\conv E_i^{(m)}M\br&&=
e(\al_i,\beta-\al_i) R(\beta)\bl e(m\al_i)\etens e\br
\bl P(i^m)\etens E_i^{(m)}M\br\\
&&\subset
\bl R(\al_i)\etens R(\beta-\al_i)\br \bl P(i^m)\etens E_i^{(m)}M\br.
\eneqn
Hence we obtain
$$e(\al_i,\beta-\al_i)M\subset \bl R(\al_i)\etens R(\beta-\al_i)\br
e(m\al_i,\gamma)M.$$
Thus $f(x_1)$ kills $e(\al_i,\beta-\al_i)M$.
\end{proof}

Let $z$ be an indeterminate of degree $d\in\Z_{>0}$ and we set
$$
A \seteq  \bR[z].
$$
\begin{prop} \label{Prop: affinization of L(i^m)}
Let $i\in I$ and $m\in \Z_{>0}$.
We take a homogeneous monic polynomial  $f(t_i) \in A[t_i]$ such that
\begin{enumerate} [\rm (i)]
\item  $\deg_{t_i} f(t_i) = m$,
\item  $f(0)\ne 0$.
\end{enumerate}
Then, there is an affinization $\Ma$ of $L(i^m)$ with $ \lc_i( \Ma) = f(t_i)$.
\end{prop}
\begin{proof}
Let $R_A$ be a quiver Hecke algebra over the polynomial ring $A$.
We set $\Lambda \seteq m\La_i$ and $a_\Lambda(t_i) : =f(t_i)$.
We now define
$$
\Ma \seteq  F_i^{a_\Lambda \hskip 0.1em (m)} (A) \in R_A^{a_\Lambda}(\Lambda - m \alpha_i)\Mod
$$
and consider the endomorphism $z \in \END_R(\Ma)$ defined by $u \mapsto zu$ for $u\in \Ma$.
We shall prove that $(\Ma, z)$ is an affinization of $L(i^m)$.

Since the functor $F_i^{a_\Lambda \hskip 0.1em (m)}$ gives an equivalence of the
category of modules over
 $R_A^{a_\Lambda}(\Lambda - m \alpha_i)$ and the one over
$R_A^{a_\Lambda}(\Lambda)=A$ by Proposition~\ref{prop:htlt}, we have
\bnum
\item
$\Ma$ is a finitely generated projective module over $R_A^{a_\Lambda}(\Lambda - m \alpha_i)$.
\item
$\END_{R_A^{a_\Lambda}(\Lambda - m \alpha_i)}(\Ma)\simeq A$.
\item
the exact sequence
$
0 \rightarrow q^2A \buildrel{z}\over\rightarrow A \rightarrow \bR \rightarrow  0
$
gives the exact sequence
$$
0 \rightarrow q^2\Ma \buildrel{z}\over\rightarrow \Ma \rightarrow L(i^m) \rightarrow  0.
$$

\ee
By (i), $\Ma$ is a finitely generated free module over $A$.

On the other hand, (ii) implies $\lc_i(\Ma)(t_i)\in A[t_i]$ and
we know that
\begin{enumerate}[\rm(a)]
 \item $  \deg_{t_i} (f(t_i) ) = m = \deg_{t_i} (\lc_i(\Ma)(t_i) )$,
\item $ f(x_m) \vert_\Ma = a_\Lambda(x_m) \vert_\Ma =0$,
\item $f(t_i)$ and $\lc_i(\Ma)(t_i)$ are monic.
\end{enumerate}
Thus,  Lemma \ref{Lem: minimal} implies that
$$
\lc_i(\Ma)(t_i) = f(t_i),
$$
which tells, by the assumption,
$$
 \mathfrak{p}_i \vert_\Ma = (-1)^m \lc_i(\Ma)(0) = (-1)^m f(0) \ne 0.
 $$
Therefore, $(\Ma, z)$ is an affinization of $L(i^m)$ with $ \lc_i( \Ma) = f(t_i)$.
\end{proof}

\begin{df}
For a homogeneous monic polynomial  $f(t_i) \in A[t_i]$ such that
\begin{enumerate}[\rm (i)]
\item  $\deg_{t_i} f(t_i) = m$,
\item  $f(0)\ne 0$,
\end{enumerate}
we denote by $ \aW(i, f(t_i))$ the affinization of $L(i^m)$ defined in Proposition \ref{Prop: affinization of L(i^m)}.
\end{df}

\begin{rem}
Let $\bR[x_1, \ldots, x_m]$ be the faithful polynomial representation of $R(m\alpha_i)$ \cite[Example 2.2]{KL09}.
Since $P(i^{ m })$ is isomorphic to $\bR[x_1, \ldots, x_m]$ up to a grading shift, one can show that
$$
\aW(i , f(t_i)) \simeq  \bR[z, x_1, \ldots, x_m]\big/
\Bigl(\sum_{j=1}^m\bR[z,x_1, \ldots, x_m](e_j(x)-  a_j(z)    )\Bigr)
$$
up to a grading shift. Here $e_j(x)$ ($1\le j\le m$) is the $j$-the elementary symmetric polynomial in $(x_1,\ldots,x_m)$
and $f(t_i)=\sum_{j=0}^m (-1)^ja_j(z)t_i^{m-j}$ ($a_j(z)\in \bR[z]$). 
\end{rem}

\Lemma\label{lem:lepi}
Let $A$ be a commutative $\corp$-algebra,
$i\in I$ and let $\Ma$ be an $R_A(\beta)$-module.
We set $m\seteq \ep_i(\Ma)$. Let $h\cl \END_{R_A}(\Ma)
\to \END_{R_A}(E_i^{(m)}\Ma)$ be the canonical map.
Then we have
$h(\lc_i(\Ma)) = \lep_i(\Ma) \lc_i(E_i^{(m)}\Ma)$
as elements of $\END_{R_A}(E_i^{(m)}\Ma)[t_i]$.

Similarly, setting  $m' \seteq \ep_i^*(M)$
and $h^*\cl \END_{R_A}(\Ma)
\to \END_{R_A}(E_i^{\;*\,(m')}\Ma)$  the canonical map,
we have
$h^*\bl\lc_i(\Ma)\br = \lc_i(E_i^{* \ms{5mu} (m')}\Ma) \lep^*_i(\Ma)  $.
\enlemma
\begin{proof}
Set $\gamma=\beta-m\al_i$ and  $\mathrm{ch}_{i,\beta}(t_i)=\sum_{\nu  \in I^\beta} \prod_{\nu_k = i} (t_i - x_k) e(\nu)
\in R(\beta)[t_i]$.
Then we have $\mathrm{ch}_{i,\beta}(t_i)e(m\al_i,\gamma)=
\mathrm{ch}_{i,m\al_i}(t_i)\etens\mathrm{ch}_{i,\gamma}(t_i)$.
Hence we have
\eqn
 \mathrm{ch}_{i,\beta}(t_i)\big\vert_{E_i^{(m)}\Ma}&&=
\bl\prod_{k=1}^m (t_i - x_k)\br \big|_{E_i^{(m)}\Ma}
\;\cdot\; \mathrm{ch}_{i,\gamma}(t_i)\big|_{E_i^{(m)}\Ma}\\
&&=\lep_i(\Ma) \lc_i(E_i^{(m)}\Ma).
\eneqn

\medskip
The case for $E_i^*$ can be proved similarly.
\end{proof}

\subsection{R-matrices}\label{sec:R-matrices}

Let $\beta \in \rlQ_+$ with $m =  \Ht(\beta)$. For  $k=1, \ldots, m-1$ and $\nu \in I^\beta$, the \emph{intertwiner} $\varphi_k  $ is defined by 
$$
\varphi_k e(\nu) =
\bc
 (\tau_k x_k - x_k \tau_k) e(\nu)
= (x_{k+1}\tau_k - \tau_kx_{k+1}) e(\nu) & \text{ if } \nu_k = \nu_{k+1},  \\
 \tau_k e(\nu) & \text{ otherwise.}
\ec
$$ 

\begin{lem} [{\cite[Lemma 1.5]{KKK18}}] \label{Lem: intertwiners} \
\begin{enumerate}[\rm (i)]
\item $\varphi_k^2 e(\nu) = ( Q_{\nu_k, \nu_{k+1}} (x_k, x_{k+1} )+ \delta_{\nu_k, \nu_{k+1}} )\, e(\nu)$.
\item $\{  \varphi_k \}_{k=1, \ldots, m-1}$ satisfies the braid relation.
\item For a reduced expression $w = s_{i_1} \cdots s_{i_t} \in \sg_m$, we set $\varphi_w \seteq  \varphi_{i_1} \cdots \varphi_{i_t} $. Then
$\varphi_w$ does not depend on the choice of reduced expression of $w$.
\item For $w \in \sg_m$ and $1 \le k \le m$, we have $\varphi_w x_k = x_{w(k)} \varphi_w$.
\item For $w \in \sg_m$ and $1 \le k < m$, if $w(k+1)=w(k)+1$, then $\varphi_w \tau_k = \tau_{w(k)} \varphi_w$.
\end{enumerate}
\end{lem}

For $m,n \in \Z_{\ge 0}$, we set $w[m,n]$ to be the element of $\sg_{m+n}$ such that
$$
w[m,n](k) \seteq
\left\{
\begin{array}{ll}
 k+n & \text{ if } 1 \le k \le m,  \\
 k-m & \text{ if } m < k \le m+n.
\end{array}
\right.
$$
Let $\beta, \gamma \in \rlQ_+$ and set $m\seteq  \Ht(\beta)$ and $n\seteq \Ht(\gamma)$.
For $M \in R(\beta)\Mod$ and $N \in R(\gamma)\Mod$, the $R(\beta)\otimes R(\gamma)$-linear map $M \otimes N \rightarrow N \conv M$ defined by $u \otimes v \mapsto \varphi_{w[n,m]} (v \otimes u)$
can be extended to an $R(\beta+\gamma)$-module homomorphism (up to a grading shift)
$$
\RR_{M,N}\cl  M\conv N \longrightarrow N \conv M.
$$

Let $\Ma$ be an affinization of a simple $R$-module $M$, and  let $N$ be a non-zero $R$-module. We define a homomorphism (up to a grading shift)
$$
\nR_{\Ma, N} \seteq  \z^{-s} \RR_{\Ma, N}\cl  \Ma \conv N \longrightarrow N \conv \Ma,
$$
where $s$ is the largest integer such that $ \RR_{\Ma, N}(\Ma \conv N) \subset \z^s (N \conv \Ma)$.
We define
$$
\Rr_{M,N} \cl M \conv N \longrightarrow N \conv M
$$
to be the homomorphism (up to a grading shift) induced from $\nR_{\Ma, N}$ by specializing at $\z=0$. By the definition, $\Rr_{M,N}$ never vanishes.
We now define
\begin{align*}
\La(M,N) &\seteq  \deg (\Rr_{M,N}) , \\
\tLa(M,N) &\seteq   \frac{1}{2} \bl \La(M,N) + (\wt(M), \wt(N)) \br , \\
\Dd(M,N) &\seteq  \frac{1}{2} \bl\La(M,N) + \La(N,M)\br.
\end{align*}

\Lemma
Let $M$ be a simple $R(\al)$-module and $N$ a simple $R(\beta)$-module.
Assume that one of them is real simple and admits an affinization of degree $d$.
Then one has 
$$
\tLa(M,N) , \   \Dd(M,N) \in \frac{\;d\;}{2}\; \Z_{\ge0}. 
$$
\enlemma
\Proof
We treat only the case when $M$ is real and admits an affinization
$(\Ma,\z)$.

Set $m=\Ht(\al)$ and $n=\Ht(\beta)$. For $w\in \sym_{r}$,
choosing a reduced expression $s_{i_1}\cdots s_{i_{\ell}}$ of $w$, we write
$\tau_w=\tau_{i_1}\cdots\tau_{i_\ell}$.

Then, by \cite[{(1.17)}]{KKK18}, we have for $u\in \Ma$ and $v\in N$
\eqn
\RR_{\Ma,N}(u\etens v)
\in   \tau_{w[n,m]} \left( \sum_{\mu \in I^\al, \nu \in I^\beta}  \prod_{(i,j) \in B(\nu,\mu)}(x_i-x_j)e(\nu,\mu) \right) v\etens u  
+\sum_{w\prec w[n,m]}\tau_w(N\etens \Ma),
\eneqn
where $B(\nu,\mu) = \{ (i,j)\mid 1 \le i \le n < j \le m+n,\ \nu_i = \mu_j  \}  $. 
We have
\eqn
 \left( \sum_{\mu \in I^\al, \nu \in I^\beta}  \prod_{(i,j) \in B(\nu,\mu)}(x_i-x_j)e(\nu,\mu) \right) v\etens u  
&&= \left( \sum_{\mu \in I^\al, \nu \in I^\beta}  \prod_{(i,j) \in B(\nu,\mu)}(-x_j)e(\nu,\mu) \right) v\etens u  \\
&&=\pm v\etens \bl\prod_{i\in I}\mathfrak{p}_{i, \al}\br^{  \frac{2(\beta, \La_i)}{(\al_i, \al_i)} } u
=c\z^s u
\eneqn
for some $c\in\corp^\times$ and $s\in\Z_{\ge0}$.
Therefore, we obtain
\eqn
\RR_{\Ma,N}(u\etens v)
\in c\z^s\tau_{w[n,m]}e(\beta,\al)(v\etens u)
+\sum_{w\prec w[n,m]}\tau_w(N\etens \Ma).
\eneqn 
Thus we conclude
\eqn
\nR_{\Ma,N}(u\etens v)
\in c\z^a\tau_{w[n,m]}e(\beta,\al)(v\etens u )
+\sum_{w\prec w[n,m]}\tau_w(N\etens \Ma)
\eneqn 
for some $a\in\Z_{\ge0}$.
Then we obtain
\eqn
2\tLa(M,N)&&=\deg(\nR_{\Ma,N})+(\al,\beta)\\
&&=\deg(\z^a\tau_{w[n,m]}e(\beta,\al))+(\al,\beta)\\
&&=\deg(\z^a)\in d\Z_{\ge0}.
\eneqn

Now let us show the case of $\Dd(M,N)$.
By Proposition~\ref{Prop: affinization}~(iii) below, we have
$$\nR_{N,\Ma}\circ\nR_{\Ma,N}=c\z^s\id_{\Ma\circ N}$$
for some $c\in\corp^\times$ and $s\in\Z_{\ge0}$.
Then we have
\eqn
2\Dd(M,N)&&=\deg\bl\nR_{N,\Ma}\circ\nR_{\Ma,N}\br\\
&&=\deg(\z^s)\in {d\Z_{\ge0}}.
\eneqn
\QED

\Cor
Let $M$ be a simple module.
We assume that  $\wt(M)\not=0$ and $M$ is \ra.
Then we have
$$\bl\wt(M),\wt (M)\br>0.$$
\encor
\Proof
Since $\La(M,M)=0$, we have
$$0\le 2\tLa(M,M)=\bl\wt(M),\wt (M)\br.$$
If $\bl\wt(M),\wt (M)\br=0$, then 
$\tLa(M,M)=0$ and one has
$$\rmat{M,M}(u\etens v)\in\tau_{w[n,n]}(v\tens u)+
\sum_{w<w[n,n]}\tau_w(M\etens M).$$
Hence $\rmat{M,M}\not\in \corp\,\id_{M\circ M}$,
which contradicts
the fact that $M\conv M$ is simple.
\QED

\begin{prop}[{\cite[Proposition 2.10, Proposition 2.11]{KP18}}] \label{Prop: affinization}
Let $M$ and $N$ be simple $R$-modules. Suppose that $M$ is real and admits an affinization $\Ma$.
\begin{enumerate}[\rm (i)]
\item $M \conv N$ has a simple head and a simple socle. Moreover, $\Im(\Rr_{M,N})$ is equal to $M \hconv N$ and $N \sconv M$ up to grading shifts.
\item
\begin{align*}
&\HOM_R( M \conv N, M \conv N) = \bR \hskip 0.2em  \id_{M \conv N},   \quad \HOM_R( N \conv M, N \conv M) = \bR \hskip 0.2em  \id_{N \conv M}, \\
 &\HOM_R( M \conv N, N \conv M) = \bR \hskip 0.2em \Rr_{M , N}, \quad \quad  \HOM_R( N \conv M, M \conv N) = \bR \hskip 0.2em \Rr_{N ,M}.
 \end{align*}
\item
\begin{align*}
&\HOM_{R_{\bR[\z]}}( \Ma \conv N, \Ma \conv N) = \bR[\z] \hskip 0.2em  \id_{\Ma \conv N},   \ \HOM_{R_{\bR[\z]}}( N \conv \Ma, N \conv \Ma) = \bR[\z] \hskip 0.2em  \id_{N \conv \Ma}, \\
 &\HOM_{R_{\bR[\z]}}( \Ma \conv N, N \conv \Ma) = \bR[\z] \hskip 0.2em \nR_{\Ma , N},  \quad  \HOM_{R_{\bR[\z]}}( N \conv \Ma, \Ma \conv N) = \bR[\z] \hskip 0.2em \nR_{N ,\Ma}.
 \end{align*}
\end{enumerate}
\end{prop}

Now assume that $N$ has an affinization $(\Na,z_\Na)$.
Then similarly to the case of $M$,
we can define $\nR_{M,\Na}$.
The proposition above implies that
$\deg \nR_{\Ma,N}=\deg \nR_{M,\Na}$, and
$\nR_{\Ma,N}\vert_{z_{\Ma}=0}=\nR_{M,\Na}\vert_{z_{\Na}=0}$ up to a constant multiple.
Hence $\Lambda(M,N)$ and $\rmat{M,N}$ (up to a constant multiple) are
well defined when either $M$ or $N$ admits an affinization.
Moreover, they do not depend on the choice of the affinization.

\medskip

Lemma \ref{Lem: self-dual} -- \ref{Lem: generalized crystal} 
below were proved when $R$ is symmetric.  However they still hold also for non-symmetric case.
Since the proofs are similar,
we omit the proofs.

\begin{lem} [{\cite[Lemma 3.1.4]{KKKO18}}] \label{Lem: self-dual}
Let $M$ and $N$ be self-dual simple $R$-modules. If one of them is real and also admits an affinization, then
$q^{\tLa(M,N)} M \hconv N$
is a self-dual simple $R$-module.
\end{lem}

\begin{lem} [\protect{cf.\  \cite[Corollary 3.8]{KKOP18}}] \  \label{Lem: M L(i)}
Let $M$ be a simple $R$-module. For $i\in I$, we have
\begin{align*}
\La(L(i), M) &= (\alpha_i, \alpha_i) \ep_i(M) + (\alpha_i, \wt(M)), \\
\La( M, L(i)) &= (\alpha_i, \alpha_i) \ep^*_i(M) + (\alpha_i, \wt(M)).
\end{align*}
\end{lem}

\Lemma[{\cite [Proposition 3.2.8 and 3.2.10]{KKKO18}}]\label{lem:additive}
Let $M$, $N$, $N'$ and $L$ be simple modules.
Assume that $M$ is \ra.
Let $N\conv N'\epito L$ be an epimorphism.
Then we have
\bnum
\item
$\La(M,L)\le \La(M,N)+\La(M,N')$,
\item if $\La(M,L)= \La(M,N)+\La(M,N')$,
then the following diagram commutes up to a constant multiple.
$$\xymatrix{
M\conv N\conv N'\ar[r]^-{\rmat{M,N}}\ar[d]&
N\conv M\conv N'\ar[r]^-{\rmat{M,N'}}&N\conv N'\conv M\ar[d]\\
M\conv L\ar[rr]^-{\rmat{M,L}}&&L\conv M.
}
$$
\ee
\enlemma
Note that $L\simeq N\hconv N'$ if either $N$ or $N'$ is \ra.

\begin{lem}[{\cite[Corollary 4.1.2]{KKKO18}}] \label{Lem: La on conv}
Let $M$ and $N$ be simple $R$-modules such that $M$ and $ N$ do not strongly commute. 
Assume that either $M$ or $N$ is real and admits an affinization. 
Then, in the Grothendieck group $K(R\gmod)$,  we can write 
$$
[M \conv N] = [M \hconv N]  + [M \sconv N] + \sum_k [S_k]
$$
for some simple $R$-modules $S_k$.
\begin{enumerate} [\rm (i)]
\item If $M$ is real and admits an affinization, then
\begin{align*}
\La(M, M \sconv N), \ \La(M, S_k)  &< \La(M,N) = \La(M,M \hconv N) , \\
\La( M \hconv N, M), \ \La(S_k, M) & < \La(N,M) = \La(N \hconv M, M)  .
\end{align*}
In particular $\de(M,S)<\de(M,N)$ for any simple subquotient $S$ of $M\conv N$.
\item If $N$ is real and admits an affinization, then
\begin{align*}
\La(N, M \hconv N), \ \La(N, S_k)  &< \La(N,M) = \La(N,N \hconv M) , \\
\La( M \sconv N, N), \ \La(S_k, N) & < \La(M,N) = \La(M \hconv N, N)  .
\end{align*}

In particular $ \de(S,N) <\de(M,N)$ for any simple subquotient $S$ of $M\conv N$.

\end{enumerate}
\end{lem}

\Cor\label{cor:abcom}
Let $M$ be a real simple module which admits an affinization.
Let $X$ be an $R$-module in $R\gmod$.
Let $n\in\Z_{>0}$ and assume that
any simple subquotient $S$ of $X$ satisfies
$\de(M,S)\le n$.
Then any simple subquotient $S$ of $M\conv X$ satisfies
$\de(M,S)<n$.
In particular, any simple subquotient of $M^{\circ n}\conv X$
strongly commutes with $M$.
\encor
\Proof
We can reduce to the case where $X$ is  simple.
Then the assertion follows from Lemma \ref{Lem: La on conv}.
\QED

\begin{lem} [{\cite[Theorem 4.1.3]{KKKO18}}] \label{Lem: La comm}
Let $M$ be a simple $R$-module, and  let $x$ be the element of the Grothendieck group $K(R\gmod)$ given by
$$
x = \sum_{b\in B(\infty)} a_b [L_b],
$$
where $L_b$ is the self-dual simple module corresponding to $b$ in the crystal $B(\infty)$ of $U_q^-(\g)$ and $a_b\in \Z[q,q^{-1}]$.
If
\begin{enumerate}[\rm(a)]
\item  $M$ is real and admits an affinization,
\item $x [M] = q^l [M]x$ for some $l \in \Z$,
\end{enumerate}
then we have, for $b\in B(\infty)$ with $a_b \ne 0$,
\begin{enumerate}[\rm(i)]
\item $M$ strongly commutes  with $L_b$,
\item $ \La(M, L_b) = l$.
\end{enumerate}
\end{lem}

\begin{lem} [{\cite[Corollary 4.1.4]{KKKO18}}] \label{Lem: La comm2}
Let $M$ and $N$ be simple $R$-modules. Suppose that
\begin{enumerate}[\rm (i)]
\item  either  $M$ or $N$ is real and admits an affinization,
\item $[M] [N] = q^l [N][M]$  for some $l \in \Z $
in the Grothendieck group $K(R\gmod)$.
\end{enumerate}
Then, $M$ and $N$ strongly commute and $\La(M,N)=-\La(N,M)=-l$.
\end{lem}

\begin{lem} [{\cite[Proposition 3.6, Corollary 3.7]{KKKO15}}] \label{Lem: generalized crystal}
Let $M$ be a simple $R(\beta)$-module. Suppose that $M$ is real and admits an affinization.
Then we have the following.
\begin{enumerate}[\rm (i)]
\item For a simple $R(\beta+\gamma)$-module $L$, the $R(\gamma)$-module $\HOM_{R(\beta+\gamma)}(M\conv R(\gamma), L)$ is either zero or has a simple socle.
\item The map $N \mapsto M\hconv N$ between the set of
the isomorphism classes of simple $R$-modules is injective.
\end{enumerate}
\end{lem}

\begin{prop}\label{prop:simpleinj}
Let $M$, $N$ and $N'$ be simple $R$-modules. Suppose that $M$ is real and admits an affinization.
If there exists non-zero homomorphism $f \in \HOM_R( M\conv N , \allowbreak N' \conv M)$,
then $N\simeq N'$ up to a grading shift.
\end{prop}
\begin{proof}
Suppose that $N \not \simeq N'$. Since $M$ is real, $M \hconv N$ is simple and appears in $\Im(f) \subset  N' \conv M $ as a subquotient,
which tells that $M\hconv N$ also appears in  $ M \conv N' $  as a subquotient.
As $N \not \simeq N'$, Lemma \ref{Lem: generalized crystal} implies $M \hconv N \not\simeq M \hconv N'$. Thus, by Lemma \ref{Lem: La on conv}, we have
$$
\La(M,N) = \La(M, M\hconv N) <\La(M, N') .
$$

On the other hand, taking the dual of $f$, we have a non-zero homomorphism \allowbreak
$
f^\star\cl M \conv N' \to N\conv M
$
up to a grading shift. Thus, by the same argument as above, we obtain
$$
\La(M,N') < \La(M,N),
$$
which is a contradiction. Therefore, $N$ and $N'$ are isomorphic to each other.
\end{proof}

\subsection{Determinantial modules}\

Let $A=\oplus_{n\ge0}A_n$ be a commutative positively graded algebra as in
Section~\ref{Sec: cyclotomic}.
We assume $A_0=\corp$ for simplicity.
Let $\Lambda \in \wlP_+$.
We take a family $a_\Lambda = \{ a_{\La, i}(t_i) \}$ of parameters in Section \ref{Sec: cyclotomic} such that $a_{\La, i}(t_i)\in A[t_i]$ is monic.
Recall that
$R^{a_\La }_A(\la)$ denotes the associated cyclotomic quiver Hecke algebra
for $\la\in\wtl$.

Let $\lambda, \mu \in \weyl \Lambda$ and we choose $w$, $v\in \weyl$ such that $ \lambda  = w\Lambda$ and $ \mu  = v\Lambda$.
Take reduced expressions $\underline{w} = s_{i_1} \cdots s_{i_l}$ and  $\underline{v} = s_{j_1} \cdots s_{j_t}$, and
set $m_k = \langle h_{i_k},   s_{i_{k+1}} \cdots s_{i_l}\Lambda \rangle$ for $k=1,\ldots l$, and $n_k = \langle h_{j_k},   s_{j_{k+1}} \cdots s_{j_t}\Lambda \rangle$ for $k=1, \ldots, t$.
We define 
\begin{align*}
\dM(\lambda, \Lambda;a_\La) & \seteq F_{i_1}^{ a_\La\,( m_1) }  \cdots F_{i_l}^{ a_\La \,( m_l) } A, \\
\dM(\lambda, \mu;a_\La) & \seteq E_{j_1}^{\,*\hskip 0.1em \,( n_1) }  \cdots E_{j_t}^{\,* \hskip 0.1em  \,( n_t) } \dM(\lambda, \La; a_\La),
\end{align*}
where $A$ is regarded as the module over $R^\La_A(\La)\simeq A$.

When $A=\corp$ and $a_\La=\{ a_{\La, i}(t_i) \}$  is given by
$a_{\La,i}(t_i)=t_i^{\ang{h_i,\La}}$, we write
$\dM(\lambda, \Lambda)$ and $\dM(\lambda, \mu)$
instead of $\dM(\lambda, \Lambda;a_\La)$ and $\dM(\lambda, \mu;a_\La)$.

\begin{prop} [{\cite{KK11} and \cite[Lemma 1.7, Proposition 4.2]{KKOP18}}] \label{Prop: dM properties}
Let $\Lambda \in \wlP_+$, and $\lambda, \mu \in \weyl \Lambda$ with $\lambda \wle \mu$.
\begin{enumerate} [\rm (i)]
\item $\dM(\lambda, \mu;a_\La)$ is a finitely generated projective $A$-module.
\item  $\dM(\lambda, \mu)$ is a real simple module.
\item If $ \langle h_i, \lambda \rangle \le 0$ and $s_i \lambda \preceq \mu $, then
$$
\ep_i( \dM(\lambda, \mu;a_\La)) = - \langle h_i, \lambda \rangle \quad \text{and} \quad  E_i^{(- \langle h_i, \lambda \rangle)} \dM(\lambda, \mu;a_\La) \simeq \dM(s_i\lambda, \mu;a_\La).
$$
\item If $ \langle h_i, \mu \rangle \ge 0$ and $ \lambda \preceq s_i \mu $, then
$$
\ep^*_i( \dM(\lambda, \mu;a_\La)) =  \langle h_i, \mu \rangle \quad \text{and} \quad  E_i^{* (\langle h_i, \mu \rangle)} \dM(\lambda, \mu;a_\La) \simeq \dM(\lambda, s_i\mu;a_\La).
$$
\end{enumerate}
\end{prop}
\Proof
(i) follows from \cite{KK11} and
(ii)--(iv) for $\dM(\lambda, \mu)$ follows from
\cite[Lemma 1.7, Proposition 4.2]{KKOP18}.
(iii) and (iv) for $\dM(\lambda, \mu;a_\La)$ follows from
the case $\dM(\lambda, \mu)$ by
noticing that $\corp\tens_A \dM(\lambda, \mu;a_\La)\simeq \dM(\lambda, \mu)$,
and hence
$\dM(\lambda, \mu;a_\La)\simeq0$ if and only if $\dM(\lambda, \mu)\simeq0$.
\QED

\begin{prop} [{\cite[Corollary 4.7]{KKOP18}}] \label{Prop: cupspidal for dM}
Let $\Lambda \in \wlP_+$.
\begin{enumerate}[\rm (i)]
\item
Let $i\in I$ and $v \in \weyl$ such that $vs_i<v$ and $n \seteq \langle h_i,\Lambda\rangle>0$.
Then
$\dM(v\Lambda, vs_i\Lambda)$ is a $\preceq$-cuspidal $R(-n v \alpha_i)$-module, where $\preceq$ is a convex order corresponding to a reduced expression of $v$ given in
{\rm Proposition \ref{Prop: convex preorder for w}}.

\item
Let $\preceq$ be a convex order corresponding to a reduced expression $\underline{w} = s_{i_1} \cdots s_{i_\ell}$ of $w\in \weyl$ given in
{\rm Proposition \ref{Prop: convex preorder for w}}.
Then the sequence obtained from
 $( \dM(w_\ell \Lambda,w_{\ell-1}\Lambda)$, $\ldots$, $\dM(w_1\Lambda,w_0\Lambda) )$
by removing $\dM(w_k\Lambda,w_{k-1}\Lambda)$ such that $w_k\Lambda=w_{k-1}\Lambda$,
is a $\preceq$-cuspidal decomposition of $\dM(w\Lambda,\Lambda)$, where
$w_k=s_{i_1}\cdots s_{i_k}$ for $k=1,\ldots, \ell$.
\end{enumerate}

\end{prop}

By Proposition~\ref{prop:htlt}, we have
\begin{align} \label{Eq: End 1}
\END_{R_A^{a_\La}} ( \dM(\lambda, \La; a_\La) ) \simeq \END_{R_A^{a_\La}(\La)} ( A )
\simeq A.
\end{align}
Hence  $\lc_i(\dM(\lambda, \La; a_\La))$ and $\lep_i(\dM(\lambda, \La; a_\La))$
are regarded as elements of $A[t_i]$.
Note that, by Lemma~\ref{lem:lepi} and $\eqref{Eq: End 1}$, if $ \langle h_i, \la \rangle \le 0$, we have
\begin{align} \label{Eq: lc lep}
  \lc_i(\dM(\lambda, \La; a_\La)) = \lep_i(\dM(\lambda, \La; a_\La)) \lc_i(  \dM(s_i\lambda, \La; a_\La) ) \in A[t_i].
\end{align}

\bigskip
For $i\in I$ and $\mu\in \weyl \Lambda$, we set
\eq
F_{i,\mu}(t_i)\seteq a_{\La,i}(t_i)\cdot\Bigl(\sum_{\nu\in I^{\La-\mu}}\hs{1ex}\prod_{1\le k\le n,\;\nu_k\not=i}
\qQ_{i,\nu_k}(t_i, x_k)e(\nu)\Bigr)\big\vert_{\dM(\mu, \La; a_\La)}
\in A[t_i].
\eneq
Here $n= \Ht(\La-\mu)$.

For $i,j\in I$ such that $i\not=j$ and
 $\mu\in\weyl \Lambda$ such that $n\seteq-\ang{h_j,  \mu }\ge0$, we set
$$
\lpq_{i,j,\, \mu }(t_i) \seteq  \Bigl(\prod_{k=1}^n \qQ_{i,j}(t_i, x_k)\Bigr)
\big\vert _{\,e(j^n,*)\dM( \mu , \Lambda; a_\La)} \in A[t_i].
$$
Thanks to $\eqref{Eq: End 1}$,  $F_{i,\mu}(t_i)$  and $\lpq_{i,j,\,\mu}(t_i)$
can be viewed as quasi-monic polynomials in $A[t_i]$.

\Prop \label{prop: lep for dM}\hfill
\bnum 
\item
For any $\mu\in\weyl\La$, we have
$$F_{i,\mu}(t_i)=\gamma\,\lc_i\bl\dM(\mu, \Lambda; a_\Lambda)\br(t_i)\cdot
\lc_i\bl\dM(s_i\mu, \Lambda; a_\Lambda)\br(t_i)$$
for some $\gamma\in\corp^\times$.
\item For $i,j\in I$ and $\la\in \weyl\La$ such that
$\la\prec s_i\la\prec s_js_i\la$,
set $\mu=s_i\la$ and $\zeta=s_j\mu$.
Then we have
$$
\lep_i( \dM(\la, \La ; a_\La) ) (t_i) =
\bc
\gamma\;\lpq_{i,j,\mu}(t_i) \lep_i(\dM(s_i\zeta, \La ; a_\La) ) (t_i)
 &  \text{if $\langle h_i, \zeta \rangle \ge 0$,} \\
\dfrac{\gamma\;\lpq_{i,j,\mu}(t_i)}{ \lep_i(\dM(\zeta, \La ; a_\La) ) (t_i) }
&  \text{if $\langle h_i, \zeta \rangle \le 0$}
\ec
$$
for some $\gamma\in\corp^\times$.
\ee
\enprop
\begin{proof}

For simplicity, we set
\begin{align*}
\dM(\xi) \seteq  \dM(\xi, \Lambda; a_{\La}) \qtq
\lc_{i, \xi} (t_i) \seteq  \lc_i( \dM(\xi)) (t_i), \ \
\lep_{i, \xi} (t_i) \seteq  \lep_i( \dM(\xi)) (t_i),
\end{align*}
for $\xi \in \weyl \Lambda$.

Let us show (i) by induction on $\Ht(\La-\mu)$.
If $\mu=\La$, it is trivial since
$\lc_i\bl\dM(s_i\La)\br= \lep_i\bl\dM(s_i\La)\br=a_{\La,i}(t_i)$ by 
Proposition \ref{Prop: affinization of L(i^m)}.
Assume that $\mu\not=\La$.

If $\ang{ h_i,\mu}<0$, then we have
$$F_{i,\mu}(t_i)=F_{i,s_i\mu}(t_i)=  \gamma \lc_i\bl\dM( \mu )\br\cdot
\lc_i\bl\dM(s_i \mu)\br \quad \text{ for some $\gamma \in \bR^\times$,}$$
where the last equality follows from induction hypothesis applied to $s_i\mu$.

Hence we may assume that
$\ang{h_i,\mu}\ge0$.
Set $\la=s_i\mu$.
Take $j\in I$ such that $\ang{h_j,\mu}<0$,
and set $\zeta=s_j\mu$.
For such $\mu$ and $\la$, we shall show (i) and (ii).
We set $l \seteq   | \langle h_i, \zeta \rangle|$, $m := \langle h_i, \mu  \rangle$ and $n := \langle h_j, \zeta  \rangle$. 

Note that 
\begin{align} \label{Eq: chi zeta mu}
\lc_{i,\zeta}(t_i)=\lc_{i,\mu}(t_i).
\end{align}

By the definitions, $\lep_{i, \la}(t_i) $  and $\lpq_{i,j,\mu}(t_i) $ are quasi-monic, and
\begin{align} \label{Eq: Ma deg}
 \deg  \lep_{i, \la}(t_i) = 2(\alpha_i, \mu ), \quad   \deg  \lpq_{i,j,\mu}(t_i) = 2(\alpha_i, \mu - \zeta ).
\end{align}

\medskip
\noi
\textbf{(Case\;1):} 
Assume that $l=\langle h_i, \zeta \rangle \ge 0 $.
Note that $ \deg (\lep_{i, \la} (t_i))  =  \deg \bl \lpq_{i,j,\mu}(t_i) \lep_{i,s_i\zeta} (t_i) \br $ by $\eqref{Eq: Ma deg}$.
Thus, by Lemma \ref{Lem: minimal}, in order to see (ii),
it suffices to show that
$$
\lpq_{i,j,  \mu }(x_m) \lep_{i,s_i\zeta}(x_m) u =0 \qquad
\text{for any $u \in e(i^m,*)\dM(\la)$.}
$$

Let us first show
\eq
 \lep_{i,s_i\zeta}(x_1) e(i, *) R_A^{a_\La}(\zeta-\al_i) e(i,*)  = 0.
\label{eq:pp1}
\eneq
If $l=\ang{h_i,\zeta}=0$, then $R_A^{a_\La}(\zeta-\al_i)\simeq0$
and \eqref{eq:pp1} holds.
Assume $ \ang{h_i,\zeta}>0$.
Lemma~\ref{Lem: minimal} implies that
$\lep_{i,s_i\zeta}(x_l) e(i^l, *) \dM(s_i \zeta) = 0$.
Since $ \dM(s_i \zeta)$ is a unique indecomposable projective $R_A^{a_\La}(s_i \zeta)$-module,
$$
\lep_{i,s_i\zeta}(x_l) e(i^l, *) R_A^{a_\La}(s_i \zeta) = 0.
$$
On the other hand, we have by \cite[Lemma 4.14]{R08}
\begin{align} \label{Eq: EFR}
E_i^{\,a_\La\;(l-1)}F_i^{\,a_\La\;(l)}R_A^{a_\La}(\zeta)\simeq
F_i^{a_\La}R_A^{a_\La}(\zeta)\simeq R_A^{a_\La}(\zeta-\al_i) e(i,*)
\end{align}
up to grading shifts.
Since $e(i, *)R_A^{a_\La}(\zeta-\al_i)$ is a direct summand of
$E_i^{\,a_\La\;l}F_i^{\,a_\La\;l}R_A^{a_\La}(\zeta)\simeq e(i^{l},*)R_A^{a_\La}(s_i \zeta)
e(i^l,*) 
$,
we obtain \eqref{eq:pp1}.

\medskip
Then, we have
$$
\lep_{i,s_i\zeta}(x_{m+n}) e(i^{m-1}, j^n,i,*) \dM(\la) = 0,
$$
because  $e(i^{m-1}, j^n,i,*)\dM(\la)$
is an $e(i^{m-1}, j^n)\etens  e(i,*)R_A^{a_\La}(\zeta-\al_i)e(i,*)$-module.
Hence, for any $u \in e(i^m, j^n ,*)\dM(\la)$, we have
\begin{align*}
0 &= (\tau_m \tau_{m+1}\cdots \tau_{m+n-1}) \lep_{i,s_i\zeta}(x_{m+n}) ( \tau_{m+n-1}  \cdots   \tau_{m+1}\tau_m  )u \\
&= \lep_{i,s_i\zeta}(x_m) \Bigl( \prod_{k=1}^n \qQ_{i,j}(x_m, x_{m+k}) \Bigr) u \\
&= \lep_{i,s_i\zeta}(x_m) \lpq_{i,j,\mu}(x_m) u.
\end{align*}

Thus, we have proved (ii).

Let us show (i).
By the induction hypothesis, we have
$ F_{i,\zeta}(t_i) =  \gamma  \lc_{i,\zeta}(t_i) \lc_{i,s_i\zeta}(t_i)$ for some $\gamma \in \bR^\times$.
Since $ F_{i,\mu}(t_i) = \lpq_{i,j,\mu}(t_i) F_{i,\zeta}(t_i)$, 
we have, up to a constant multiple, 
\begin{align*}
F_{i,\mu}(t_i)
&=\lpq_{i,j,\mu}(t_i)F_{i,\zeta}(t_i)
 = \lpq_{i,j,\mu}(t_i)\lc_{i,\zeta}(t_i)\lc_{i,s_i\zeta}(t_i) 
=   \lpq_{i,j,\mu}(t_i)\lc_{i,\zeta}(t_i) \lep_{i, s_i \zeta} (t_i) \lc_{i, \zeta}(t_i) \\
 &  =    \lep_{i, \lambda} (t_i)  \lc_{i,\mu}(t_i)  \lc_{i, \mu}(t_i)
=\lc_{i,s_i\mu}(t_i)\lc_{i,\mu}(t_i). 
\end{align*}
Here the second equality follows from the induction hypothesis applied to
$\zeta$,
and the fourth equality follows from (ii) and \eqref{Eq: chi zeta mu}.

\medskip
\noi
\textbf{(Case\,2):} We assume that $\langle h_i, \zeta \rangle < 0 $.
Since $e(i, *) \dM(\mu) = 0$, we have
$$
0 = (\tau_n \cdots \tau_2\tau_1) (\tau_1 \tau_2 \cdots \tau_n) v = \Bigl(\prod_{k=1}^n \qQ_{i,j}(x_{n+1}, x_k) \Bigr) v = \lpq_{i,j,\mu}(x_{n+1})v
$$
for any $v \in e(j^n, i, *) \dM(\mu)$.
Hence we have
$\lpq_{i,j,\mu}(x_{1})e(i,*)\dM(\zeta)=0$.

Thus,  Lemma \ref{Lem: minimal} implies that
\begin{align} \label{Eq: divided}
\text{$\lep_{i,\zeta}(t_i)$ divides $\lpq_{i,j,\mu}(t_i)$.}
\end{align}

We shall set
$$
S_{i,\mu} (t_i)\seteq  \bl\lc_{i,\mu}(t_i)\br^2.
$$
Then (i) is equivalent to
\eq
F_{i,\mu} (t_i) = \lep_{i,\la}(t_i) S_{i,\mu}(t_i).
\eneq
Now we have, by $\eqref{Eq: chi zeta mu}$ and the induction hypothesis, 
\eq\ba{rcl}
F_{i,\mu}(t_i)&&=\lpq_{i,j,\mu}(t_i)F_{i,\zeta}(t_i)
=\lpq_{i,j,\mu}(t_i)\lc_{i,\zeta}(t_i)\lc_{i,s_i\zeta}(t_i)\\
&&=\lpq_{i,j,\mu}(t_i)\lc_{i,\zeta}(t_i)^2\lep_{i, \zeta}(t_i)^{-1}
=\lpq_{i,j,\mu}(t_i)\lep_{i, \zeta }(t_i)^{-1}
S_{i,\mu}(t_i).
\ea
\label{eq:FS}
\eneq

Note that $\lpq_{i,j,\mu}(t_i)\lep_{i, \zeta}(t_i)^{-1}$
belongs to $A[t_i]$ by $\eqref{Eq: divided}$.

We now shall use results of \cite{KK11} with $\beta=\La- \mu$.
\footnote{
The convention of $F_i, E_i$ in \cite{KK11} is different from ours.
The notation $E_iM$ is $e(*,i)M$ in \cite{KK11} and
$e(i,*)M$ in our paper, etc.}

In the diagram \cite[(5.6)]{KK11}, $K'_1$, $K_0'$ and $F_i^\La E_i^\La R^\La(\beta)$
in the first row vanish.
Hence the diagram reduces (under our notations) to
\eqn
&&0\To q^{2(\al_i,\mu)}A[t_i]\tens_A R^{a_\La}_A(\mu)
\To[\mathscr{A}] A[t_i]\tens_A R^{a_\La}_A(\mu)\To[p] e(i,*)R^{a_\La}_A(\mu-\al_i)e(i,*)
\To0.
\eneqn 
Here $ \mathscr{A}$ is an $ R_A^{a_\La}(\mu)$-bilinear map\footnote{ The homomorphism $\mathscr{A}$ is denoted by $A$ in \cite{KK11}. } 
and
$p$
is given by $p(f(t_i)\tens a)= f (x_1)e(i)\etens a$.

Under the notation in \cite[(5.10)]{KK11},
$\varphi_0(t_i)\seteq  \mathscr{A} (1)$ commutes with $ R_A^{a_\La} (\mu)$.
Hence $\varphi_0(t_i)\in A[t_i]$.
Since $\vphi_0(t_i)$ is the image of $ \mathscr{A}$,
we have $\vphi_0(x_1)e(i,*) R_A^{a_\La} (\mu-\al_i)e(i,*)=0$.
Setting $m=\ang{h_i,\mu}$, 
we can regard $e(i^m,*) \dM(s_i\mu)$  as an $e(i,*) R_A^{a_\La} (\mu-\al_i)e(i,*)$-module,
and hence we have
$\vphi_0(x_m)e(i^m,*) \dM(s_i\mu)=0$.
Then Lemma~\ref{Lem: minimal} implies
that  $\lep_{i, s_i\mu}(t_i)$ divides $\vphi_0(t_i)$. 
Since $\varphi_0(t_i)$ is a quasi-monic polynomial in $t_i$ of degree $\ang{h_i,\mu}$
by \cite[Proposition 5.3]{KK11},
$\varphi_0(t_i)$ coincides with
$\lep_{i, s_i\mu}(t_i)$ up to constant multiples.

By \cite[Lemma 5.5]{KK11}, we have
$$
F_{i,\mu}(t_i) = \gamma\varphi_0(t_i) S_{i,\mu}(t_i) + (
\text{a polynomial in $t_i$ of degree $<\deg_{t_i} S_{i,\mu}(t_i)$}),
$$for some $\gamma\in \corp^\times$.
Hence \eqref{eq:FS} implies that
\begin{align} \label{Eq: fraction}
\lep_{i, s_i\mu}(t_i)=\varphi_0(t_i) =\lpq_{i,j,\mu}(t_i)\lep_{i,\zeta}(t_i)^{-1}
\qt{up to a constant multiple,}
\end{align}
and $F_{i,\mu}(t_i)=\lep_{i, s_i\mu}(t_i) S_{i,\mu}(t_i)$  up to a constant multiple.
\end{proof}

Let $z$ be an indeterminate of degree $d\in 2\Z_{>0}$ such that $2(\al_i,\La)\in d\,\Z$, and set
$$A=\corp[z].$$
We take a family $a_\La=\{a_{\La,i}(t_i)\}$ of parameters
as in Section~\ref{Sec: cyclotomic} such that
$a_{\La,i}(t_i)\in A[t_i]$ is a monic polynomial and
\eq
a_{\La,i}(0)\in\corp^\times z^{2(\al_i,\La)/d}.
\eneq
Then one can consider the cyclotomic quiver Hecke algebra and
the modules  $\dM(\la,\mu;a_\La)$ for $\la, \mu\in \weyl \La$.

\begin{thm} \label{Thm: aff for dM}
For $\la, \mu \in  \weyl \La  $ with $\la \preceq \mu$,
$\dM(\la, \mu; a_\La)$ is an affinization of $\dM(\la,\mu)$.
\end{thm}
\begin{proof} We shall first treat the case $\mu = \La$.
We know already that $\dM(\la, \La ; a_\La)$
is a finitely generated projective $A$-module,
and $\corp\tens_A\dM(\la, \La ; a_\La)\simeq\dM(\la, \La)$.
Hence we have an exact sequence
$$
0 \rightarrow q^{d} \dM(\la, \La; a_{\La}) \buildrel{z}\over\rightarrow \dM(\la, \La; a_{\La}) \rightarrow \dM(\la, \La) \rightarrow  0.
$$
The only remaining condition which we have to check is
that $\gp_i\vert_{\dM(\la, \La; a_{\La})}\not=0$.

We shall use induction on $\Ht(\La-\la)$.
Since  the case $\la= \La$ is trivial, we may assume that
$\la\not=\La$.
If there exists $j\in I$ such that $j \ne i$ and  $\ang{h_j,\la}<0$, then
we have
$\gp_{i}\vert_{\dM(\la, \La; a_\La)}=\gp_{i}\vert_{\dM(s_j\la, \La; a_\La)}$
does not vanish by the induction hypothesis applied to $s_j\la$.
Hence, we may assume that $\langle h_i, \la \rangle < 0$.
We have
$$\gp_{i}\vert_{\dM(\la, \La; a_\La)}=\pm\lc_i\bl\dM(\la, \La; a_\La)\br(0).$$
By Proposition~\ref{prop: lep for dM} (i), 
$\lc_i\bl\dM(\la, \La; a_\La)\br(t_i)$ divides $F_{i,\la}(t_i)$.
Hence we have
$$
F_{i,\la}(0)\in \corp\,\lc_i\bl\dM(\la, \La; a_\La)\br(0).
$$
Hence it is enough to show that
$F_{i,\la}(0)$ does not vanish.
We have
$F_{i,\la}(t_i)= F_{i,s_i\la}(t_i)$ and
$ F_{i,s_i\la}(0)$ is a product of
$a_{i,\La}(0)$, a  power of $\lc_j(\dM(s_i\la, \La; a_\La))(0)$ ($j\in I$)
and an element of $\corp^\times$.
The induction hypothesis applied to $s_i\la$ shows that they does not vanish.
Hence $F_{i,\la}(0)$ does not vanish.

\smallskip
\noi
The case $\mu \prec \La$
follows from Lemma \ref{lem:affei} applied to $\dM(\la,\La; a_\La)$.
\end{proof}

\begin{prop} [\protect{cf.\  \cite[Proposition 10.2.3]{KKKO18}}] \label{Prop: La for D}
 Let $\lambda,\mu \in \wlP_+$, and $s,s',t,t' \in \weyl$ such that $\ell(s's) = \ell(s')+\ell(s)$, $\ell(t't) = \ell(t')+\ell(t)$, $s's\lambda \preceq t'\lambda$, and $s'\mu \preceq t't\mu$.
Then we obtain
\begin{enumerate} [\rm (i)]
\item $\dM(s's \lambda, t'\lambda)$ and $\dM(s'\mu, t't\mu)$ strongly commute,
\item $\Lambda(\dM(s's\lambda, t'\lambda), \dM(s'\mu, t't\mu))  = (s's\lambda + t'\lambda, t't\mu - s'\mu)$,
\item $ \tLa (\dM(s's\lambda, t'\lambda), \dM(s'\mu, t't\mu))  = (  t'\lambda, t't\mu - s' \mu )$,
\item $ \tLa ( \dM(s'\mu, t't\mu), \dM(s's\lambda, t'\lambda) )  = ( s'\mu-t't\mu, s's\lambda )$.
\end{enumerate}
\end{prop}
\begin{proof}
By \cite[Proposition 9.1.6]{KKKO18} (cf.\ \cite[(10.2)]{BZ05}), we have
$$
\mD(s's\lambda, t'\lambda) \mD(s'\mu, t't\mu) = q^{(s's\lambda + t'\lambda,  s'\mu - t't\mu )} \mD(s'\mu, t't\mu) \mD(s's\lambda, t'\lambda).
$$
Thus, the assertions follow from
$[\dM(\la,\mu)]=\mD(\la,\mu)$ (\cite[Proposition 4.1]{KKOP18})
and Lemma \ref{Lem: La comm2}.
\end{proof}

\vskip 2em

\section{Braiders in $R\gmod$}

\subsection{Non-degenerate braiders}

Let $R$ be a quiver Hecke algebra of \emph{arbitrary} type.
Then any module has a structure of braiders in $R\gmod$.
We set $\La = \rlQ_+$.

\begin{prop} \label{Prop: c braider}
Let $\beta \in \rlQ_+$, and let $M$ be a simple $R(\beta)$-module.
We assume that, for every $i\in I$, an $(R(\beta+\alpha_i), R(\alpha_i))$-bilinear homomorphism
$\coR_i\cl  M \conv R(\alpha_i) \rightarrow  q^{\phi_i} R(\alpha_i)  \conv  M$ is given.
Then, there exists a graded braider $(M, \coR_M, \phi)$ in $R\gmod$
such that the morphism
$ \coR_M(R(\alpha_i))\cl  M \conv R(\alpha_i) \rightarrow q^{\phi(\alpha_i)} R(\alpha_i) \conv M$ coincides with $\coR_i$ for every $i\in I$.
Moreover, such a graded braider  is unique up to an isomorphism.
\end{prop}
\begin{proof}
For $\alpha \in \rlQ_+$ with $n=\Ht(\alpha)$, we set
$$
\widetilde{R}(\alpha) \seteq  \bR[x_1^{\pm1}, \ldots, x_n^{\pm1}] \otimes_{\bR[x_1, \ldots, x_n]} R(\alpha).
$$
The algebra structure of $\widetilde{R}(\alpha)$ is given
similarly to $R(\alpha)$ with the use of \eqref{eq:taupol}.

For $i\in I$, we set $\widetilde{R}_i  \seteq  R_i \otimes_{R(\alpha_i)} \widetilde{R}(\alpha_i) \cl  M \conv \widetilde{R}(\alpha_i) \rightarrow \widetilde{R}(\alpha_i)  \conv M $.
Then one has $ \widetilde{R}_i = c_i z_i^{m_i} \RR_{M,  \widetilde{R}(\alpha_i)}$ for some $c_i \in \bR$ and $m_i\in \Z$ by Proposition~\ref{Prop: affinization}~(iii).
Here, $z_i \in \End_{\widetilde{R}(\alpha_i)}(\widetilde{R}(\alpha_i)) $ is the right multiplication by $x_1$.
For any $\gamma \in \rlQ_+$ with $\Ht(\gamma)=m$, we define $R_M(\widetilde{R}(\gamma) )$
by $a_\gamma \RR_{M, \widetilde{R}(\gamma)}$, where
$$
a_\gamma = \sum_{\nu \in I^\gamma} \left( \prod_{k=1}^n c_{\nu_k} x_k^{m_{\nu_k}} e(\nu) \right) \in \widetilde{R}(\gamma).
$$
Since $a_\gamma$ belongs to the center of $\widetilde{R}(\gamma)$,  $R_M(\widetilde{R}(\gamma)) $ is an
$( R(\beta+\gamma), R(\gamma) )$-bilinear homomorphism.
Since it sends $M \conv R(\gamma)$ to $R(\gamma)\conv M$, we obtain an
$( R(\beta+\gamma), R(\gamma) )$-bilinear homomorphism $R_M(R(\gamma)) \cl  M \conv R(\gamma) \rightarrow R(\gamma)\conv M$.
Thus, it is enough to define $R_M(X) = R_M(R(\gamma)) \otimes_{R(\gamma)}X \cl  M \conv X \rightarrow X \conv M $.

The uniqueness is obvious by the construction.
\end{proof}

\begin{df}
Let $M$ be a simple $R$-module. A graded braider $(M, R_M,\phi)$  in $R\gmod$ is \emph{non-degenerate}  if
$R_M(L(i)) \cl  M \conv L(i) \rightarrow q^{\phi(\al_i)}L(i)\conv M$ is a non-zero homomorphism.
\end{df}

By Proposition~\ref{Prop: c braider} and the construction of a graded braider there,
we have the following uniqueness statement.

\Lemma\label{lem: c braider}
Let $M$ be a simple $R$-module. Then we have
\bnum
\item there exists a non-degenerate graded braider
structure $(M,R,\phi)$,
\item if $(M,R',\phi')$ is another non-degenerate graded braider structure,
then $\phi=\phi'$ and there exists a group homomorphism
$c\cl  \rlQ \rightarrow \bR^\times$ such that
$ R'(X) = c(\beta) R(X)$ for any $X \in R(\beta)\gmod$.
\ee
\enlemma
For a non-degenerate graded braider $(M,R,\phi)$, one has
\eq
R_{M}(R(\al_i))=\nR_{M,R(\al_i)}\qt{up to a constant multiple.}
\eneq
Note that we have
\eq
\phi(\al_i)=-\La(M,L(i))\qt{for a non-degenerate graded braider $(M,R,\phi)$.}
\eneq

\Prop\label{prop:effB}
Let $M$ be a simple module which is \ra.
Let $(M, R_M,\phi)$ be a non-degenerate graded braider in $R\gmod$.
Let $\beta\in\rlQ_+$ and let $N$ be a simple $R(\beta)$-module.
Then we have
\bnum
\item
$\La(M,N)\le -\phi(\beta)$.
\item
The morphism $R_M(N)\cl M\circ N\to q^{\phi(\beta)}N\circ M$
vanishes if $\La(M,N)< -\phi(\beta)$,
and does not vanish if $\La(M,N)=-\phi(\beta)$.
\ee
\enprop
\Proof
We argue by induction on $\Ht(\beta)$.
If $\beta=0$, then the assertion is trivial.
If $\beta\not=0$ take $i\in I$ such that $\eps_i(N)\not=0$.
Then set $N_0=\tE_i(N)$. Then we have
$N\simeq L(i)\hconv N_0$.
By Lemma~\ref{lem:additive},
we have
\eq
\La(M,N)\le \La(M,L(i))+\La(M,N_0)\le - \phi(\al_i) -\phi(\beta-\al_i)
= - \phi(\beta).
\label{eq:La(M,N)}
\eneq
Hence we obtain (i).

Let us show (ii).
If $\La(M,N)< - \phi(\beta)$, then
$R_M(N)=0$ since
$\HOM(M\circ N,N\circ M)=\corp\rmat{M,N}$ is concentrated in degree $\La(M,N)$.

Assume that $\La(M,N)= -\phi(\beta)$. Then $\La(M,N_0)= -\phi(\beta-\al_i)$
and $\La(M,N)=\La(M,L(i))+\La(M,N_0)$ by \eqref{eq:La(M,N)}.
Hence
$R_M(N_0)$ is equal to
$\rmat{M,N_0}$ up to a constant multiple,
and $R_M(N)=R_M(N_0)\circ R_M(L(i))$ is equal to
$\rmat{M,N}$ up to a constant multiple by  Lemma~\ref{lem:additive}.
\QED

\Cor\label{cor:phiadd}
Let $M$ be a simple module which is \ra,
and let $(M, R_M,\phi)$ be a non-degenerate graded braider in $R\gmod$.
Let $N$, $N'$ and $L$ be simple modules,
and assume that there is an epimorphism $N\conv N'\epito L$.
If $R_M(L)$ is non-zero,
then so are $R_M(N)$ and $R_{M}(N')$.
Equivalently, if
we have $\La\bl M,   L \br=
\phi\bl\wt(L)\br$, then
$\phi\bl\wt(N)\br=\La(M,N)$
and $\phi\bl\wt(N')\br=\La(M,N')$.
\encor
Note that $L\simeq N\hconv N'$ when either $N$ or $N'$ is \ra.
\Proof
By Lemma~\ref{lem:additive}, we have
$$\La(M,N)+\La(M,N')\ge\La\bl M,  L \br=
\phi\bl\wt(N)\br+\phi\bl \wt(N')\br
\ge \La(M,N)+\La(M,N').$$
Hence we obtain $\phi\bl\wt(N)\br=\La(M,N)$
and $\phi\bl\wt(N')\br=\La(M,N')$.
\QED

As seen below,
the real commutativity of a family of non-degenerate graded braiders can be checked easily.

\Lemma\label{lem:realbrai}
Let $\{S_a\}_{a\in A}$ be a family of simple modules, and
assume that every $S_a$ is \ra.
Let $(S_a,R_{S_a},\phi_a)$ be
the associated non-degenerate graded braiders.
The family $\{(S_a,R_{S_a},\phi_a)\}_{a\in A}$
is a real commuting family of graded braiders
if and only if $\phi_a(\wt(S_b))+\phi_b(\wt(S_a))=0$ for any $a,b\in A$.
\enlemma
\Proof \ 
Assume that $\phi_a(\wt(S_b))+\phi_b(\wt(S_a))=0$ for any $a,b\in A$. 
Since
$0\le \allowbreak 2\de(S_a,S_b)=\La(S_a,S_b)+\La(S_b,S_a)
\le\phi_a(\wt(S_b))+\phi_b(\wt(S_a))=0$,
we have $\phi_a(\wt(S_b))=\La(S_a,S_b)$ and
$\de(S_a,S_b)=0$. Hence $R_{S_a}(S_b)$ is an isomorphism
for any $a,b$.
Thus $\{(S_a,R_{S_a},\phi_a)\}_{a\in A}$
is a real commuting family of graded braiders. 

The converse can be proved similarly. 
\QED

\Exam\label{Ex: constants}
We shall provide an example of real simple modules $C_1$, $C_2$
where the associated non-degenerate braiders form
a real commuting family of braiders,
but any associated non-degenerate braiders
$(C_k, R_{C_k},\phi_k)$ ($k=1,2$)
cannot satisfy
$R_{C_k}(C_k)=\id_{C_k\circ C_k}$ ($k=1,2)$ and
$R_{C_2}(C_1)\circ R_{C_1}(C_2)=\id_{C_1\circ C_2}$ at once.
This is the reason why we should employ \eqref{cond:a}, \eqref{cond:b} 
in Definition~\ref{Def: real cf}
for the definition of real commuting family.

Take $\g=A_2$, $I=\{1,2\}$, $\qQ_{12}(u,v)=c_1u+c_2v$
for $c_1$, $c_2\in\corp^\times$.
Then $\bQ_{12}=c_1$ and $\bQ_{21}=c_2$.
Let $C_1=\ang{12}$, $C_2=\ang{21}$.
Here, the $R(\al_1+\al_2)$-module $C_1$ is the one-dimensional vector space with a basis $\ang{12}$
on which $R(\al_1+\al_2)$ acts by
$e(1,2)\ang{12}=\ang{12}$, $x_k\ang{12}=0$ ($k=1,2$) and
$\tau_1\ang{12}=0$. The $R(\al_1+\al_2)$-module $C_2$ is similarly defined.

Then $R_{C_1}$ is defined as follows for some constants $a_1,a_2\in\corp^\times$.
\eqn
\Bigl(R_{C_1}(L(1)_z)\Bigr)(\ang{12}\etens 1_z)&=&z^{-1}a_1\vphi_2\vphi_1(1_z\etens \ang{12})
=z^{-1}a_1\tau_2(\tau_1z+1)(1_z\etens \ang{12})\\
&=&a_1\tau_2\tau_1(1_z\etens \ang{12}),\\
\Bigl(R_{C_1}(L(2)_z)\Bigr)(\ang{12}\etens 2_z)&=&a_2\vphi_2\vphi_1(2_z\etens \ang{12})
=a_2(\tau_2(x_2-x_3)+1)\tau_1(2_z\etens \ang{12})\\
&=&a_2(z\tau_2+1)\tau_1(2_z\etens \ang{12}).
\eneqn

Similarly $R_{C_2}$ is defined for some constants $b_1,b_2\in\corp^\times$:
\eqn
\Bigl(R_{C_2}(L(1)_z)\Bigr)(\ang{21}\etens 1_z)
&=&b_1(z\tau_2+1)\tau_1(1_z\etens \ang{21}),\\
\Bigl(R_{C_2}(L(2)_z)\Bigr)(\ang{21}\etens 2_z)
&=&b_2\tau_2\tau_1(2_z\etens \ang{21}).
\eneqn

By a calculation, we can see that
\eqn
R_{C_1}(C_1)&&=a_1a_2c_1\id_{C_1\circ C_1},\\
R_{C_2}(C_2)&&=b_1b_2c_2\,\id_{C_2\circ C_2},\\
R_{C_2}(C_1)R_{C_1}(C_2)
&&=-a_1a_2b_1b_2c_1c_2\id_{C_1\circ C_2}.
\eneqn
\enexam

\subsection{Localization}
Let $\catC$ be a full subcategory of $R\gmod$ satisfying
\setlength{\mylength}{\textwidth}
\addtolength{\mylength}{-10ex}
\eq
&&\hs{5ex}\parbox{\mylength}{$\catC$ contains $\one$ and
is stable by taking subquotients, convolutions, extensions and
grading shifts.}
\eneq
Let $\{C_a\}_{a\in A}$ be a family of simple objects of $\catC$.
Let $(C_a,R_{C_a},\phi_a)$ be non-degenerate graded braiders in $R\gmod$.
Assume that every $C_a$ admits  affinization.

Assume that $\{(C_a,R_{C_a},\phi_a)\}_{a\in A}$ is a real commuting family of
graded braiders.

Set $\lG=\Z^{\oplus A}$ and $\lG_{\ge 0}=\Z_{\ge0}^{\oplus A}$,
and for $\al\in\lG_{\ge0}$, let $C^\al\in\catC$ be the object constructed in
Subsection~\ref{sec:graded}.
Let $\tcatC\seteq\catC[ C_a^{\circ-1}\mid a\in A]$ be the localization of $\catC$.
We still denote by $\conv$  the tensor functor of $\tcatC$.
Let $\Phi\cl \catC\to \tcatC$ be the canonical functor. They enjoy the following properties.
\bna
\item $\tcatC$ is an abelian category and
$\Phi$ is exact,
\item $\tC_a\seteq\Phi(C_a)$ is
an invertible central graded braider in $\tcatC$.
\ee
We define
$\tC^\al$ in $\tcatC$ for $\al\in\lG$ such that
$\tC^{\al+\beta}\simeq\tC^\al\conv\tC^{\beta}$ (up to a grading shift)
for $\al,\beta\in \lG$ and
$\tC^\al=\Phi(C^\al)$ for $\al\in\lG_{\ge0}$.
Then
\bna
\item[(c)] any object of $\tcatC$ is isomorphic to
$\tC^\al\conv\Phi(X)$ for some $X\in\catC$ and $\al\in \lG$.
\ee

\medskip

\Prop\label{prop:str}
Let $\Phi\cl \catC\to\tcatC$ be the canonical functor.
Then we have the following properties.
\bnum
\item Let $S$ be a simple object of $\catC$.
The object $\Phi(S)$ is a simple object of $\tcatC$ if
$\phi_a(\wt(S))=\La(C_a,S)$ for any $a\in A$.
Otherwise, $\Phi(S)$ vanishes.
\item Any object of $\tcatC$ has finite length.
\item For simple modules $S$, $S'\in\catC$ such that $\Phi(S)\simeq\Phi(S')\not\simeq0$,\label{item:3}
we have $S\simeq S'$. \label{item:5}
\ee
\enprop
\Proof In the course of the proof, we ignore grading shifts.

\noi
(i)\ Let $S$ be a simple object of $\catC$ such that
$\phi_a(\wt(S))\not=\La(C_a,S)$ for some $a\in A$.
Then $R_{C_a}(S)\cl C_a\conv S\to S\conv C_a$ vanishes.
Hence $\Phi(\id_S)=0$ and $\Phi(S)\simeq0$.

Now assume that $\phi_a(\wt(S))=\La(C_a,S)$ for any $a\in A$.
Then $R_{C^\al}(S)$ does not vanish for any $\al\in\lG_{\ge0}$
by Proposition~\ref{prop:effB}.
It means that $\Phi(\id_S)$ does not vanish.
Hence $\Phi(S)$ does not vanish.
Let us show that any subobject of $\Phi(S)$ is either zero or $\Phi(S)$.
In order to see this,
it is enough to show that for any $\al\in\lG_{\ge0}$ and
a non-zero morphism $f\cl X\to S\conv C^\al$ in $\catC$,
the morphism $\Phi(f)\cl\Phi(X)\to \Phi( S\conv C^\al)$ is an epimorphism.
The image $f(X)$ contains $S\sconv C^\al$.
Hence $X'\to S\sconv C^\al$ is an epimorphism where
$X'=f^{-1}(S\sconv C^\al)\subset X$.
Hence $\Phi(X')\to \Phi(S\sconv C^\al)$ is an epimorphism.
On the other hand, we have  $\phi_{C_\al}(\wt(S))=\La(C^\al,S)$,
and hence $S\sconv C^\al$ coincides with the image of $ R_{C_\al}(S)$.
Since $\Phi\bl R_{C_\al}(S)\br$ is an isomorphism, the morphism
$\Phi(S\sconv C^\al)\to\Phi(  S\conv C^\al )$ is an isomorphism. 
Hence $\Phi(X')\to \Phi(S\conv C^\al)$ is an epimorphism, which implies that
$\Phi(X)\to \Phi( S\conv C^\al)$ is an epimorphism.

\medskip\noi
(ii)\ It is enough to show
that $\Phi(X)$ has a finite length for any $X\in\catC$.
Let $0=X_{  -1  }\subset X_{  0 }\subset\cdots\subset X_n=X$ be a composition series
of $X$. Then (i) implies that
$\Phi(X_k/X_{k-1})$ is simple or vanishes. Hence $\Phi(X)$ has a finite length.

\medskip\noi (iii)\ By $\Phi(S')\not\simeq0$, (i) implies that
 $\phi_{C^\al}(\wt(S'))=\La(C_\al,S')$.
If $\Phi(S)\simeq\Phi(S')$, then there exist $\al\in\lG_{\ge0}$ and
 a non-zero morphism $C^\al\conv S\to q^{-\La(C_\al,S')}S'\conv C^\al$.
Hence Proposition~\ref{prop:simpleinj} implies that $S$ and $S'$ are isomorphic.
\QED

\Prop\label{prop:Z}
Assume that $A$ is a finite set.
Let $\lZ_\catC(C_a\mid a\in A)$ be the full subcategory of $\catC$
consisting of objects $X\in\catC$ such that
$R_{C_a}(X)$ is an isomorphism for any $a\in A$.
Assume that $\al\in\lG$ satisfies $\al-e_a\in\lG_{\ge0}$ for every $a\in A$,
and set $C=C^\al$.
\bnum
\item\label{item:1}
$\lZ_{R\gmod}(C_a\mid a\in A)$ contains $\one$ and
is stable by taking subquotients, convolutions, extensions,
grading shifts.
\item For $X\in\shc$, $X\in\lZ_\catC( C_a \mid a\in A)$ if and only if $R_{C}(X)$ is an isomorphism.
\item Setting $\lZ=\lZ_\catC(C_a \mid a\in A)$,
 the functor $\lZ[C_a^{\circ-1}\mid a\in A]\To\tcatC=\catC[ C_a^{\circ-1}\mid a\in A]$
is an equivalence of categories.
\ee
\enprop
\Proof 
\noi
(i) Let us show that if $X\in\lZ_{R\gmod}(C_a\mid a\in A)$, then any subobject $Y$ of $X$ belongs to
$\lZ_{R\gmod}(C_a\mid a\in A)$. For any $a\in A$, $R_{C_a}(X)$ is a monomorphism.
Hence $R_{C_a}(Y)\cl C_a\conv Y\to Y\conv C_a$ is a monomorphism.
Since $C_a\conv Y$ and $Y\conv C_a$ have same dimension,
$R_{C_a}(Y)$ is an isomorphism.
The other properties of $\lZ_{R\gmod}(C_a\mid a\in A)$ can be proved similarly.

\smallskip
\noi
(ii) It is obvious that $R_C(X)$ is an isomorphism for any $X\in\lZ$.
Conversely, assume that $R_C(X)$ is an isomorphism.
 Then $R_C(X)=\bl R_{C^{\al-e_a}}(X)\tens C_a\br\circ\bl
C^{\al-e_a}\tens R_{C_a}(X)\br$ is injective, and hence
$R_{C_a}(X)$ is injective.
Then we conclude that it is an isomorphism.

\smallskip
\noi
(iii)
We shall show that $\lZ[C^{\tens-1}]\to \catC[C^{\tens-1}]$ is 
essentially surjective (see Lemma~\ref{lem:C}).
It is enough to show that
for any $X\in\catC$, $\Phi(X)\in\tcatC$ is contained in the image of $\lZ[C^{\tens-1}]$.
By Corollary~\ref{cor:abcom}, 
for $n\gg0$ , every simple subquotient of $C^{\circ n}\conv X$
commutes with $C$.
Hence, by replacing $X$ with $C^{\circ n}\conv X$,
we may assume from the beginning that every simple subquotient of $X$
commutes with $C$. 
Then we have $\ell(X)=\ell(C^{\circ n}\conv X)=\ell(X\conv C^{\circ n})$
for any $n\in\Z_{\ge0}$.
Here $\ell$ denote the length function.
Since $\Im\bl R_{C^{\circ (m+n)}}(X)\br=\Im\bl R_{C^{\circ m}}(\Im R_{C^{\circ n}}(X))\br$,
$\ell\bl\Im R_{C^{\circ n}}(X)\br$ is a decreasing sequence in $n$.
Hence, by replacing $X$ with $\Im\bl R_{C^{\circ n}}(X)\br$ for $n\gg0$, we may assume from the beginning that
$\ell\bl\Im R_{C}(X)\br=\ell(X)=\ell(C\conv X)$.
Then we conclude that $R_C(X)\cl C\conv X\to X\conv C$ is an isomorphism.
\QED

\Exam  
Take  $\g=A_2$, $I=\{1,2\}$ and $C=L(1)$.
Set $ \tcatC =  R\gmod[C^{\circ-1}]$ and $\Phi\cl R\gmod \longrightarrow \tcatC$ the canonical functor.
Then $\Phi(L(2))\simeq   C^{\circ{-1}}    \conv   \Phi(L(1,2))$,
$\Phi(L(2,1))\simeq 0$ because 
$-1=\La(C,L(2,1))< \phi_C(\wt(L(2,1))= \allowbreak \La(C,L(2))  +\La(C,L(1)) \allowbreak =1$.
Hence any simple object in $\tcatC$ is isomorphic to
 $  L(1,2)^{\circ m} \conv C^{\circ\; n}$ for some $m\in\Z_{\ge0}$ and $n\in\Z$.
\enexam
\subsection{Left duals}

In this subsection, we take
$\beta\in\prtl$, a simple $R(\beta)$-module $C$
and a real non-degenerate graded braider $(C,R,\phi)$.
We assume that $C$ has an affinization.
We shall consider the localization $\tRm \seteq  R\gmod[C^{\circ-1}]$.

\Prop\label{prop:simpledual}
Let $X$ and $Y$ be simple $R$-modules.
Let
$\eps\cl X\conv Y\epito C$ be an epimorphism.
Then we have
\bnum
\item
$\phi(\wt(X))=\La(C,X)$,  $\phi(\wt(Y))=\La(C,Y)$, and
$\La(C,X)+\La(X,C)=0$,
\item $R_C(X)$ is an isomorphism,
\item
$(C^{\circ-1}\conv X,Y)$ is a dual pair in $\tRm$.
\ee
\enprop
\Proof
Since
$\phi\bl\wt(C)\br=\La(C,C)$,
Corollary~\ref{cor:phiadd} implies that
$\phi\bl\wt(X)\br=\La(C,X)$ and $\phi\bl\wt(Y)\br=\La(C,Y)$.
Since $\phi(\wt(C))=0$, one has
$\La(C,X)+\La(C,Y)=0$.

Taking the dual of $\eps$ we have a monomorphism
$$\eta\cl C\monoto q^n Y\conv X$$
for some $n\in\Z$.
Then by  \cite[Lemma 3.1.5 (i)]{KKKO18},
the compositions
\eq
&X\conv C\To[\eta] q^n X\conv Y\conv X\To[\eps]q^n C\conv X,\\
&C\conv Y\To[\eta] q^n Y\conv X\conv Y\To[\eps]q^n Y\conv C
\eneq
are non-zero.
Hence, we have
$\La(X,C)=-n$ and $\La(C,Y)=-n$.
Hence $\La(C,X)+\La(X,C)=0$.
Thus we obtain (i).
Note that we have
$$n=\phi(\wt(X))=-\phi(\wt(Y)).$$

\noi (ii) follows from (i).

\medskip
\noi
(iii)
Set $\tX=C^{\circ-1}\conv X \in \tRm $.
The isomorphism
$R_C(X)\cl C\conv X\to q^{-n}X\conv C$  in $\tRm$ 
induces $\tX\simeq q^{n}X\conv C^{\circ-1}$.
Let
\eqn
&&\teps\seteq C^{\circ-1}\circ \eps\cl \tX\conv Y
\to \one
\qtq
\teta\seteq\eta\conv  C^{\circ-1}
\cl \one\to  q^{n}Y\conv X\conv C^{-1}\simeq
Y\conv\tX
\eneqn
be morphisms in $\tRm$.

Note that
$$\xymatrix@C=8ex{C\conv\one\ar[r]^-{C\circ \tilde\eta}\ar[d]_{\eta}&C\conv Y\conv \tX\ar[d]_{R_C(Y)\circ\tX}\\
q^n Y\conv X\ar[r]^-{\sim}&q^{n}Y\conv C\conv \tX
}$$
commutes up to a constant multiple.
Indeed, after convoluting $C$ from the right,
the diagram commutes up to a constant multiple,
which follows from
$R_C(C)\in\corp^\times\id_{C\circ C}$ and
the commutativity of the diagram below:
$$\xymatrix@C=13ex{
C\conv \one\conv C\ar[r]_-{C\circ\teta\circ C}\ar[dd]^-{R_C(\one\circ C)}
\ar@/^2pc/[rr]|-{\akew[.5ex] C\circ\eta\akew[.5ex]}
&C\conv Y\conv \tX\conv C
\ar[r]\ar[d]^-{R_C(Y)\circ \tX\circ C}&C\conv Y\conv X\ar[d]^-{R_C(Y)\circ X}\\
& Y\conv C\conv \tX\conv C\ar[dr]^-{\sim}&Y\conv C\conv X\ar[d]^-{R_C(X)}\\
\one\conv C\conv C\ar[rr]^-{\eta\circ C}&&Y\conv X\conv C.
}
$$
Let us show that $(\tilde\eps,\tilde\eta)$ is a quasi-adjunction.

The composition
$$X\conv C\To[X\circ\eta] X\conv Y\conv X\To[\eps\circ X]C\conv X$$
is non-zero by \cite[Lemma 3.1.5 (i)]{KKKO18},
and hence it is an isomorphism since $C\conv X$ and $X\conv C$
are simple by (ii). 
It implies that the composition
$$\tX\conv \one\To[\tX\circ\teta] \tX\conv Y\conv \tX\To[\teps\circ\tX] \one\conv \tX$$
is an isomorphism.

Similarly, since
$C\conv Y\To[\eta\circ Y] Y\conv X\conv Y\To[Y\circ\eps]Y\conv C$
does not vanish,
Proposition~\ref{Prop: affinization} (ii) implies the composition is equal to
$R_C(Y)$ up to a constant multiple.
Hence
it is an isomorphism in $\tRm$.
Hence, the composition
 $$\one\conv Y\To[\teta\circ Y] Y\conv\tX\conv Y\To[Y\circ\teps] Y\conv\one$$
is also an isomorphism.
Indeed, it is enough to show that it is an isomorphism
after operating $C\conv*$,
which follows from the following commutative diagram (up to constant multiples).
$$\xymatrix@C=12ex@R=5ex{
C\conv \one\conv Y\ar[r]^-{C\circ\teta\circ Y}
\ar[dd]^{\bwr}&C\conv Y\conv \tX\conv Y\ar[r]^-{C\circ Y\circ\teps}
\ar[d]^-{R_C(Y)\circ\tX\circ Y} &C\conv Y\conv \one\ar[d]^-{R_C(Y)\circ\one}\\
&q^nY\conv C\conv\tX\conv Y\ar[r]^-{Y\circ C\circ\teps}\ar[d]^{\wr}&q^nY\conv C\conv\one\ar[d]^{\wr}\\
C\conv Y\ar[r]^{\eta\circ Y}&q^nY\conv X\conv Y\ar[r]^{Y\circ\eps}&q^nY\conv C.
}
$$
\QED

Let $i\in I$.
Let $z$ be an indeterminate of degree $(\al_i,\al_i)$
and set $\tL(i)=\corps[z]\ang{i}_z$.
Let $\tL(i)$ endow with the $R(\al_i)$-module structure by
$x_1\ang{i}_z=z\ang{i}_z$.
Then $(\tL(i),z)$ is an affinization of $L(i)$.

For $\ell\in\Z_{>0}$, we set
$L_\ell(i)=\tL(i)/z^\ell\tL(i)$.

\Th\label{th:gendual}
Let $i\in I$ and assume that
$\eps^*_i(C)=1$.
For $\ell\in\Z_{>0}$, we set
$K_\ell=E_i^*(C^{\circ\ell})$.
Then $C^{\circ-\ell}\conv K_\ell$ is a left dual of
$L_\ell(i)$ in $\tRm$.
\enth

The rest of this subsection is devoted to the proof of this theorem.
Set $m=\Ht(\beta)$.

For simplicity, we write
$$L_\ell\seteq  L_\ell(i)\qtq C_\ell\seteq C^{\circ\ell}.$$
Note that $L_1(i) = L(i)$.
We write $\ang{i}_\ell\in L_\ell$ for the image of
$\ang{i}_z\in\tL(i)$. Hence
$L_\ell=R(\al_i)\ang{i}_\ell$.

\medskip
We have an epimorphism
$K_1\conv L_1\epito C$.
Hence, Proposition~\ref{prop:simpledual}
implies that  $C^{\circ-1}\circ K_1$ is a left dual of $L_1$.
The same proposition and Lemma \ref{Lem: M L(i)} imply  
$\phi\bl\wt(L_1)\br=\La(C,L_1)=-(\wt(K_1),\wt(L_1))=-(\beta-\al_i,\al_i)$.
Hence we obtain
\eq
\phi(\al_i)=(\beta-\al_i,\al_i).\label{Eq: computation1}
\eneq

\begin{lem}   \label{Lem: sur and inj}
There exist a surjective $R(\ell \beta)$-module homomorphism
$$
\ep_\ell\cl K_\ell\conv  L_\ell\epito C_\ell
$$
 and an injective $R(\ell \beta)$-module homomorphism
$$
\eta_{\,\ell}\cl C_\ell\monoto q^{\ell \phi(\alpha_i) }   L_\ell \conv K_\ell.
$$
\end{lem}
\Proof
Since $C_\ell$ is a simple module which satisfies
$\eps^*_i(C_\ell)=\ell$, we have
$$(x_{\ell m})^\ell e(\ell\beta-\al_i,i)C_\ell=0$$
by Lemma~\ref{Lem: minimal} and $\lep^*_i(C_\ell)(t)=t^\ell$. 
Hence we obtain a non-zero morphism
$$\eps_\ell\cl K_\ell\conv L_\ell\to C_\ell$$
 by
$u\etens (x_1^k\ang{i}_z)\mapsto (x_{\ell m})^k u$ for $u\in e(*,i)C_\ell$.
It is an epimorphism.

Recall that  $(M \conv N)^\star \simeq q^{ (\wt(M), \wt(N) )} N^\star \conv M^\star$ for
$R$-modules $M$ and $N$.
Taking the dual of $K_\ell\conv L_\ell\epito C_\ell$, we have
\begin{align*}
(C_\ell)^\star  \monoto q^{(\alpha_i, \ell \beta - \alpha_i ) }
L_\ell^\star \conv (K_\ell)^\star
\simeq  q^{(\alpha_i, \ell \beta - \alpha_i ) - (\ell-1)(\alpha_i, \alpha_i) }    L_\ell\conv  (K_\ell)^\star.
\end{align*}
Since $E_i^*$ commutes with the duality $\star$,
setting $(C_\ell)^\star\simeq q^a C_\ell$, we obtain
$$(K_\ell)^\star\simeq \bl E_i^*C_\ell\br^\star\simeq  E_i^*\bl(C_\ell)^\star\br\simeq q^a E_i^*(C_\ell)\simeq q^a K_\ell.$$
Hence, we obtain
\begin{align*}
C_\ell\monoto q^{ \ell ( \beta - \alpha_i, \alpha_i ) }    L_\ell\conv K_\ell=
q^{\ell\phi(\al_i)} L_\ell\conv K_\ell.
\end{align*}
Here the last equality follows from $\eqref{Eq: computation1}$.
\QED

We now define
\begin{align*}
\tK_\ell  \seteq   (C_\ell)^{\circ(-1)} \conv K_\ell\in \tRm.
\end{align*}

Note that
$$\tK_\ell \simeq q^{ - \ell\phi(\ell\beta-\al_i)}
K_\ell\conv(C_\ell)^{\circ(-1)}\simeq
q^{\ell\phi(\al_i)}K_\ell\conv(C_\ell)^{\circ(-1)}.$$
Hence, Lemma \ref{Lem: sur and inj} gives morphisms
\begin{equation} \label{Eq: ep eta}
\begin{aligned}
\tep_{\ell}\cl  \tK_\ell \conv  L_\ell \twoheadrightarrow \triv\qtq\teta_{\ell}\cl   \triv \rightarrowtail  L_\ell \conv  \tK_\ell.
\end{aligned}
\end{equation}

We shall show that
$(\teps_\ell,\teta_\ell)$ is a quasi-adjunction
by induction on $\ell$.

We know already that
the assertion holds when $\ell=1$.
Suppose that $\ell>1$
and $(\tep_{k}, \teta_{k})$ is a quasi-adjunction for $k<\ell$.

\bigskip
Let us take $a,b\in\Z_{>0}$ such that $\ell=a+b$.
By the definition, we have the following short exact sequences:
\begin{equation} \label{Eq: seq for L}
\begin{aligned}
&0 \longrightarrow q_i^{2b} L_{a} \To[\;z^b\;] L_\ell \longrightarrow L_b \longrightarrow 0.
\end{aligned}
\end{equation}
Here, $q_i=q^{(\al_,\al_i)/2}$, and $q_i^{2b} L_{a} \To[z^b] L_\ell$ is given by
$\ang{i}_{a}\mapsto z^b\ang{i}_{\ell}$.

For $M \in R(\beta)\Mod$ and $N \in R(\gamma)\Mod$,
there is a short exact sequence (see \cite[Proposition 2.18]{KL09}) 
\begin{align*}
 0 \longrightarrow  M \conv (E_i^*N) \longrightarrow E_i^*(M \conv N) \longrightarrow q^{-( \alpha_i, \gamma)}( E_i^*M) \conv  N  \longrightarrow 0.
\end{align*}
Applying this to $M=C_{a}$ and $N=C_b$, we obtain the exact sequence
\begin{align}
0 \longrightarrow C_{a} \conv K_b \To[\;f\;] K_\ell
 \To[\;g\;] q^{-b(\beta,\alpha_i)} K_{a}\conv C_b  \longrightarrow 0.
\label{eq:extabl}
\end{align}

By $\eqref{Eq: computation1}$, we have
\begin{align*}
K_{a}\conv C_b \simeq q^{-b\phi(a\beta-\alpha_i  ) } C_b\conv K_{a}
= q^{ b(\beta-\al_i, \alpha_i ) }  C_b\conv K_{a}
\end{align*}
in $\tRm$,
which yields an exact sequence in $\tRm$:
\begin{equation} \label{Eq: seq1}
\begin{aligned}
0\To C_{a} \conv K_b \To K_\ell
 \To  q^{-b(\al_i\al_i)}  C_b\conv K_{a}\longrightarrow 0.
\end{aligned}
\end{equation}

{}From $\eqref{Eq: seq1}$, we obtain an exact sequence in $\tRm$
\begin{equation} \label{Eq: seq for tK}
\begin{aligned}
0 \longrightarrow\tK_{b}  \longrightarrow \tK_{\ell} \longrightarrow q_i^{-2b} \tK_{a} \longrightarrow 0.
\end{aligned}
\end{equation}

In order to see that $(\tep_{\ell}, \teta_{\ell})$ is a quasi-adjunction,
we shall apply Lemma \ref{Lem: ep eta ind} to
the short exact sequences \eqref{Eq: seq for L} and \eqref{Eq: seq for tK}.
Hence it is enough to show that the following two diagrams in $\tRm$ commute up to constant
multiples:
\eq
&&\ba{c}\xymatrix{
 \tK_b \conv L_\ell  \ar[rr] \ar[d] &&  \tK_b \conv L_{b}   \ar[d]^{\tep_{b}} \\
 \tK_\ell \conv L_\ell  \ar[rr]^{\tep_\ell} && \triv \\
 q_i^{2b} \tK_\ell \conv L_{a}  \  \ar[rr] \ar[u]
&&\tK_{a} \conv L_{a}  \ar[u]_{\tep_{a}},
 }\ea
\label{diagram:eps}
\eneq
and
\eq
&&\ba{c}\xymatrix{
 q_i^{-2b} L_\ell\conv\tK_a   && L_a\conv \tK_a \ar[ll] \\
L_\ell\conv\tK_\ell \ar[u]\ar[d] && \triv\ar[ll]_{\teta_{\ell}}\ar[u]_{\teta_a} \ar[d]^{\teta_{b}}\\
  L_{b}\conv\tK_\ell
&&L_{b}\conv \tK_{b}\;.\ar[ll]
 }\ea\label{diagram:eta}
\eneq
Since the diagram \eqref{diagram:eta} is the dual of
\eqref{diagram:eps} after the exchange of $a$ and $b$,
it is enough to show the commutativity (up to constant
multiples) of \eqref{diagram:eps}.
Then it is reduced to the commutativity (up to constant multiples)
of the following two diagrams in $R\gmod$:
\eq
&&\ba{c}\xymatrix@C=7ex@R=6ex{
C_{a}\conv K_b\conv L_\ell\ar[r]\ar[d]^{f}&C_{a}\conv K_b\conv L_b\ar[d]^{\eps_b}\\
K_\ell\conv L_\ell\ar[r]^{\eps_\ell}&C_\ell
}\ea\label{dia:A}
\eneq
and
\eq
&&\ba{c}
\xymatrix@C=7ex{
q_i^{2b}K_\ell\conv L_{a}\ar[r]^-{g}\ar[dd]^{z^b}
&q^{-b(\beta-\al_i,\al_i)}K_{a}\conv C_b\conv L_{a}\ar[d]^{R_{C_b}(L_{a})}\\
&K_{a}\conv L_{a}\conv C_b\ar[d]^{\eps_{a}}\\
K_\ell\conv L_\ell\ar[r]^{\eps_\ell}&C_\ell.
}\ea\label{dia:B}
\eneq
Let us first show the commutativity of \eqref{dia:A}.

The module $C_{a}\conv K_b\conv L_\ell$ is generated
by $s\seteq u\etens v\etens\ang{i}_\ell$ as an $R(\ell\beta)$-module.
Here, $u\in C_{a}$ and $v\in e(*,i)C_b$.
Then the element $s$ is sent by
$C_{a}\conv K_b\conv L_\ell\to C_{a}\conv K_b\conv L_b\to C_\ell$
as follows:
\eqn
&&u\etens v\etens\ang{i}_\ell\longmapsto u\etens (v\etens\ang{i}_{b})
\longmapsto u\etens v.
\eneqn
On the other hand, by
$C_{a}\conv K_b\conv L_\ell\to K_\ell\conv L_\ell\to C_\ell$,
the element $s$ is sent as follows:
\eqn
&&u\etens v\etens\ang{i}_\ell\longmapsto (u\etens v)\etens\ang{i}_{\ell}
\longmapsto u\etens v.
\eneqn
Since these two images coincide,  the diagram \eqref{dia:A} commutes.

\medskip
Now, let us show the commutativity (up to constant multiples)
of \eqref{dia:B}.

To do it, we need an explicit construction of the morphism $g$.
We regard $E^{ * (b)}_iC_b$ as the subspace
$$T\seteq\set{ v\in e\bl b(\beta-\al_i), i^b\br C_b}%
{\text{$\bl e\bl b(\beta-\al_i)\br)\etens x_k\br v=0$ for $1\le k\le
b$}}\subset C_b,$$
because $e(*,i^b)C_b\simeq E^{ * (b)}_iC_b\etens L(i^b)$
as an $R(b(\beta-\al_i))\tens R(b\al_i)$-module.
Then,
we have an epimorphism
$$T\conv L(i)^{\circ b}\epito C_b$$
given by $v\etens \ang{i}^{\circ b}\mapsto v$.

Note that $K_\ell$ is generated (as an $R(\ell\beta-\al_i)$-module) by
$C_{a}\conv K_b\subset K_\ell$ and
the elements $\bl e(a\beta-\al_i)\etens H\br(u\etens v)$
with $u\in e(*,i)C_{a}$ and $v\in T$.
Here $H=\tau_{bm}\cdots\tau_1e(i,b\beta) \allowbreak \in R(b\beta+\al_i)$.
Note that $e(a\beta-\al_i)\etens H\in R(\ell\beta)$.
Then $g\cl K_\ell\to q^{-b(\beta,\alpha_i)} K_{a}\conv C_b$ is given by
$$\text{$g(C_{a}\conv K_b)=0$ and
$g\bl\bl e(a\beta-\al_i)\etens H\br(u\etens v)\br=u\etens v$.}$$

Let $F\seteq \eps_\ell \circ  z^b $ and $G\seteq\eps_{a}\circ R_{C_b}(L_{a})\circ g$
be the two compositions $  q_i^{2b}K_\ell\conv L_a\to C_\ell$ appearing in \eqref{dia:B}.
Then it is easy to see that
$C_{a}\conv K_b\conv L_a\subset K_\ell\conv L_a$ is sent to
zero by $G$,
and $F(C_{a}\conv K_b\conv L_a)=0$
follows from
$x_{bm}^be(*,i)\,C_b=0$.

Hence in order to see the commutativity of \eqref{dia:B},
it is enough to show that
$F$ and $G$ send
$w\seteq\bl e(a\beta-\al_i)\etens H\br(u\etens v) \etens \lan  i \ran_a $
to the same element in $C_\ell$ up to a constant multiple.

Let us first calculate $F(w)$.
We have
\eq
&&\ba{rl}
F(w)&=\eps_\ell\Bigl(\bl e(a\beta-\al_i)\etens H\br(u\etens v)
\etens z^b\ang{i}_{\ell}\Bigr)\\[1ex]
&=x_{\ell m}^b\bl e(a\beta-\al_i)\etens H\br \eps_\ell
\bl(u\etens v)\etens \ang{i}_{\ell}\br\\[1ex]
&=\bl e(a\beta-\al_i)\etens x_{bm+1}^b\br
\bl e(a\beta-\al_i)\etens H\br(u\etens v)\\
&=\bl e(a\beta-\al_i)\etens (x_{bm+1}^b\tau_{bm}\cdots\tau_1)\br(u\etens v).
\ea
\label{eq:F(w)}
\eneq

\smallskip
Now, we shall calculate $G(w)$, which is given by
\eq
G(w)=\eps_{a}\Bigl(u\etens  R_{C_b}(L_{a})
(v\etens \ang{i}_{a})\Bigr).\eneq

\Lemma
For $v\in T$, we have \ro up to a constant multiple\rf
$$R_{C_b}(\tL(i))(v\etens\ang{i}_z)
=x_{bm+1}^b\tau_{bm}\cdots\tau_1\bl \ang{i}_z\etens v\br.$$
\enlemma
\Proof Since $(C, R_C,\phi)$ is non-degenerate,
we have $R_C(\tL(i))=\nR_{C,\tL(i)}$ up to a constant multiple.
Since $\La(C_b,L(i))=\La(T,L(i))+\La(L(i)^{\circ b},L(i))$,
the  epimorphism $T\conv L(i)^{\circ b}\epito C_b$ implies
that the diagram
\eq&&\hs{3ex}\xymatrix@C=12ex{
T\conv L(i)^{\circ b}\conv \tL(i)\ar[r]^-{\nR_{L(i)^{\circ b},\tL(i)}}\ar[d]
&T\conv \tL(i)\conv L(i)^{\circ b}\ar[r]^-{\nR_{T,\tL(i)}}
&\tL(i)\conv T\conv L(i)^{\circ b}\ar[d]\\
C_b\conv \tL(i)\ar[rr]^{\nR_{C_b,\tL(i)}}&&\tL(i)\conv C_b
}\label{di:rcl}
\eneq
commutes. We have
\eqn
\nR_{L(i)^{\circ b},\tL(i)}(\ang{i}^{\circ b}\etens\ang{i}_z)
&&=\RR_{L(i)^{\circ b},\tL(i)}(\ang{i}^{\circ b}\etens\ang{i}_z)
=\vphi_b\cdots\vphi_1(\ang{i}_z\etens\ang{i}^{\circ b}).
\eneqn
Now we shall show
\eq
\vphi_b\cdots\vphi_1(\ang{i}_z\etens\ang{i}^{\circ b})
=x_{b+1}^b\tau_b\cdots\tau_1(\ang{i}_z\etens\ang{i}^{\circ b})
\label{eq:phib}
\eneq
by induction on $b$.
Assuming for $b-1$, we have
\eqn
\vphi_b\cdots\vphi_1(\ang{i}_z\etens\ang{i}^{\circ b})
&&=\vphi_bx_{b}^{b-1}\tau_{b-1}\cdots\tau_1(\ang{i}_z\etens\ang{i}^{\circ b})\\
&&=x_{b+1}^{b-1}\vphi_b\tau_{b-1}\cdots\tau_1(\ang{i}_z\etens\ang{i}^{\circ b})\\
&&=x_{b+1}^{b-1}(x_{b+1}\tau_b-\tau_bx_{b+1})
\tau_{b-1}\cdots\tau_1(\ang{i}_z\etens\ang{i}^{\circ b})\\
&&=x_{b+1}^{b}\tau_b\tau_{b-1}\cdots\tau_1(\ang{i}_z\etens\ang{i}^{\circ b}).
\eneqn
Thus the induction proceeds, and we have obtained \eqref{eq:phib}.

\medskip
Hence we obtained
$$\nR_{L(i)^{\circ b},\tL(i)}(\ang{i}^{\circ b}\etens\ang{i}_z)
=x_{b+1}^{b}\tau_b\tau_{b-1}\cdots\tau_1(\ang{i}_z\etens\ang{i}^{\circ b}).$$

On the other hand, since $\al_i\not\in W^*(T)$, we have
by \cite[Proposition 2.12]{KKOP18}
$$\nR_{T,\tL(i)}(y\etens \ang{i}_z)
=\tau_{b(m-1)}\cdots\tau_1(\ang{i}_z\etens y)$$
for $y\in T$.
Hence we have
\eqn
&&\hs{-1ex}\bl\nR_{T,\tL(i)}\conv L(i)^{\circ b}\br\circ \bl T\conv\nR_{L(i)^{\circ b},\tL(i)}\br(y\etens
\ang{i}^{\circ b}\etens\ang{i}_z)\\
&&\hs{2ex}=\bl\nR_{T,\tL(i)}\conv L(i)^{\circ b}\br\Bigl( y\etens
\bl x_{b+1}^b\tau_b \cdots\tau_1(\ang{i}_z\etens\ang{i}^{\circ b})\br\Bigr)\\
&&\hs{2ex}=x_{bm+1}^b\tau_{bm}\cdots\tau_{bm-b+1}\Bigl(\bl\nR_{T,\tL(i)}
( y\etens\ang{i}_z)\br\etens\ang{i}^{\circ b}\Bigr)\\
&&\hs{2ex}
=x_{bm+1}^b\tau_{bm}\cdots\tau_1(\ang{i}_z\etens y\etens\ang{i}^{\circ b}).
\eneqn
Together with the commutativity of
diagram \eqref{di:rcl},  we obtain the desired result.
\QED

Now, we resume the calculation of $G(w)$.
By the lemma above, we obtain (up to a constant multiple)
\eqn
G(w)&&=\eps_{a}\Bigl(u\etens
(x_{bm+1}^b\tau_{bm}\cdots\tau_1)(\ang{i}_{a}\etens v)\Bigr)\\
&&=\Bigl( e(a\beta-\al_i)\etens(x_{bm+1}^b\tau_{bm}\cdots\tau_1)\Bigr)
\Bigr(\bl\eps_{a}(u\etens \ang{i}_{a})\br\etens v\Bigr)\\
&&=\Bigl( e(a\beta-\al_i)\etens(x_{bm+1}^b\tau_{bm}\cdots\tau_1)\Bigr)
(u\etens  v).
\eneqn

Comparing it with \eqref{eq:F(w)},
we obtain $F(w)=G(w)$ up to a constant multiple.
Thus the induction proceeds, and we conclude that
$(\teps_\ell,\teta_\ell)$ is a quasi-adjunction for any $\ell$.
It completes the proof of Theorem~\ref{th:gendual}.

\vskip 2em

\section{Localization of $\catC_w$}
We fix an element $w\in \weyl$ throughout this section.
\subsection{Graded braiders associated with a Weyl group element}

For $\La \in \wlP_+$, we define
\begin{align*}
\dC_{w, \La} \seteq  \dM(w \Lambda, \Lambda ).\qquad
\end{align*}
If no confusion arises, we simply write $\dC_\La$ 
for $\dC_{w, \La}$. 
For $i\in I$, we simply set
$$
\dC_i \seteq  \dC_{\La_i}. 
$$
Note that, for $\La = \sum_{k=1}^n \Lambda_{i_k}$, $ \dC_\La$ is a self-dual real simple $R$-module which is isomorphic to  $\dC_{\La_{i_1}} \conv \dC_{\La_{i_2}} \conv \cdots \conv \dC_{\La_{i_n}}  $ up to a grading shift \cite[Proposition 4.2]{KKOP18}.
We shall show that $ \dC_i$ is a real graded braider in the category $R\gmod$
with the non-degenerate graded braider structure.

For $i\in I$, we set
$$
\la_i \seteq
\left\{
\begin{array}{ll}
w \La_i + \La_i& \text{ if } w \La_i \ne \La_i,  \\
0& \text{ if } w \La_i  =  \La_i.
\end{array}
\right.
$$
Note that $\dC_i \simeq \triv$ if and only if $ w\La_i = \La_i$. By Lemma \ref{Lem: M L(i)}, we have
\begin{align}
\La(\dC_i, L(j)) = (\alpha_j, \alpha_j) \ep_j^*(\dC_i) + (\alpha_j, w\La_i - \La_i) = (\la_i, \alpha_j).
\end{align}

Applying Lemma~\ref{lem: c braider}, we have the non-degenerate
graded braider
$(\dC_i, \coR_{\dC_i}, \dphi_i)$ for $i\in I$.

\begin{prop} \label{Prop: canonical braiders}
The family $(\dC_i, \coR_{\dC_i}, \dphi_i)_{ i\in I}$
is a real commuting family of graded braiders in $R\gmod$ with the non-degenerate graded braider structure  and
$$
\dphi_i(\beta)  = - ( \la_i , \beta) \qquad \text{for any $\beta\in \rtl$.}
$$
\end{prop}

\Proof
The reality follows from Lemma~\ref{lem:realbrai} and
$(\la_i,\wt(\dC_j))+(\la_j,\wt(\dC_i))=0$.
\QED

\begin{thm} \label{Thm: R Ci iso}
For $i\in I$ and any $R(\beta)$-module $N$ in $\catC_w$,
$ \coR_{\dC_i}(N)$ is an isomorphism.
\end{thm}
\begin{proof}
We take a reduced expression $ \underline{w} = s_{i_1} \cdots s_{i_\ell} $ of $w$, and set
$w_0 = \id$, $w_k = s_{i_1}\cdots s_{i_k}$, $ \beta_k \seteq  w_{k-1} \alpha_{i_k} $ and
$\dS_k \seteq  \dM(w_k \La_{i_k}, w_{k-1} \La_{i_k} )$
for $k=1, \ldots, \ell$. Then $\dS_k$ is the cuspidal $R(\beta_k)$-module corresponding to the positive root $\beta_k$
by Proposition \ref{Prop: cupspidal for dM}.
Setting $\la = \La_i$, $\mu = \La_{i_k}$,  $s' = w_k$, $s = w_k w$, $ t' = \id$ and $ t = w_{k-1}$ in Proposition \ref{Prop: La for D},
we have that $\dC_i$ and $ \dS_k$ strongly commute and
\begin{align*} \label{Eq: deg for R Rnorm}
\La( \dC_i, \dS_k ) = (  \la_i, \beta_{k}) = -\dphi_i(\beta_{k}).
\end{align*}
Hence Proposition~\ref{prop:effB} implies that
$R_{\dC_i}(\dS_k)$ is an isomorphism.
Let $\catC$ be the full subcategory of
$\catC_w$ consisting of objects $M$ such that $R_{\dC_i}(M)$ is an isomorphism.
Then $\catC$ has the following proprieties:
\bna \item
$\catC$ is stable by taking extensions, kernels, cokernels, and convolutions,
\item $\catC$ contains all the $\dS_k$'s.
\ee
Since $\catC_w$ is the smallest subcategory of $\catC_w$
which has these properties, $\catC=\catC_w$.
\end{proof}

We set $\lG \seteq  \Z^{\oplus I}$ and define a $\Z$-bilinear map $\gH\cl \lG \times \lG \rightarrow \Z$ by
$$
\gH( e_i, e_j ) \seteq  - \tLa(\dC_i, \dC_j) = (\La_i, w \La_j - \La_j).
$$
Then, for $i,j\in I$, we have
\begin{align*}
\dphi_i( -  \wt( \dC_j )) &= ( w\La_i + \La_i, w\La_j - \La_j)
 = (   \La_i, w\La_j) - (  \La_j, w\La_i)\\
&= \gH(e_i, e_j) - \gH(e_j, e_i).
\end{align*}
Thanks to Theorem \ref{Thm: graded localization} and
Proposition~\ref{Prop: canonical braiders}, we have the localization of $\catC_w$
by the non-degenerate graded braiders $\dC_i$, which we denote by
$$
\lRg \seteq  \catC_w [ \dC_i^{\conv -1 } \mid i \in I ].
$$
By the choice above of $ \gH$, for $ \alpha = \sum_{i\in I} m_i e_i \in  \lG_{\ge0}$, we have
$$
 \dC_\La  \simeq (\triv, \alpha) , \qquad \dC_\La^{-1}  \simeq q^{\gH(\al,\al)}(\triv, -\alpha)
$$
where $ \La = \sum_{i\in I} m_i \La_i$ by Lemma \ref{Lem: self-dual}.
Thus, for $\La = \sum_{i\in I} a_i \La_i \in \wlP$, 
we simply write $\dC_\La := ( \triv,  \sum_{i\in I} a_i e_i )\in \lRg$ .

Therefore, there exists  a monoidal functor $\Phi\cl \catC_w \to \lRg$
such that
\bnum
\item the objects $\Phi(\dC_i)$ are invertible in $\lRg$,
\item  $\Phi\bl R_{\dC_\La}(X) \br$ is an isomorphism for any $\La\in\wlP_+$ and $X\in \Cw$, 
\item for any simple object $S$ of $\Cw$,
the object $\Phi(S)$ is simple in $\tcatC_w$,\label{item: iii}
 \item
every simple object of $\lRg$ is isomorphic to $\dC_{ \La } \circ \Phi(S) $ for some simple object $S$ of $\catC_w$ and $ \La \in  \wlP $,
\item
 for two simple objects $S$ and $S'$ in $\catC_w$ and $\La,\La'\in \wlP$,
$\dC_{\La} \conv  \Phi(S) \simeq \dC_{\La'} \conv  \Phi(S') $ in $\lRg$
if and only if $q^{\gH(\La,\mu)} \dC_{\La+\mu} \conv S  \simeq  q^{\gH(\La',\mu)} \dC_{\La'+\mu} \conv  S'  $ in
$\catC_w$
for some $\mu\in \wlP$
such that $\La+\mu, \La'+\mu \in \wlP_+$, 
\item the category $\lRg$ is abelian and every objects has
 finite length. \label{item: vi}
\end{enumerate}
Here, \eqref{item: iii}--\eqref{item: vi} follow from Proposition~\ref{prop:str}.
 The grading shift functor $q$ and the contravariant functor $M \mapsto M^\star$ on $\catC_w$   are extended to $\lRg$.
Moreover we have
$$(M\conv N)^\star \simeq q^{(\wt(M),\wt(N))}N^\star\conv M^\star
\qt{for $M$, $N\in\tCw$.}$$
Since every simple object of $\lRg$ is isomorphic to $\dC^{\al} \circ \Phi(S) $ for some simple module $S$ of $\catC_w$,
one can prove the following.
\Lemma
For any simple module $M\in\lRg$, there exists a unique $n\in\Z$ such that
$q^nM$ is self-dual.
\enlemma

 \begin{cor}
The Grothendieck ring $K(\lRg)$ of $\lRg$ is isomorphic to
 the left ring of quotients of the ring $K(\catC_w)$ with respect to the  multiplicative subset
$$S\seteq \Bigl\{q^k\prod_{i\in I}[\dC_i]^{a_i}\;\big|\;
k\in \Z, \ (a_i)_{i \in I} \in \Z_{\ge 0}^{I} \Bigr\}.$$
\end{cor}
\begin{proof}
Note that the set
$S$ is a left denominator set in $K(\catC_w)$ since $[\dC_i]$ is $q$-commuting with any class of simple object.
Since $S$ is invertible in $K(\lRg)$, we have
an algebra map
$S^{-1}K(\catC_w)\to K(\lRg)$.
It is evidently surjective.
In order to see injectivity, it is enough to show that
the left multiplication on $K(\catC_w)$ by $[\dC_i]$ is injective.
It is well known.
\end{proof}

\begin{rem}
 The ring $\Q(q^{1/2}) \otimes_{\Z[q^{\pm1}]}K(\lRg)$, which is a localization of the quantum cluster algebra $\Q(q^{1/2}) \otimes_{\Z[q^{\pm1}]}K(\catC_w)$,
is the subalgebra of the skew field of fractions $K$ of an initial quantum torus $\Q(q^{1/2})[X_i^{\pm1} \ |   \ i \in \K]$ generated by all the 
cluster variables 
and the inverses of the frozen variables. Thus it is the quantum cluster algebra in the sense of \cite{BZ05}.
\end{rem}

\subsection{Rank-one case}
Let us briefly describe the localization $\tcatC_w$ without proof
when $\g$ is of rank one.
Set $I=\{i\}$. Let $\tR$ be the localization of $R\gmod$ by $L(i)$.
Note that $R\gmod$ is a symmetric monoidal category,
since $\RR_{M,N}$ (see Section~\ref{sec:R-matrices}) is always an isomorphism.
We have the decomposition
$$\tR=\soplus_{n\in \Z}\tR_0\conv L(i)^{\circ n}.$$
Here $\tR_0$ is the full subcategory of $\tR$
consisting of objects of weight $0$. In another words,
it is the full subcategory consisting of the objects of the form
$M\conv L(i)^{\circ - n}$ for $n\in \Z_{>0}$ and $M\in R(n\al_i)\gmod$.
The monoidal category $\tR_0$ is described as follows.

\smallskip
Let $\lS$ be the ring of symmetric polynomials.
Hence, $\lS=\soplus_{d\in\Z_{\ge0}}\lS_d$ is positively graded and
$$\lS_d=\prolim_{n\in\Z_{>0}}\corp[x_1,\ldots,x_n]^{\sym_n}_d,$$
where $\corp[x_1,\ldots,x_n]^{\sym_n}_d$ is the space of symmetric polynomials
of $n$ variables with degree $d$.
Then $$\lS=\corps[e_k\mid k\in\Z_{>0}],$$
where $e_k$ is the elementary symmetric polynomial of degree $k$.
Let us take the generating function $$E(t)=\sum_{k=0}^\infty e_k t^k,$$
where $e_0=1$.

The ring $\lS$ has a well known structure of (co-commutative) Hopf algebra by
the coproduct defined by
$$\Delta(E(t))=E(t)\tens E(t).$$
Namely,
$$\Delta(e_n)=\sum_{k=0}^ne_k\tens e_{n-k}.$$
Then its antipode $S$ is given by
$$S\bl E(t)\br=E(t)^{-1}$$
and its counit is given by
$$\eps(E(t))=1.$$
Let us denote by  $\lS\gmod$  the category of graded $\lS$-modules which are
finite-dimensional over $\corp$. Then,  $\lS\gmod$
has a structure of monoidal category.
Then, $\tR_0$ is equivalent to the category
$\lS\gmod$ as a monoidal category.

By this correspondence, the full subcategory
$R(n\al_i)\gmod\tens L(i)^{\circ-n}\subset\tR_0$
is precisely equivalent to the full subcategory
of $\lS\gmod$ consisting of modules $M$ such that
$e_k\vert_M=0$ for $k>n$, namely
$\corp[x_1,\ldots,x_n]^{\sym_n}\gmod$.
This equivalence is obtained by the Morita equivalence
through
$\corp[x_1,\ldots,x_n]^{\sym_n}\simeq\END_{R(n\al_i)}\bl P(i^n)\br$.
Here $P(i^n)$ is a unique indecomposable projective $R(n\al_i)$-module.

\subsection{Left rigidity of $\tcatC_w$}\

First note the following lemma which immediately follows from
 Proposition~\ref{prop:simpledual}. 

\Lemma \label{lem:dualpairs} Let $M$ be a simple module in $\catC_w$.
Then $M$ has a right dual  {\rm(}respectively, left dual{\rm)} in $\lRg$ if and only if there exists a surjective homomorphism $M\conv X \epito \dC_\La$
 {\rm(}respectively, $X \conv M \epito \dC_\La${\rm)} for some module $X$ in $\catC_w$ and for some $\La \in \wlP_+$.
\enlemma

We now consider duals of objects in the localization $\lRg$.
For an object $M$ in $\lRg$,
we denote by $\rD(M)$  the right dual of $M$ and
 by $\lD(M)$  the left dual of  $M$, if they exist.

\begin{thm} \label{thm:rightdual}
Every simple object $M$ in $\lRg$ has a right dual.
\end{thm}
\begin{proof}
For simplicity, we ignore the grading in the course of the proof.
We will proceed by induction on $l = \ell(w)$.
We take a reduced expression $ \underline{w} = s_{i_1} \cdots s_{i_l} $ of $w$, and set
$w_0 = \id$ and $w_{< k+1} = w_{\le k} = s_{i_1}\cdots s_{i_k}$.
Note that
the full subcategory consisting of objects in $\catC_w$ which have right duals
in $\lRg$  is closed by taking the kernels,
the cokernels and the convolution products.
Note also that any simple object is contained in the smallest
full subcategory of $\catC_w$ which is
closed by taking the kernels,
the cokernels and the convolution products and contains
all the cuspidal modules $\dS_k\seteq \dM(w_{\le k} \La_{i_k}, w_{< k}\La_{i_k})$ $(1 \le k\le l)$.
Hence it is enough to show that  every $\dS_k$
has a right dual.

For this,
we shall show that, for any $k$, we can find an epimorphism
$\dS_k\conv X\to \dC_\La$ for some $X\in\catC_w$ and some $\La\in\wtl_+$.
We argue by induction on $l$.

Since there exists a surjective homomorphism
$$
\dS_l \conv \dM( w_{< l} \La_{i_l}, \La_{i_l}  ) \epito \dM(w\La_{i_l}, \La_{i_l})= \dC_{i_l}
$$
by \cite[Proposition 4.6]{KKOP18}, the object $\dS_l$ has a right dual in $\lRg$.

Now assume that $k<l$, and set $w'\seteq w_{< l}$.
Note that   $\dS_k \in \catC_{w'}$.
By the induction hypothesis, there exist  a simple module $X \in  \catC_{w'}$,
$a_i \in \Z_{\ge 0}$ ($i\in I$) and a surjective homomorphism
$$
\dS_k \circ X \epito \conv_{i\in I} \dM(w_{<l} \La_i, \La_i )^{\conv a_i}.
$$
Hence it is enough to show that
$\dM(w_{<l} \La_i, \La_i )$ has a right dual in $\lRg$ for all $i\in I$.

If $i \neq i_l$, then $\dM(w_{<l} \La_i, \La_i )=\dM( w \La_i, \La_i )= \dC_i$.

Suppose that $i = i_l$. By the T-system (\cite[Proposition 10.2.5~(i)]{KKKO18}),
we have
\eqn
\dM(w_{<l} \La_{i_l}, \La_{i_l} ) \conv \dM(w\La_{i_l}, s_{i_l}\La_{i_l})
\epito \conv_{j\not=i_l} { \dM(w \La_j, \La_j)   }^{\conv -a_{j,i_l}}.
\eneqn
Note that the T-system (\cite[Proposition 3.2]{GLS13}, \cite[Proposition 10.2.5~(i)]{KKKO18}) also holds for non-symmetric case if $q$ is replaced by $q_i$.
Hence $\dM(w_{<l} \Lambda_{i_l}, \Lambda_{i_l})$ has a right dual.
\end{proof}

We shall show that any object in $\lRg$ has a left dual.
By Proposition \ref{Prop: canonical braiders}, there exists a localization of $R\gmod$ by the non-degenerate graded braiders $\dC_i$,
which is denoted by $\widetilde{\mathscr{R}}$; i.e.,
$$
\widetilde{\mathscr{R}} \seteq  R\gmod[\dC_i^{\conv -1} \mid i\in I].
$$
Since $\catC_w$ is a full subcategory of $R\gmod$, the canonical embedding $\catC_w \rightarrowtail R\gmod$ induces the fully faithful monoidal functor
 $$
 \imath\cl  \lRg \rightarrowtail \widetilde{\mathscr{R}}
$$
by the construction of localization.
Moreover, the essential image is stable under taking subquotients, extensions and grading shifts by (a graded version of) Proposition~\ref{prop:subloc}.

Set
$
I_w \seteq  \{ i\in I \mid w\La_i \ne \La_i \}$.
Note that $I_w = \{ i_1, \ldots, i_l\}$ for any reduced expression $\underline{w} = s_{i_1}\cdots s_{i_l}$ of $w$.

\Th \label{th: left dual} Assume that $I=I_w$.
\begin{enumerate}[\rm(i)]
\item Every object in $\widetilde{\mathscr{R}}$ has a left dual.
\item  Moreover, the left dual of any object in $\widetilde{\mathscr{R}}$ belongs to $\lRg$.
\end{enumerate}
\enth
\begin{proof}
By Theorem~\ref{th:gendual} and $\eps^*_i(\dC_i)=1$,
$L_\ell(i)$ has a left dual
$(\dC_i^{\circ \ell})^{\circ-1}\conv E_i^*(\dC_i^{\circ \ell})$ in $\widetilde{\mathscr{R}}$
for any $i\in I$ and $\ell \in \Z_{>0}$.
Since $E_i^*(\dC_i^{\circ \ell})$ belongs to $\Cw$,
$\lD(L_\ell(i))$  belongs to $\tCw$.

 It suffices to show that any finite-dimensional $R$-module has a left dual which is contained in $\lRg$.
Let $Q$ be a nonzero finite-dimensional $R$-module.
Since the generators $x_i$ act on any finite-dimensional $R$-module nilpotently,
there exist $R$-modules $P_1$ and $P_2$ which are direct sums of modules of the form $L_{\ell_1}(\nu_1) \conv L_{\ell_2}(\nu_2) \conv \cdots \conv L_{\ell_t}(\nu_t)$
such that
$$
P_1 \To[\phi] P_2 \To Q \To 0
$$
is exact.
Since objects of the form $L_{\ell_1}(\nu_1) \conv L_{\ell_2}(\nu_2) \conv \cdots \conv L_{\ell_t}(\nu_t)$ have left duals in $\tCw$,
$P_1$ and $P_2$ have left duals $\lD(P_1)$ and $\lD(P_2)$ respectively in $\tCw$.
Thus, the cokernel $Q$ also has a left dual in $\tCw$, which is isomorphic to
the kernel of the morphism $\lD(\phi)\cl\lD(P_2)\to\lD(P_1)$.
\end{proof}

\begin{thm} \label{Thm: eqiv R C} 
Assume that $I=I_w$. 
The functor $ \imath\cl  \lRg \To\widetilde{\mathscr{R}}$ is an equivalence of categories.
\end{thm}
\begin{proof}
Let $S$ be a simple object in $ \widetilde{\mathscr{R}} $.
By Theorem~\ref{th: left dual}, the left dual $\lD(S)$ belongs to $\lRg$.
Then $\lD(S)$ does not vanish.
Let us take a simple quotient $f\cl \lD(S)\epito M$ of $\lD(S)$.
Then $M$ has a right dual $\rD(M)$ in $\tCw$
and we have
a morphism $\rD(f)\cl \rD(M)\to S$.
Let us show that  $\rD(f)$ is a monomorphism.
Let $N=\ker(\rD(f))$.
Then we have an exact sequence
$$0\to N\To[\psi] \rD(M)\to S.$$
Taking its left dual, we obtain
$$\lD(S)\epito M\To[{\;\lD(\psi)\;}]\lD(N).$$
Since the composition vanishes, we obtain $\lD(\psi)=0$. Hence $\psi=0$.
Thus we conclude that $\rD(f)$ is a monomorphism.

\smallskip
Since $S$ is simple, $\rD(f)$ is an isomorphism.
Hence $\rD(M)\simeq S$, and hence
$ S \in \lRg$.

Thus, we have proved that any simple  object of $\tR$ belongs to $\tCw$.
Since $\lRg$ is stable by extension, we conclude that $ \lRg \simeq \widetilde{\mathscr{R}}$.
\end{proof}

\Cor
Assume that $I=I_w$.
Then $X\in R\gmod$ belongs to $\Cw$ if and only if
$R_{\dC_i}(X)$ is an isomorphism for any $i\in I$.
\encor
\Proof
Assume that $R_{\dC_i}(X)$ is an isomorphism for any $i\in I$.
Let us show that $X\in\Cw$. 

\smallskip
Since it suffices that any simple subquotient of $X$ belongs to $\Cw$, 
we may assume that $X$ is simple from the beginning (see
Proposition~\ref{prop:Z}~\eqref{item:1}).
There exists $Z\in\Cw$ and  $\La\in\wlP_+$ such that
$X\simeq \dC_{-\La}\conv Z$ in $\tRm$. We may assume that $Z$ is simple.
Then $Z\simeq \dC_\La\conv X$ by Proposition~\ref{prop:str}~\eqref{item:3}. 
Hence we have
$\gW(X)\subset\gW(Z)\subset \rl_+\cap w\rl_-$ 
by \cite[Lemma 2.2]{KKOP18}. 
\QED

Since $-\wt(M) \in \sum_{i\in I_w} \Z_{\ge0} \alpha_i$ for any module $M \in  \catC_{w}$, letting $R'$ be the quiver Hecke algebra associated with the Cartan matrix $\cmA' =( \langle h_i, \alpha_j \rangle )_{i,j\in I_w}$, we regard $\catC_w$ as a subcategory of $R'\gmod$. Then
Theorem~\ref{th: left dual} and Theorem \ref{Thm: eqiv R C} give the following.

\begin{cor} \label{Cor: left rigid}
The category $\lRg$ is left rigid, i.e., every object of $\lRg$ has a left dual in $\lRg$.
\end{cor}

The following conjecture arises naturally from Corollary \ref{Cor: left rigid}.

\begin{conj} \label{Conj: rigid tCw}
The category $\lRg$ is rigid, i.e., every object of $\lRg$ has left and right duals.
\end{conj}

In the case of the \emph{longest} element $w_0 \in \weyl$, Conjecture \ref{Conj: rigid tCw} is true.

\begin{thm}\label{th:fin}
We assume that the quiver Hecke algebra $R$ is of finite type. Let $w_0$ be the longest element of $\weyl$. Then
the category $ \lRg[w_0]$ is rigid.
\end{thm}
\begin{proof} Note that $\catC_{w_0}=R\gmod$ in this case.
Thanks to Corollary \ref{Cor: left rigid}, it suffices to show that every object has a right dual.

For  $\beta \in \rlQ_+$ with $m=\Ht(\beta)$, we have the automorphism of $R(\beta)$ defined by
\begin{align*}
e(\nu_1, \ldots, \nu_m) & \mapsto e(\nu_m, \ldots, \nu_1), \\
x_i e(\nu_1, \ldots, \nu_m) & \mapsto x_{m-i+1} e(\nu_m, \ldots, \nu_1), \\
\tau_j e(\nu_1, \ldots, \nu_m) & \mapsto- \tau_{m-j} e(\nu_m, \ldots, \nu_1),
\end{align*}
which induces a covariant functor $ \mathfrak{a}\cl  R\gmod \rightarrow R\gmod $ such that
$\mathfrak{a} ^2 \simeq \id$ and $ \mathfrak{a}(M\conv N) \simeq \mathfrak{a}(N)\conv \mathfrak{a}(M)$.
For $i\in I$, let $i^*$ be an index such that $\alpha_{i^*} = -w_0 \alpha_i$.
Set $\beta_i=\La_i-w_0\La_i$.
Then module $\dC_i$ is a unique (up to an isomorphism) simple
$R(\beta_i)$-module
satisfying the condition
$$\eps_j^*(\dC_i)=\delta_{j,i}\qt{for any $j\in I$.}$$
Since $\eps_j^*(\ga(\dC_i))=\eps_j(\dC_{i})=\delta_{j,i^*}$,
and $\beta_{i^*}=\beta_i$, we obtain
$$\mathfrak{a}(\dC_i) \simeq \dC_{i^*}\qt{for $i\in I$. }$$

Hence the functor $\mathfrak{a}$ induces
the functor $ \tilde{\mathfrak{a}} \cl   \lRg[w_0] \rightarrow \lRg[w_0]$
which satisfies
$$
\tilde{\mathfrak{a}} ^2 \simeq\id \quad \text{ and } \quad \tilde{\mathfrak{a}}(X\conv Y) \simeq \tilde{\mathfrak{a}}(Y)\conv \tilde{\mathfrak{a}}(X) \quad \text{ for } X,Y\in \lRg[w_0].
$$

It is now straightforward to check that $ \tilde{\mathfrak{a}}( \lD(\tilde{\mathfrak{a}}(X)))$ is a right dual of $X$.
\end{proof}

\Prop \label{prop:periodic}
Assume the conditions and notations in Theorem~\ref{th:fin}.
Then, we have isomorphisms \ro ignoring the grading shifts\rf
$$(\lD)^3(L(i))\simeq L(i^*)\conv \dC_{-\al_i}\qt{for any $i\in I$,}$$ 
and
\eq
(\lD)^6(M)&&\simeq M\conv \dC_{\beta+w_0\beta}
\qt{for any simple $R(\beta)$-module $M$.}\label{eq:period6}
\eneq
\enprop

\Proof
In the course of the proof, we ignore grading shifts.
Since we have an epimorphism
$$\dM(w_0\La_i,s_i\La_i)\conv L(i)\epito \dC_{i},$$
we have
$$\lD(L(i))\simeq \dM(w_0\La_i,s_i\La_i)\circ \dC_{i}^{\circ-1}.$$
Now by T-system (\cite[Proposition 10.2.5]{KKKO18}), we have an epimorphism
$$\dM(w_0s_i\La_i,\La_i)\conv \dM(w_0\La_i, s_i\La_i)
\epito \dM(w_0\la,\la).$$
Here $\la =s_i\La_i+\La_i$.
Hence we have
$$\lD\bl\dM(w_0\La_i, s_i\La_i)\br\simeq\dM(w_0s_i\La_i,\La_i)\conv\dC_{-\la}.$$
Finally, the epimorphism
$$\dM(w_0\La_i,w_0s_i\La_i)\conv\dM(w_0s_i\La_i,\La_i)\to
\dM(w_0\La_i,\La_i)$$
gives
$$\lD\bl\dM(w_0s_i\La_i,\La_i)\br\simeq L(i^*)\conv \dC_{-\La_i}.$$
Indeed, $\dM(w_0\La_i,w_0s_i\La_i)\simeq L(i^*)$ because
$w_0s_i\La_i-w_0\La_i=\al_{i^*}$.

Thus we obtain
\eqn
(\lD)^3(L(i))&&\simeq
(\lD)^2\bl\dM(w_0\La_i,s_i\La_i)\conv\dC_{-\La_i}\br\\
&&\simeq (\lD)\bl\dM(w_0s_i\La_i,\La_i)\conv\dC_{-\la}\br\conv\dC_{-\La_i}\\
&&\simeq L(i^*)\conv\dC_{-\La_i}\conv\dC_\la\conv\dC_{-\La_i}\\
&&\simeq L(i^*)\conv\dC_{-\al_i}.
\eneqn

Then one has
\eqn
(\lD)^6(L(i))
&&\simeq (\lD)^{-3}\bl L(i^*)\conv\dC_{-\al_i}\br\\
&&\simeq L(i) \conv\dC_{-\al_{i^*}}\conv\dC_{\al_i}\\
&&\simeq L(i)\conv\dC_{w_0\al_{i}+\al_i}.
\eneqn
Hence we have obtained \eqref{eq:period6} for $\beta\in\prtl$ such that $\Ht(\beta)=1$.
Let us show \eqref{eq:period6} by induction on $\Ht(\beta)$.
Assume that $\Ht(\beta)>1$.
Let $M$ be a simple $R(\beta)$-module.
Take $i\in I$ such that $\eps_i(M)>0$,
and set $M_0=\tE_i(M)$. Then we have
$M\simeq L(i)\hconv M_0$.
Hence $M$ is the image of a non-zero homomorphism
$f\cl L(i)\conv M_{0}  \to M_{0} \conv L(i)$.
Note that such an $f$ is unique up to a constant multiple.
Applying the exact functor $(\lD)^6$, we conclude that
$(\lD)^6(M)$ is the image of a non-zero morphism
$$(\lD)^6(f)\cl (\lD)^6(M_0)\conv(\lD)^6( L(i))\to(\lD)^6( L(i))\conv (\lD)^{6}(M_0).$$
By the induction hypothesis,
$(\lD)^6(f)$ is isomorphic to
$$\Bigl(M_0\conv L(i)\to L(i)\conv M_0\Bigr)\conv \dC_{\beta+w_0\beta}.$$
Since it is non-zero, the uniqueness property of $f$ implies that
it is equal to $f\conv \dC_{\beta+w_0\beta}$ up to a constant multiple.
Hence its image is isomorphic to
$M\conv\dC_{\beta+w_0\beta}$.
\QED

\begin{conj} \label{Conj: functoriality}
We can choose isomorphism \eqref{eq:period6}
functorially in $M\in \lRg[w_0]$.
\end{conj}

\end{document}